\documentclass[11pt]{amsart}

\usepackage{graphics}
\usepackage{color}
\usepackage[a4paper, margin=3cm]{geometry}

\usepackage{amssymb,enumerate}
\usepackage{amsmath}
\allowdisplaybreaks
\usepackage{bbm}           
\usepackage{bm}
\usepackage{eso-pic,graphicx}
\usepackage{tikz}
\usepackage{cite}
\usepackage{esint}
\usepackage[colorlinks=true, pdfstartview=FitV, linkcolor=blue, citecolor=blue, urlcolor=blue]{hyperref}

\usepackage{graphicx}

\usepackage{amsmath,amsfonts}
\usepackage{rotating}
\usepackage{amssymb}
\usepackage{verbatim}
\usepackage{rotating}
\usepackage{mathrsfs}
\DeclareSymbolFont{largesymbols}{OMX}{cmex}{m}{n}
\makeatletter
\def\Ddots{\mathinner{\mkern1mu\raise\p@
\vbox{\kern7\p@\hbox{.}}\mkern2mu
\raise4\p@\hbox{.}\mkern2mu\raise7\p@\hbox{.}\mkern1mu}}
\makeatother
\newtheorem*{thA}{Theorem A}
\newtheorem*{thB}{Theorem B}
\newtheorem*{thC}{Theorem C}
\newtheorem*{thD}{Theorem D}

\def\XXint#1#2#3{{\setbox0=\hbox{$#1{#2#3}{\int}$}
\vcenter{\hbox{$#2#3$}}\kern-.5\wd0}}

\begin{document}

\newtheorem{definition}{Definition}
\newtheorem{theorem}[definition]{Theorem}
\newtheorem{proposition}[definition]{Proposition}
\newtheorem{conjecture}[definition]{Conjecture}
\def\theconjecture{\unskip}
\newtheorem{corollary}[definition]{Corollary}
\newtheorem{lemma}[definition]{Lemma}
\newtheorem{claim}[definition]{Claim}
\newtheorem{sublemma}[definition]{Sublemma}
\newtheorem{observation}[definition]{Observation}
\theoremstyle{definition}

\newtheorem{notation}[definition]{Notation}
\newtheorem{remark}[definition]{Remark}
\newtheorem{question}[definition]{Question}

\newtheorem{example}[definition]{Example}
\newtheorem{problem}[definition]{Problem}
\newtheorem{exercise}[definition]{Exercise}
 \newtheorem{thm}{Theorem}
 \newtheorem{cor}[thm]{Corollary}
 \newtheorem{lem}{Lemma}[section]
 \newtheorem{prop}[thm]{Proposition}
 \theoremstyle{definition}
 \newtheorem{dfn}[thm]{Definition}
 \theoremstyle{remark}
 \newtheorem{rem}{Remark}
 \newtheorem{ex}{Example}
 \numberwithin{equation}{section}
\def\C{\mathbb{C}}
\def\R{\mathbb{R}}
\def\Rn{{\mathbb{R}^n}}
\def\Rns{{\mathbb{R}^{n+1}}}
\def\Sn{{{S}^{n-1}}}
\def\M{\mathbb{M}}
\def\N{\mathbb{N}}
\def\Q{{\mathbb{Q}}}
\def\Z{\mathbb{Z}}
\def\X{\mathbb{X}}
\def\Y{\mathbb{Y}}
\def\F{\mathcal{F}}
\def\L{\mathcal{L}}
\def\S{\mathcal{S}}
\def\supp{\operatorname{supp}}
\def\essi{\operatornamewithlimits{ess\,inf}}
\def\esss{\operatornamewithlimits{ess\,sup}}

\numberwithin{equation}{section}
\numberwithin{thm}{section}
\numberwithin{definition}{section}
\numberwithin{equation}{section}

\def\earrow{{\mathbf e}}
\def\rarrow{{\mathbf r}}
\def\uarrow{{\mathbf u}}
\def\varrow{{\mathbf V}}
\def\tpar{T_{\rm par}}
\def\apar{A_{\rm par}}

\def\reals{{\mathbb R}}
\def\torus{{\mathbb T}}
\def\t{{\mathcal T}}
\def\heis{{\mathbb H}}
\def\integers{{\mathbb Z}}
\def\z{{\mathbb Z}}
\def\naturals{{\mathbb N}}
\def\complex{{\mathbb C}\/}
\def\distance{\operatorname{distance}\,}
\def\support{\operatorname{support}\,}
\def\dist{\operatorname{dist}\,}
\def\Span{\operatorname{span}\,}
\def\degree{\operatorname{degree}\,}
\def\kernel{\operatorname{kernel}\,}
\def\dim{\operatorname{dim}\,}
\def\codim{\operatorname{codim}}
\def\trace{\operatorname{trace\,}}
\def\Span{\operatorname{span}\,}
\def\dimension{\operatorname{dimension}\,}
\def\codimension{\operatorname{codimension}\,}
\def\nullspace{\scriptk}
\def\kernel{\operatorname{Ker}}
\def\ZZ{ {\mathbb Z} }
\def\p{\partial}
\def\rp{{ ^{-1} }}
\def\Re{\operatorname{Re\,} }
\def\Im{\operatorname{Im\,} }
\def\ov{\overline}
\def\eps{\varepsilon}
\def\lt{L^2}
\def\diver{\operatorname{div}}
\def\curl{\operatorname{curl}}
\def\etta{\eta}
\newcommand{\norm}[1]{ \|  #1 \|}
\def\expect{\mathbb E}
\def\bull{$\bullet$\ }

\def\blue{\color{blue}}
\def\red{\color{red}}

\def\xone{x_1}
\def\xtwo{x_2}
\def\xq{x_2+x_1^2}
\newcommand{\abr}[1]{ \langle  #1 \rangle}

\newcommand{\Norm}[1]{ \left\|  #1 \right\| }
\newcommand{\set}[1]{ \left\{ #1 \right\} }
\newcommand{\ifou}{\raisebox{-1ex}{$\check{}$}}
\def\one{\mathbf 1}
\def\whole{\mathbf V}
\newcommand{\modulo}[2]{[#1]_{#2}}
\def \essinf{\mathop{\rm essinf}}
\def\scriptf{{\mathcal F}}
\def\scriptg{{\mathcal G}}
\def\m{{\mathcal M}}
\def\scriptb{{\mathcal B}}
\def\scriptc{{\mathcal C}}
\def\scriptt{{\mathcal T}}
\def\scripti{{\mathcal I}}
\def\scripte{{\mathcal E}}
\def\V{{\mathcal V}}
\def\scriptw{{\mathcal W}}
\def\scriptu{{\mathcal U}}
\def\scriptS{{\mathcal S}}
\def\scripta{{\mathcal A}}
\def\scriptr{{\mathcal R}}
\def\scripto{{\mathcal O}}
\def\scripth{{\mathcal H}}
\def\scriptd{{\mathcal D}}
\def\scriptl{{\mathcal L}}
\def\scriptn{{\mathcal N}}
\def\scriptp{{\mathcal P}}
\def\scriptk{{\mathcal K}}
\def\frakv{{\mathfrak V}}
\def\C{\mathbb{C}}
\def\D{\mathcal{D}}
\def\R{\mathbb{R}}
\def\Rn{{\mathbb{R}^n}}
\def\rn{{\mathbb{R}^n}}
\def\Rm{{\mathbb{R}^{2n}}}
\def\r2n{{\mathbb{R}^{2n}}}
\def\Sn{{{S}^{n-1}}}
\def\bbM{\mathbb{M}}
\def\N{\mathbb{N}}
\def\Q{{\mathcal{Q}}}
\def\Z{\mathbb{Z}}
\def\F{\mathcal{F}}
\def\L{\mathcal{L}}
\def\G{\mathscr{G}}
\def\ch{\operatorname{ch}}
\def\supp{\operatorname{supp}}
\def\dist{\operatorname{dist}}
\def\essi{\operatornamewithlimits{ess\,inf}}
\def\esss{\operatornamewithlimits{ess\,sup}}
\def\dis{\displaystyle}
\def\dsum{\displaystyle\sum}
\def\dint{\displaystyle\int}
\def\dfrac{\displaystyle\frac}
\def\dsup{\displaystyle\sup}
\def\dlim{\displaystyle\lim}
\def\bom{\Omega}
\def\om{\omega}

\author[Y. Li]{Yuru Li}
\address{Yuru Li:
School of Mathematical Sciences \\
Beijing Normal University \\
Laboratory of Mathematics and Complex Systems \\
Ministry of Education \\
Beijing 100875 \\
People's Republic of China}
\email{yrli@mail.bnu.edu.cn}

\author[J. Tan]{Jiawei Tan}
\address{Jiawei Tan:
School of Mathematics and Computing Science, Guilin University of Electronic Technology \\
Center for Applied Mathematics of Guangxi (GUET) \\
Guangxi Colleges and Universities Key Laboratory of Data Analysis and Computation \\
Guilin 541004 \\
People's Republic of China}
\email{jwtan@guet.edu.cn}

\author[Q. Xue]{Qingying Xue$^{*}$}
\address{Qingying Xue:
	School of Mathematical Sciences \\
	Beijing Normal University \\
	Laboratory of Mathematics and Complex Systems \\
	Ministry of Education \\
	Beijing 100875 \\
	People's Republic of China}
\email{qyxue@bnu.edu.cn}


\thanks{ The second author is partially supported by the Science and Technology Project of Guangxi (Guike AD23023002). The third author is supported by the National Key Research and Development Program of China (No. 2020YFA0712900) and NNSF of China (No. 12271041).
\thanks{$^{*}$ Corresponding author, e-mail address: qyxue@bnu.edu.cn}}

\date{\today}
\title[Two inequalities for the general commutators]
{\bf Two inequalities for commutators of singular integral operators satisfying H\"{o}rmander conditions of Young type}

\begin{abstract}
In this paper, we systematically study  the Fefferman-Stein inequality and Coifman-Fefferman inequality for the general commutators of singular integral operators that satisfy H\"{o}rmander conditions of Young type. Specifically, we first establish the pointwise sparse domination for these operators. Then, relying on the dyadic analysis, the Fefferman-Stein inequality with respect to arbitrary weights and the quantitative weighted Coifman-Fefferman inequality are demonstrated. We decouple the relationship between the number of commutators and the index $\varepsilon$, which essentially improved the results of P\'{e}rez and Rivera-R\'{\i}os (Israel J. Math., 2017). As applications, it is shown that all the aforementioned results can be applied to a wide range of operators, such as singular integral operators satisfying the $L^r$-H\"{o}rmander operators, $\omega$-Calder\'{o}n-Zygmund operators with $\omega$ satisfying a Dini condition, Calder\'{o}n commutators, homogeneous singular integral operators and Fourier multipliers.

\end{abstract}\maketitle

\section{Introduction and main results}
The main purpose of this paper is to investigate the
weak type endpoint estimates with arbitrary weights and the quantitative weighted Coifman-Fefferman inequality for the general commutators of singular integral operators satisfying $L^A$-H\"{o}rmander conditions of non-convolution type.
\subsection{Definitions and notations} 
\
\newline
\indent The singular integral operator $T$ of convolution type with a kernel $K$, defined as
$$Tf(x)=p.v.\int_{\Rn}K(x-y)f(y)dy,$$ and bounded on $L^2$,
is widely recognized as a cornerstone of the classical Calder\'{o}n-Zygmund theory. 
The properties of $T$ are intimately tied to the smoothness of the kernel $K$, with any variation in the smoothness condition directly influencing the behavior of $T$. Consequently, investigating the properties of $T$ under different kernel conditions has become a significant area of research in Harmonic analysis. Below, we will review the development in this field based on the research progress under different kernel conditions.

In 1972, Coifman\cite{hor+coi7} showed that if the kernel $K$ satisfies the Lipschitz condition:
\begin{equation}\label{ie20}
  |K(x-y)-K(x)|\leq C\frac{|y|^{\alpha}}{|x|^{\alpha+n}},\,0<\alpha<1,\,|x|>2|y|,
\end{equation}
then the following Coifman type estimate holds,
\begin{equation}\label{ie5.30}
\int_{\mathbb{R}^n}|Tf(x)|^pw(x)dx\leq C\int_{\mathbb{R}^n}\big(Mf(x)\big)^pw(x)dx,\,\hbox{for any }0<p<\infty, \, w\in A_{\infty}.
\end{equation}
 Here $A_{\infty}$ denotes the classical Muckenhoupt $A_{\infty}$ class of weights, which will be introduced in Section \ref{sec2}.
 
The Lipschitz condition for the Coifman type estimate \eqref{ie5.30} can be weakened to a class of Dini condition \cite{kur}. Recall that the homogeneous kernel $K(x)=\Omega(x/|x|)|x|^{-n}$ is said to satisfy the Dini condition (denoted by ${\mathcal H}_{\rm Dini}$) if
$$\int_{0}^1w_{\infty}(t)\frac{dt}{t}<\infty,$$
where $w_{\infty}(t):=\sup\{|\Omega(\theta_1)-\Omega(\theta_2)|
:\ \theta_1,\theta_2\in \mathbb{S}^{n-1},\,|\theta_1-\theta_2|\leq t\}$.

To introduce more known results, we need to introduce the definition of $L^r$-H\"{o}rmander condition. 
\begin{definition}[\textbf{$L^r$-H\"{o}rmander condition}, \cite{kur}]\label{def50}
A function $K$ is called to satisfy $L^r$-H\"{o}rmander condition if there exist positive constants $c$ and $C$ such that for any $y\in \mathbb{R}^n$ and $R>c|y|$,
\begin{align*}
\sum_{m=1}^{\infty}(2^mR)^{\frac{n}{r'}}\Big(\int_{2^mR<|x|<2^{m+1}R}
      |K(x-y)-K(x)|^r dx\Big)^{\frac{1}{r}}&\leq C<\infty,\quad 1\leq r<\infty;\\
\sum_{m=1}^{\infty}(2^mR)^n\sup_{\{x:\ 2^mR<|x|\leq 2^{m+1}R\}}|K(x-y)-K(x)|&\leq C<\infty,\quad r=\infty.
\end{align*}
\end{definition}
Let ${\mathcal H}_r$ be the class of kernels satisfying the $L^r$-H\"{o}rmander condition, then  the following including relationship holds:
$${\mathcal H}_{\rm Dini}\subset {\mathcal H}_{\infty} \subset {\mathcal H}_r \subset {\mathcal H}_s \subset {\mathcal H}_1,\,\, 1<s<r<\infty.$$
It was Kurtz and Wheeden\cite{kur}, who first studied the weighted norm inequalities of the operator $T$ with $K$ satisfying $L^r$-H\"{o}rmander condition. Surprisingly, Martell, P\'erez and Trujillo \cite{hor+mar} showed that inequality \eqref{ie5.30} no longer holds if $K\in {\mathcal H}_r$ for $1\leq r<\infty$.

In 2005, Lorente et al. \cite{hor+lor1} introduced the following $L^A$-H\"{o}rmander condition by replacing the $L^r$ norm in Definition \ref{def50} with a more general Orlicz norm. Let $A$ be a Young function (see Section \ref{section2.5}). $K$ is said to satisfy the $L^A$-H\"{o}rmander condition if there exsit $c\geq1$ and $C>0$ such that for any $y\in \mathbb{R}^n$ and $R>c|y|$,
\begin{equation}\label{ie100}
\sum_{m=1}^{\infty}(2^mR)^n\left\|(K(\cdot-y)-K(\cdot))\chi_{2^mR<|x|\leq 2^{m+1}R}(\cdot)\right\|_{A,B(0,2^{m+1}R)}\leq C.
\end{equation}

When the kernel satisfies the $L^A$-H\"{o}rmander condition, the Coifman type estimates, the weighted weak type endpoint estimates and two-weight estimates were established in \cite{hor+lor3},\cite{hor+lor1} and \cite{hor+lor2}. More related works about $L^r$-H\"{o}rmander condition can be found in the references\cite{hor+rub,wat}. It is worth noting that the $L^r$-H\"{o}rmander condition also plays a crucial role in studying rough singular integral operators in the works of Watson\cite{wat}.

Consider now the singular integral operator $T$ of non-convolution type. It is known that many important operators in Harmonic analysis are of non-convolution type, such as Calder\'{o}n communators, non-standard singular integral operators, Coifman-Meyer type operators. Let $T$ be a bounded linear operator on $L^2$, and for a bounded function $f$ with compact support, we always assume that $T$ can be represented in the form that 
\begin{equation}\label{ie5.38}
T(f)(x)=\int_{\mathbb{R}^n}K(x,y)f(y)dy,\quad x\notin {\rm supp}~f,
\end{equation}
where the kernel $K:\mathbb{R}^n\times\mathbb{R}^n \backslash \{(x,x):x\in \mathbb{R}^n\}\rightarrow \mathbb{R}$ is a locally integrable function.
In 2008, Guo et al.\cite{hor+guo} introduced the following non-convolution type $L^r$-H\"{o}rmander kernel condition in studying of the Riesz transform associated with the Schr\"{o}dinger operators.
\begin{definition}[\cite{hor+guo}]
Let $K$ be a locally integrable function, we say that $K$ satisfies the $L^r$-H\"{o}rmander condition:\\
{\rm(1)} when $1\leq r<\infty$,
\begin{align*}
\sup_{Q}\sup_{x,z\in \frac{1}{2}Q}\sum_{k=1}^{\infty}|2^kQ|\Big(\frac{1}{|2^kQ|}\int_{2^kQ\backslash 2^{k-1}Q}|K(x,y)-K(z,y)|^rdy\Big)^{1/r}&<\infty,\\
\sup_{Q}\sup_{x,z\in \frac{1}{2}Q}\sum_{k=1}^{\infty}|2^kQ|\Big(\frac{1}{|2^kQ|}\int_{2^kQ\backslash 2^{k-1}Q}|K(y,x)-K(y,z)|^rdy\Big)^{1/r}&<\infty;
\end{align*}
{\rm(2)} when $r=\infty$,
\begin{align*}
\sup_{Q}\sup_{x,z\in \frac{1}{2}Q}\sum_{k=1}^{\infty}|2^kQ|\mathop{\mathrm{ess\,sup}}\limits_{y \in 2^kQ\backslash 2^{k-1}Q}|K(x,y)-K(z,y)|&<\infty,\\
\sup_{Q}\sup_{x,z\in \frac{1}{2}Q}\sum_{k=1}^{\infty}|2^kQ|\mathop{\mathrm{ess\,sup}}\limits_{y \in 2^kQ\backslash 2^{k-1}Q}|K(y,x)-K(y,z)|&<\infty.
\end{align*}
\end{definition}

In multilinear setting, it was Li\cite{li3} who showed that the operators with the multilinear $L^r$-H\"{o}rmander condition can be dominated pointwisely by a finite number of sparse operators and satisfy the quantitative weighted estimates. Recently, Cao and Yabuta \cite{cao1} obtained the quantitative weighted estimates for Littlewood-Paley square functions satisfying the multilinear $L^r$-H\"{o}rmander condition.

This paper focuses on the following non-convolution type $L^A$-H\"{o}rmander condition, which is an essential extension of the $L^A$-H\"{o}rmander condition \eqref{ie100}.
\begin{definition}[\textbf{$L^A$-H\"{o}rmander condition of non-convolution type}]\label{def5.2}
Given a Young function $A$, define
\begin{align*}
   & H_{K,A,1}=\sup_{Q}\sup_{x,z\in \frac{1}{2} Q}\sum_{k=1}^{\infty}|2^kQ|\left\|(K(x,\cdot)-K(z,\cdot))\chi_{2^k Q \backslash 2^{k-1} Q}\right\|_{A,2^k Q},  \\
   & H_{K,A,2}=\sup _Q \sup _{x, z \in \frac{1}{2} Q} \sum_{k=1}^{\infty}|2^kQ|\left\|(K(\cdot, x)-K(\cdot, z)) \chi_{2^k Q \backslash 2^{k-1} Q}\right\|_{A, 2^k Q},
\end{align*}
if $H_{K,A}=\max\left\{H_{K,A,1},H_{K,A,2}\right\}<\infty$, then we call $K$ satisfies $L^A$-H\"{o}rmander condition. 
\end{definition}
Let ${\mathcal H}_A$ be the class of kernels satisfying the $L^A$-H\"{o}rmander condition $($without ambiguity about the kernel function $K$, shorten $H_{K,A}$ to $H_{A}$$)$.
An operator $T$ for which the kernel $K\in {\mathcal H}_{A}$ satisfies \eqref{ie5.38} and the size condition $K(x,y)\leq C_K|x-y|^{-n},\,x\neq y$, is called an
$A$-H\"{o}rmander operator.

\subsection{The general commutators}
\
\newline
\indent The study of commutators of  Calder\'{o}n-Zygmund operators can be traced back to the celebrated work given by Coifman, Rochberg and Weiss\cite{coi6}, in which they defined the commutator $[b,T]$ by
$$[b,T]f(x)=bT(f)(x)-T(bf)(x),$$
where $T$ is the standard Calder\'{o}n-Zygmund operator and $b$ is a locally integrable function.
Although the initial interest in studying such operators was related to the generalization of the classical factorization theorem of Hardy spaces, many applications in other areas have also been discovered, especially in partial differential equations, as referenced in \cite{hor+faz,hor+gre}. It is worth noting that the weak $(1,1)$ estimate for $[b,T]$ does not hold. However, it still enjoys certain weak $L\log L$ type estimate (see \cite{per0}).

In \cite{hor+per12}, P\'{e}rez and Pradolini introduced the high order commutator of $T$ defined by
$$T^m_bf(x)=\big[b,T^{m-1}_b\big]f(x),\,T^1_{b}f(x)=[b,T]f(x),$$
and they established the weak type endpoint estimates for $T^m_{b}$ with arbitrary weights.

\begin{thA}[\cite{hor+per12}]\label{thmL}
Let $T$ be a Calder\'{o}n-Zygmund operator and $m\in \mathbb{N}$. Assume that $b\in \mathrm {BMO}$, then for every weight $w$ and $0<\varepsilon<1$,
\begin{equation}\label{ie5.32}
w\big(\{x\in \mathbb{R}^n:|T^m_{ b}f(x)|>\lambda\}\big)\leq C_{\varepsilon,T} \int_{\mathbb{R}^n}\Phi_m
\Big(\frac{\|b\|^m_{\mathrm{BMO}}|f(x)|}{\lambda}\Big)M_{L(\log L)^{m+\varepsilon}}w(x)dx
\end{equation}
holds for all $\lambda>0$, where $\Phi_m(t)=t\log^m(e+t)$.
\end{thA}
In 2002, P\'{e}rez and Trujillo-Gonz\'{a}lez\cite{hor+per11} first defined the general commutator as follows
\begin{equation}\label{ie5.31}
T_{\vec b}f(x)=\int_{\mathbb{R}^n}\prod_{j=1}^m(b_j(x)-b_j(y))K(x,y)f(y)dy,
\end{equation}
where $\vec{b}=(b_1,\ldots,b_m)$ is a finite family of locally integrable functions. It is worth mentioning that weighted strong as well as weak type endpoint estimates for $T_{\vec b}$ were given in \cite{hor+per11}. 
\begin{remark}
Note that when $m=1$, $T_{\vec{b}}$ coincides with the classical linear commutator $[b,T]$; If $b_i=b,\,i=1,\ldots,m$, then $T_{\vec{b}}$ exactly is $m$ order iterated commutators $T^m_{b}$; Assume that $\vec{\alpha}=(\alpha_1,\ldots,\alpha_m)\in \mathbb{N}^m$, if we take $b_1=\cdots=b_{\alpha_1},\ldots,b_m=\cdots=b_{\alpha_m}$, in this case, $$T_{\vec{b}}f(x)=T_{\vec{b},\vec{\alpha}}f(x):=\int_{\mathbb{R}^n}\prod_{j=1}^m
(b_j(x)-b_j(y))^{\alpha_j}f(y)dy.$$
Recently, using extrapolation techniques, the weighted boundedness results for $T_{\vec{b},\vec{\alpha}}$ have been obtained in \cite{beny}.
\end{remark}
\subsection{Historical background}
\
\newline
\indent In the weighted theory, the Fefferman-Stein inequality and Coifman-Fefferman inequality have attracted the attention of many researchers. Now we briefly review the background of these two inequalities.
 
{\bf $\bullet$ Fefferman-Stein inequality.} The origin of the Fefferman-Stein inequality can be traced back to the celebrated work of Fefferman and Stein\cite{fef1}. They obtained the following endpoint estimate for the classical Hardy-Littlewood maximal operator $M$,
\begin{equation}\label{ie5.33}
w\big(\{x\in \mathbb{R}^n:|Mf(x)|>\lambda\}\big)\leq C_n \int_{\mathbb{R}^n}\frac{|f(x)|}{\lambda}Mw(x)dx,
\end{equation}
where $w$ is an arbitrary weight and the constant $C_n$ only depends on the dimension $n$. Later on, Muckenhoupt and Wheeden conjectured that a similar estimate holds for any Calder\'{o}n-Zygmund singular integral operator $T$, namely, there exists a constant $C$ such that for any weight $w$ and $\lambda>0$,
\begin{equation}\label{ie5.34}
w\big(\{x\in \mathbb{R}^n:|Tf(x)|>\lambda\}\big)\leq C \int_{\mathbb{R}^n}\frac{|f(x)|}{\lambda}Mw(x)dx,
\end{equation}
which is known as the {\bf Muckenhoupt-Wheeden conjecture}.

Recall that $w$ satisfies Muckenhoupt $A_1$ condition (or belongs to the class $A_1$) if there exists $c>0$ such that $Mw(x)\leq [w]_{A_1}w(x),\,\text{a.e.} \ x \in \mathbb{R}^n$, then \eqref{ie5.34} implies that  for any $w\in A_1$,
\begin{equation}\label{ie5.35}
w\big(\{x\in \mathbb{R}^n:|Tf(x)|>\lambda\}\big)\leq C[w]_{A_1}\int_{\mathbb{R}^n}\frac{|f(x)|}{\lambda}w(x)dx,
\end{equation}
which is referred to as the {\bf weak conjecture of Muckenhoupt and Wheeden}. 

In recent decades, Numerous significant investigations have been made in these field from different perspective.
It was Chanillo and Wheeden\cite{cha1} who verified that the inequality \eqref{ie5.34} holds when $T$ is a square function; Later on, \eqref{ie5.34} was shown to be true by Buckley\cite{buc} if $w_{\delta}(x)=|x|^{-n(1-\delta)},\,0<\delta<1$; However, in 2011, Reguera\cite{hor+reg1} demonstrated that the 
dyadic version of the Muckenhoupt-Wheeden conjecture is not valid for the Haar multiplier. Building upon this, Reguera and Thiele \cite{reg} proved that \eqref{ie5.34} does not hold for the Hilbert transform, which disproves the Muckenhoupt-Wheeden conjecture. It was not until 2015 that Nazarov et al.\cite{hor+naz1} constructed a “smooth” weight function for the Hilbert transform and proved that \eqref{ie5.35} is not valid, consequently refuting the weaker version of the Muckenhoupt-Wheeden conjecture. Further studies on the Muckenhoupt-Wheeden conjecture can be found in  \cite{hor+cri,hor+ler19} and the references therein.

Based on the above extensive studies of Muckenhoupt-Wheeden conjecture, people began to pay attention to the deformation of Fefferman-Stein inequality and its dependence on the weight constant, as follows:
\begin{enumerate}
\item[{$\bullet$}] If $M$ in \eqref{ie5.34} is replaced by a slightly larger operator, does \eqref{ie5.34} hold?

\item[{$\bullet$}]  What is the dependence between the constant in \eqref{ie5.35} and the $A_1$ weight constant?
\end{enumerate}
An affirmative answer to the first questionwas given by P\'{e}rez. In \cite{per7}, he proved that if the Hardy-Littlewood maximal operator $M$ on right-hand side is replaced by
$M^2$ ($M^2=M\circ M$) or even by the operator
$M_{L(\log L)^{\varepsilon}}$ with any $\varepsilon>0$, then \eqref{ie5.34} remains valid, that is, 
$$w\big(\{x\in \mathbb{R}^n:|Tf(x)|>\lambda\}\big)\leq C \int_{\mathbb{R}^n}\frac{|f(x)|}{\lambda}M^2w(x)dx.$$
For the second question,
Lerner, Ombrosi and P\'{e}rez\cite{hor+ler} based on a particular case of the Coifman type estimate, proved that for any $w\in A_1$, 
$$w\big(\{x\in \mathbb{R}^n:|Tf(x)|>\lambda\}\big)\leq C[w]_{A_1}\log(e+[w]_{A_1})\int_{\mathbb{R}^n}\frac{|f(x)|}{\lambda}w(x)dx. $$
It was shown by Lerner et al.\cite{hor+ler17} that the logarithm type dependence on the constants of the weight is necessary. 

As mentioned above, the commutators exhibit more singularity than the operators themselves, so the Fefferman-Stein inequality with respect to commutators naturally become a subject worthy of intensive study. After P\'{e}rez and Pradolini \cite{hor+per12} obtained \eqref{ie5.32} in 2001, P\'{e}rez and Rivera-R\'{\i}os \cite{per9} further considered the weak type endpoint estimates of multilinear commutators $T_{\vec b}$, the results are as follows:
\begin{thB}[\cite{per9}]\label{thmM}
Let $T$ be a Calder\'{o}n-Zygmund operator and $m\in \mathbb{N}$. Assume that $\vec{b}=(b_1,\ldots,b_m)\in \mathrm{BMO}^m$, then for every weight $w$ and $0<\varepsilon<1$,
\begin{equation}\label{ie5.36}
w\big(\{x\in \mathbb{R}^n:|T_{\vec b}f(x)|>\lambda\}\big)\lesssim \frac{1}{\varepsilon^{m+1}}\int_{\mathbb{R}^n}\Phi_m
\Big(\frac{\|\vec{b}\|_{\mathrm{BMO}}|f(x)|}{\lambda}\Big)M_{L(\log L)^{m+\varepsilon}}w(x)dx,
\end{equation}
where $\Phi_m(t)=t\log^m(e+t)$ and $\big\|\vec{b}\big\|_{\mathrm{BMO}}=\prod_{i=1}^m\|b_i\|_{\mathrm{BMO}}$.
\end{thB}

It is observed that the right-hand side of \eqref{ie5.36} has an exact estimate with a constant $\varepsilon$. In particular, when the weight function $w$ belongs to $A_{\infty}$ or $A_{1}$,  utilizing the reverse H\"{o}lder inequality immediately yields a quantitative
weighted estimates for the commutators, which coincides with \cite[Theorem 2.9]{hor+ort1}. By examining \eqref{ie5.36}, we can also note that the exponent of $\varepsilon$ depends on the iteration number $m$ of the commutators and when $m=1$, the constant on the right-hand side is $\frac{1}{\varepsilon^2}$.
In 2017, Lerner, Ombrosi and Rivera-R\'{\i}os\cite{ler5} proved that the result still holds for $\frac{1}{\varepsilon}$ when $m=1$,  which improved \eqref{ie5.36}.
\vspace{0.1cm}

{\bf $\bullet$ Coifman-Fefferman inequality.} The inequality of the form \eqref{ie5.30} is called the Coifman-Fefferman inequality. The estimate can be viewed as a certain control of the Hardy-Littlewood maximal function $M$ with respect to the operator $T$, which implies that $M$ can be used to characterize the weighted boundedness of $T$ in some extent. One of the most usual techniques for proving such results is to establish a good-$\lambda$ inequality between $T$ and $M$, which leads to an unavailability of the constant $C$ with respect to the dependence of the weight $w$ or the parameter $p$ in the detailing. To overcome this shortcoming, many experts have considered the alternative approaches to prove \eqref{ie5.30}. For instance, the proof presented in \cite{hor+alv} uses a pointwise estimate involving the sharp maximal function; The proof in \cite{cur} utilizes an extrapolation techniques avoiding
good-$\lambda$ inequality that enables the derivation of estimates like \eqref{ie5.30} for any $A_1$ weights.

\vspace{0.2cm}

Due to the fact that the Coifman-Fefferman inequality can be utilized in the proofs of numerous other weighted norm inequalities, it plays a crucial role in the weighted theory. 
Firstly, as mentioned above, \eqref{ie5.30} can be used to characterize the weighted boundedness of $T$, i.e., if $T$ satisfies \eqref{ie5.30}, then $T$ is bounded on $L^p(w)$ for any weight $w\in A_p, 1<p<\infty$. It is not widely known that the inequality \eqref{ie5.30} has been proven to be the crucial estimate in establishing the following non-standard two-weight result for $T$\cite{per7}, namely, for any weight $w$, it holds that
\begin{equation}\label{ie1.2}
\int_{\Rn}|Tf(x)|^pw(x)dx\leq c\int_{\Rn}|f(x)|^pM^{\lfloor p\rfloor+1}w(x)dx,
\end{equation}
where $\lfloor p\rfloor$ denotes the integer part of $p$ and $M^{\lfloor p\rfloor+1}$ means the $(\lfloor p\rfloor+1)$-fold composition of $M$. In fact, \eqref{ie1.2} is actually a strong type Fefferman-Stein inequality, and this discovery also illustrates the connection between two inequalities. Another essential consequence of inequality \eqref{ie5.30} is the crucial role it plays in the solution of Sawyer's conjecture, see \cite{hor+cru} for more details. In addition, \eqref{ie5.30} is closely related to the $C_p$ theory, which has been explored in \cite{hor+muc} and \cite{hor+saw}.
\vspace{0.1cm}

In recent years, how to specify the dependence of the implicit constants on the weight $w$ of the Coifman-Fefferman inequality has become an interesting topic. In this direction, Lerner et al.\cite{hor+ler} obtained the explicit dependence of the constant with respect $w\in A_p, 1\leq p<\infty$. Later on, 
the quantitative Coifman-Fefferman inequalities were proved for a variety of singular operators, 
including vector valued extensions and multilinear 
Calder\'{o}n-Zygmund operators in \cite{hor+ort1}.
In 2019, the quantitative Coifman-Fefferman inequality about rough singular integral operator $T_{\Omega}$ with $\Omega\in L^{\infty}(\mathbb{S}^{n-1})$ was established in \cite{hor+li}. Since no good-$\lambda$ estimate relating $T_{\Omega}$ and $M$ is known to hold, instead, the proof of this result is a consequence of the sparse domination for $T_{\Omega}$ together with the technique of principal cubes. Very recently, Lerner\cite{hor+ler10} gives a sufficient condition to describe the Coifman-Fefferman inequality of the non-degenerate singular integral operator.

\vspace{0.1cm}

The Coifman-Fefferman inequality of the commutator of Calder\'{o}n-Zygmund operator with a $\rm BMO$ function was given in \cite{hor+per10}: 
$$\|[b,T]f\|_{L^p(w)}\leq c\|b\|_{\rm {BMO}}\|M^2f\|_{L^p(w)},\,\hbox{for any }w\in A_{\infty}.$$
In 2012, Ortiz-Caraballo et al.\cite{hor+ort1} derived the following quantitative weighted inequality, which was achieved via establishing a pointwise domination of commutators. 
\begin{thC}[\cite{hor+ort1}]
Let $T$ be a Calder\'{o}n-Zygmund operator, $b\in{\rm BMO}$ and $w\in A_{\infty}$. For the $m$-th $(m\geq 2)$ commutator $T^{m}_b$, there exists a constant $c_{n,T}$ such that
$$\|T^m_{b}f\|_{L^p(w)}\leq c_{n,T}2^m\max\{1,p^{m+1}\}[w]^{m+1}_{A_{\infty}}
\|b\|^m_{\rm BMO}\|M^{m+1}f\|_{L^p(w)},\, 0<p<\infty.$$
\end{thC}
 \vspace{0.1cm}

\subsection{Motivation}
\
\newline
\indent 
It is universally acknowledged that sparse domination, owing to its localization and sparseness, has become a crucial and effective tool in quantitative weighted Fefferman-Stein inequality and Coifman-Fefferman inequality over the past few decades. One of its remarkable improvement was Lerner's demonstration \cite{hor+ler15,hor+ler14} that the Calder\'{o}n-Zygmund operator can be dominated by the sparse operators $\mathcal{A}_{\mathcal{S}}$ as follows: 
$$\mathcal{A}_{\mathcal{S}}f(x)=\sum_{Q\in \mathcal{S}}\frac{1}{|Q|}\int_Q|f(x)|dx\chi_{Q}(x),$$
where each $Q$ is a cube with its sides parallel to the axis and $\mathcal{S}$ is a sparse family.
This crucial fact facilitates a streamlined and alternative approach in proving the $A_2$ theorem which was initially established by Hyt\"{o}nen \cite{hor+hyt}. Subsequently, Lerner and Nazarov\cite{hor+ler18}, Conde-Alonso and Rey\cite{hor+con} independently demonstrated that the Calder\'{o}n-Zygmund operator $T$ can be dominated pointwisely by a finite number of sparse operators in the way that
$$|Tf(x)|\leq c_{n,T}\sum_{j=1}^{3^n}\mathcal{A}_{\mathcal{S}_j}f(x).$$
The ideas of sparse domination have been applied extensively to a wide range of extensions, such as commutators\cite{ler5}, rough singular integrals\cite{hor+con1}, Bochner-Riesz multipliers\cite{hor+bene} and so on. We refer the interested reader to \cite{hor+ber,cao1,li3,hor+tan} and the references therein.

As mentioned above, Lerner, Ombrosi and Rivera-R\'{\i}os\cite{ler5} obtained the Fefferman-Stein inequalities by establishing the pointwise sparse domination of $[b,T]$, as stated in the following theorem.
\begin{thD}[\cite{ler5}]\label{0.thmN}
Let $T$ be an $\omega$-Calder\'{o}n-Zygmund operator $($for the definition, see Section \ref{sec2}$)$ and $\omega$ satisfies Dini condition: $\int_{0}^1\omega(t)\frac{dt}{t}<\infty$. If $b\in {\mathrm{BMO}}$, then for every weight $w$ and $0<\varepsilon<1$,
\begin{equation*}
w\big(\{x\in \mathbb{R}^n:\big|[b,T]f(x)\big|>\lambda\}\big)\lesssim \frac{1}{\varepsilon}\int_{\mathbb{R}^n}\Phi
\Big(\frac{\|b\|_{\mathrm{BMO}}|f(x)|}{\lambda}\Big)M_{L(\log L)^{1+\varepsilon}}w(x)dx.
\end{equation*}
In particular, when $w\in A_{1}$,
$$w\big(\{x\in \mathbb{R}^n:\big|[b,T]f(x)\big|>\lambda\}\big)\lesssim
[w]_{A_1}\Phi([w]_{A_{\infty}})\int_{\mathbb {R}^n}\Phi\Big(\frac{\|b\|_{\mathrm{BMO}}|f(x)|}{\lambda}\Big)w(x)dx,$$
where $\Phi(t)=t\log(e+t)$.
\end{thD}

Our motivation orginates from the following three aspects:
\vspace{0.2cm}

\textbf{(1).}  The study of $A$-H\"{o}rmander operators and their commutators has drawn significant attention due to their deep connections with weighted inequalities. In 2008, Lorente et al. \cite{hor+lor2} made substantial progress by establishing the Coifman-Fefferman inequality for commutators associated with $A_{\infty}$ weights. Subsequently, they further extended these results \cite{hor+lor3} by proving the Fefferman-Stein inequality under the framework of maximal functions dependent on a Young function  $A$. More recently, Ibañez-Firnkorn and Rivera-Ríos \cite{iba} employed the sparse domination technique to obtain  various quantitative weighted results for $A$-H\"{o}rmander operator $T$ and the $m$ order iterated commutators $T^m_{b}$, including strong weighted estimates, Coifman-Fefferman inequalities, weak-type endpoint estimates, and local exponential decay estimates. As a notable application, they generalized prior results from \cite{ler5} by proving Fefferman-Stein inequalities for $T^{m}_b$, where  $T$ is assumed to be an $\omega$-Calder\'{o}n-Zygmund operator. These developments have significantly expanded our understanding of the interplay between H\"{o}rmander-type conditions, Young functions, and the underlying properties of operators and their kernels. 

While these advances are primarily focused on the commutators  $T^{m}_b$, it remains an open problem whether these results can be extended to more general framework. In particular, what kinds of sparse domination and weak type endpoint estimates with arbitrary weights does $T_{\vec b}$ satisfy forms our first motivation.
\vspace{0.2cm}

\textbf{(2).} A further challenge arises in understanding the precise dependence of the exponent $\varepsilon$ appearing in Theorem B, particularly its relationship with the parameters $m$ and the number of functions $b_i$ involved in the commutators $T_{\vec{b}}$. 
Theorem B indicates that  $\varepsilon$ is inherently tied to $m$ and Theorem D holds only in the case $m=1$, which limits the flexibility and generality of the estimates, especially when dealing with higher-order commutators or varying numbers of functions $b_i$. This dependence becomes particularly problematic when considering higher-order commutators or scenarios with varying numbers of functions $b$, as it obscures the individual influence of these parameters on the behavior of $\varepsilon$.

Our second motivation is to decouple the dependence of  $\varepsilon$ on $m$. Such a decoupling would not only provide deeper insight into the behavior of  $\varepsilon$ but also allow for sharper and more refined estimates that better capture the individual contributions of $m$ and the number of $b_i$'s. 
\vspace{0.2cm}

\textbf{(3).}  Our third motivation is to improve the precision of the constants and to extend the range of indices in the weighted Coifman-Fefferman inequalities for $T_b^m$ or  $T_{\vec{b}}$. By refining the dependence of the constants on the parameters of the operators and the weights, as well as expanding the scope of the inequalities to a broader class of indices, we aim to provide a more comprehensive and flexible framework for understanding these inequalities. This will help to understand the intricate relationship between the structure of $T_{\vec{b}}$, the weights, and the associated norm inequalities.

\subsection{Main results}\label{sec1.5}
\
\newline
\indent In this paper, we first establish a pointwise sparse domination principle for the general commutators of $A$-H\"{o}rmander operators, providing a useful tool for their quantitative analysis. As a key application of this framework, we derive the Fefferman-Stein inequality in the context of arbitrary weights. Furthermore, we develop a quantitative version of the weighted Coifman-Fefferman inequality, offering sharper constants and a broader range of applicability. These results extend and unify many previous works, which contribute to a more comprehensive and profound understanding of the theory of commutators in the weighted setting.

 Before stating our results, we need to give a  definition as follows.
\begin{definition}\label{def5.3}
Let $1\leq p_0\leq p_1<\infty$. Define the class of functions $\mathcal{Y}(p_0,p_1)$ as the class of functions $A$ for which there exsit constants
$c_{A,p_0},~c_{A,p_1}$ and $t_A\geq 1$ such that
\begin{equation*}
t^{p_0}\leq c_{A,p_0}A(t),~t>t_A;\quad t^{p_1}\leq c_{A,p_1}A(t),~t\leq t_A.
\end{equation*}
\end{definition}
We now present several examples of Young functions that belong to $\mathcal{Y}(p_0,p_1)$:
\begin{enumerate}
  \item If $A(t)=t^r,\,1<r<\infty$, then $A\in \mathcal{Y}(r,r)$;
  \item Let $A(t)=t\log^{\alpha}(e+t),\,\alpha>0$, we have $A\in \mathcal{Y}(1,1)$;
  \item When $A(t)=t\log^{\alpha}(e+t)\log^{\beta}(e+\log(e+t)),\,\alpha,\beta>0$, then $A\in \mathcal{Y}(1,1)$.
\end{enumerate}
The above examples essentially cover several types of Young functions that we want to investigate. Therefore, the requirement that $A \in \mathcal{Y}(p_0,p_1)$ does not constitute a substantial constraint on our conclusions.

Now we are ready to state our main results. The first one is the following pointwise sparse domination for the general commutators of $A$-H\"{o}rmander operators.
\begin{theorem}\label{thm5.1}
Let $m\in\mathbb{N}$ and $A$ be a Young function. Assume that $T$ is an $\bar{A}$-H\"{o}rmander operators. If $A\in \mathcal{Y}(p_0,p_1)~(1\leq p_0\leq p_1<\infty)$, then for any locally integrable functions $\vec{b}=(b_1,\ldots,b_m)$ on $\mathbb{R}^n$ and bounded function $f$ with compact support, there exist $3^n$ sparse families $\{\mathcal{S}_j\}_{j=1}^{3^n}$ such that
\begin{equation}\label{ie0.1}
|T_{\vec{b}}f(x)|\leq C_{n,m}C_{T}\sum_{j=1}^{3^n}\sum_{\vec{\gamma}\in\{0,1\}^m}\mathcal{A}^{\vec\gamma}_{A,\mathcal{S}_j}(\vec{b},f)(x),\quad \text{a.e.} \ x \in \mathbb{R}^n,
\end{equation}
where $C_T=c_{n,p_0,p_1}\max\{c_{A,p_0},c_{A,p_1}\}\big(H_{\bar{A}}+\|T\|_{L^2\rightarrow L^2}\big)$ and
\begin{equation*}
\mathcal{A}^{\vec\gamma}_{A,\mathcal{S}_j}(\vec{b},f)(x)=
    		\begin{cases}
    			\sum\limits_{Q\in \mathcal{S}_j}\big\|\prod_{i=1}^m(b_i-\langle b_i\rangle_Q)f\big\|_{A,Q}\chi_{Q}(x),  &\vec{\gamma}=\vec 0,\\
    			\sum\limits_{Q\in \mathcal{S}_j}\prod\limits_{s:\gamma_s=1}\big|b_{s}(x)-\langle b_{s}\rangle_Q\big|\Big\|\prod\limits_{s:\gamma_s=0}(b_{s}-\langle b_{s}\rangle_Q)f\Big\|_{A,Q}\chi_{Q}(x), &\vec{\gamma}\neq\vec 0.
    		\end{cases}
\end{equation*}
\end{theorem}

For the endpoint case, using Theorem \ref{thm5.1}, we establish the weak type Fefferman-Stein inequality with arbitrary weights of the general commutator $T_{\vec{b}}$.
\begin{theorem}\label{thm5.2}
Let $m\in\mathbb{N}$ and $\vec{b}=(b_1,\ldots,b_m)\in \mathrm{BMO}^m$ with $\|b_i\|_{\mathrm{BMO}}=1,i=1,\ldots,m$. Assume that $A_0, A_1,\ldots,A_m$ are Young functions satisfying $A_{0}\in \mathcal{Y}(p_0,p_1)~(1\leq p_0\leq p_1<\infty)$ and $A^{-1}_{i}(t)\bar{A}^{-1}_{i-1}(t)C^{-1}(t)\leq ct, t\geq 1, i=1,\ldots,m$, where $C(t)=e^{t}-1$.
If $T$ is an $\bar{A}_0$-H\"{o}rmander operator and $A_i$ is submultiplicative, i.e., $A_i(xy)\leq A_i(x)A_i(y)$, then for any weight $w$ and every family of Young functions $\varphi_{\vec{l}}:=\varphi_{(l_1,l_2)},\,0\leq l_2 < l_1\leq m$, it holds that 
\begin{align*}
    w\big(\{x\in \mathbb{R}^n:|T_{\vec b}f(x)|>\lambda\}\big)&\lesssim K_{\varphi_{(m,m)}}\int_{\mathbb{R}^n}A_m\Big(\frac{|f(x)|}{\lambda}\Big)M_{\varphi_{(m,m)}}w(x)dx\\
    &+\sum_{l_1=1}^m\sum_{l_2=0}^{l_1-1}K_{\varphi_{(l_1,l_2)}}
    \int_{\mathbb{R}^n}A_{m-l_1}\Big(\frac{|f(x)|}{\lambda}\Big)M_{\Phi_{l_1-l_2}\circ\varphi_{(l_1,l_2)}}w(x)dx,
\end{align*}
with
\begin{equation}\label{ie8.3}
K_{\varphi_{(l_1,l_2)}}=
    		\begin{cases}
c_n+\int_{1}^{\infty}\frac{\varphi^{-1}_{(m,m)}(t)A_m(\log^2(e+t))}{t^2\log^3(e+t)}dt,  &l_1=l_2=m,\\		
    c_n+\beta_{m,l_1,l_2}+\int_{1}^{\infty}\frac{\varphi^{-1}_{(l_1,l_2)}\circ\Phi^{-1}_{l_1-l_2}(t)
    A_{m-l_1}(\log^{2l_1+1}(e+t))}{t^2\log^{l_1+2}(e+t)}dt,	 &\text{others},
    \end{cases}
    \end{equation}
where $\beta_{m,l_1,l_2}$ is an absolute constant and $\Phi_{i}(t)=t\log^i(e+t)$.

\end{theorem}

\begin{remark}
We now make some comments on Theorem \ref{thm5.2}. First of all, let $b_i=b,\,i=1,\ldots,m$, then the results of Theorem \ref{thm5.2} cover the conclusions in \cite[Theorem 2.7]{iba}. 
Secondly, if $m=0$, that is, $T_{\vec b}=T$, \cite[Theorem 3.1]{hor+lor3} proved that the $\bar{A}_0$-H\"{o}rmander operator $T$ satisfied the following Fefferman-Stein inequality under some certain conditions:
\begin{equation}\label{ie5.37}
w\big(\{x\in \mathbb{R}^n:|Tf(x)|>\lambda\}\big)\lesssim \int_{\mathbb{R}^n}\frac{|f(x)|}{\lambda}M_{A_0}w(x)dx.
\end{equation}
We note that the conclusions of Theorem \ref{thm5.2} are complementary to \eqref{ie5.37}.
Indeed, Martell, Pérez, and Trujillo \cite{hor+mar} demonstrated that the $\mathcal{H}_r$ condition is not sufficiently indicative of the validity of the $A_1$ conjecture. In particular, the weak-type $A_1$ estimate for the operator $T$ cannot be directly deduced from \eqref{ie5.37}. This limitation stems from the fact that the Orlicz maximal operators on the right-hand side of \eqref{ie5.37} are constrained by the choice of the initial Young function $A_0$. In contrast, the Orlicz maximal operators considered in our results are not subject to such restrictions. Moreover, in Theorem \ref{thm5.2}, the Young function $A_0$ is coupled with the function $f$, which is in stark contrast to the scenario in \eqref{ie5.37}, where $f$ remains independent of $A_0$. This distinction underscores a key difference: the dependency of the weight on the Young function in \eqref{ie5.37} imposes limitations on flexibility, whereas our results exhibit greater generality by avoiding such dependencies.
Therefore, these two results can be viewed as complementary.\vspace{0.1cm}

Finally, comparing Theorem \ref{thm5.2} with \cite[Theorem 3.8]{hor+lor3}, our results establish a connection between the boundedness constant and the family of Young functions $\varphi_{\vec{l}}:=\varphi_{(l_1,l_2)},\,0\leq l_2 < l_1\leq m$, which lays the groundwork for deriving the subsequent Theorem \ref{thm5.3}.
\end{remark}

Using Theorem \ref{thm5.2}, we can obtain the Fefferman-Stein inequalities for the general commutators of $\omega$-Calder\'{o}n-Zygmund operators with arbitrary weights.

\begin{theorem}\label{thm5.3}
Let $m\in\mathbb{N}$ and $\vec{b}=(b_1,\ldots,b_m)\in \mathrm{BMO}^m$. Assume that $T$ is an $\omega$-Calder\'{o}n-Zygmund operator with $\omega$ satisfying the Dini condition: $\int_{0}^1\omega(t)\frac{dt}{t}<\infty$. Then for any weight $w$ and every $0<\varepsilon<1$, we have
\begin{equation}\label{ie5.28}
w\big(\{x\in \mathbb{R}^n:|T_{\vec b}f(x)|>\lambda\}\big)
\lesssim \frac{1}{\varepsilon}\int_{\mathbb{R}^n}\Phi_m\Big(\frac{\big\|\vec{b}\big\|_{\mathrm {BMO}}|f(x)|}{\lambda}\Big)M_{L(\log L)^m(\log \log L)^{1+\varepsilon}}w(x)dx,
\end{equation}
where $\Phi_m(t)=t\log^m(e+t)$.
\end{theorem}

\begin{remark}\label{remmark5.2}
First of all, in terms of operators, note that when $\omega(t)=t^\delta$ for some $0<\delta<1$, the $\omega$-Calder\'{o}n-Zygmund operator coincides with the classical Calder\'{o}n-Zygmund operator, which means that our results is an essential extension of Theorem B. Secondly, in terms of constant, compared with \cite{per9}, \eqref{ie5.28} completely separates the dependence between the crucial constant $\varepsilon$ and the number $m$ of the commutator, that is, the constant $\frac{1} {\varepsilon^{m+1}}$ in \eqref{ie5.36} has been improved to $\frac{1} {\varepsilon}$. Finally, in terms of weight, the fact $M_{L(\log L)^m(\log\log L)^{1+\varepsilon}}w(x)\leq C_{\varepsilon}M_{L(\log L)^{m+\varepsilon}}w(x)$ shows that Theorem \ref{thm5.3} generalizes the results of \cite{per9}.

\end{remark}

As an application of Theorem \ref{thm5.3}, we can directly obtain the following two classical forms of weak type endpoint estimates.

\begin{corollary}\label{cor5.1}
Let $m$, $T$ and $\vec{b}$ be as the same as in Theorem \ref{thm5.3}, then for any weight $w$ and every $0<\varepsilon<1$, there exists a constant $C_\varepsilon$ independent of $m$ such that
\begin{equation}\label{ie5.27}
w\big(\{x\in \mathbb{R}^n:|T_{\vec b}f(x)|>\lambda\}\big)\leq \frac{C_{n,m,T} C_{\varepsilon}}{\varepsilon}\int_{\mathbb{R}^n}\Phi_m\Big(\frac{\big\|\vec{b}\big\|_{\mathrm{ BMO}}|f(x)|}{\lambda}\Big)M_{L(\log L)^{m+\varepsilon}}w(x)dx,
\end{equation}
where $C_{\varepsilon}=e\log ^{1+\varepsilon}(e+\log (2e))+2^{1+\varepsilon}(1+1/\varepsilon)^{1+\varepsilon}$.

More generally, we can even obtain the following estimate:
\begin{equation}\label{ie5.24}
w\big(\{x\in \mathbb{R}^n:|T_{\vec b}f(x)|>\lambda\}\big)\leq \frac{C_{n,m,T}}{\varepsilon}\int_{\mathbb{R}^n}\Phi_m\Big(\frac{\big\|\vec{b}\big\|_{\mathrm{ BMO}}|f(x)|}{\lambda}\Big)M_{L(\log L)^{m+\varepsilon}}w(x)dx.
\end{equation}
\end{corollary}

\begin{corollary}\label{cor5.2}
Let $m$, $T$ and $\vec{b}$ be as above. If $w\in A^{\rm{weak}}_{\infty}$ $($see Section \ref{section2.4}$)$, then
\begin{equation}\label{ie5.25}
\begin{aligned}
    w\big(&\{x\in \mathbb{R}^n:|T_{\vec b}f(x)|>\lambda\}\big)\\
    &\lesssim K(w)[w]^{m}_{A^{\rm{weak}}_{\infty}}\log(e+[w]_{A^{\rm{weak}}_{\infty}})
    \int_{\mathbb{R}^n}\Phi_m\Big(\frac{\big\|\vec{b}\big\|_{\mathrm {BMO}}|f(x)|}{\lambda}\Big)Mw(x)dx.
\end{aligned}
\end{equation}
Furthermore, if $w\in A_1$, then we have
\begin{equation}\label{ie5.26}
\begin{aligned}
    w\big(&\{x\in \mathbb{R}^n:|T_{\vec b}f(x)|>\lambda\}\big)\\
    &\lesssim K(w)[w]_{A_1}[w]^{m}_{A^{\rm{weak}}_{\infty}}\log(e+[w]_{A^{\rm{weak}}_{\infty}})
    \int_{\mathbb{R}^n}\Phi_m\Big(\frac{\big\|\vec{b}\big\|_{\mathrm {BMO}}|f(x)|}{\lambda}\Big)w(x)dx,
\end{aligned}
\end{equation}
where $K(w)=\Big\{e,e^{2(\tau_n[w]_{A^{\rm{weak}}_{\infty}})^{-1}}\Big\}$, $\tau_n$ is as stated in Lemma $\ref{lem2.2}$.
\end{corollary}

\begin{remark}\label{remmark5.2}
We make some comments on Corollaries \ref{cor5.1} and \ref{cor5.2}. If $m=1$, then $T_{\vec b}=[b,T]$, which means that our results cover Theorem D; when $b_i=b,i=1,\ldots,m$, Theorem 3.1 in \cite{iba} is just a special case of Corollaries \ref{cor5.1} and \ref{cor5.2}. In addition, since the $A^{\rm{weak}}_{\infty}$ weight class is a broader class of weights than the $A_{\infty}$ weight class, and it fails to meet the doubling condition (as exemplified in \cite[Example 3.2]{hor+and1}), the constant $K(w)$ in Corollary \ref{cor5.2} is indispensable.
\end{remark} 

For the quantitative weighted Coifman-Fefferman inequalities of the general commutators, we have
\begin{theorem}\label{thm5.4}
Let $0<p<\infty$, $m\in\mathbb{N}$ and $\vec{b}=(b_1,\ldots,b_m)\in \mathrm{BMO}^m$. Assume that $A,B$ are Young functions such that $A\in \mathcal{Y}(p_0,p_1)~(1\leq p_0\leq p_1<\infty)$ and $B^{-1}(t)(C^{-1}(t))^{m}\leq cA^{-1}(t)$ for $t\geq 1$, where $C(t)=e^{t}-1$. If $T$ is an $\bar{A}$-H\"{o}rmander operator, then for any $w\in A_{\infty}$,
\begin{equation}\label{ie1.1}
\int_{\Rn}|T_{\vec b}(f)(x)|^pw(x)dx\lesssim \prod_{s=1}^m\|b_s\|^p_{\rm BMO}[w]^{mp}_{A_{\infty}}[w]^{\max\{1,p\}}_{A_{\infty}}\int_{\Rn}(M_Bf(x))^pw(x)dx.
\end{equation}
\end{theorem}

\begin{remark}\label{rem5.4}
It is clear that when $b_i=b,i=1,\ldots,m$, Theorem \ref{thm5.4} coincides with \cite[Theorem 2.3]{iba}. In addition, in order to overcome the dependence on the $A_{\infty}$ constant caused by extrapolation, we employ sparse domination to obtain the quantitative estimate for the full range $0<p<\infty$ with a precise control of the $A_{\infty}$ constant. This not only refines \cite[Theorem 2.3]{iba} by extending the range of $p$, but also enhances the results in \cite{hor+lor2} by relaxing the kernel condition and providing more precise quantitative estimates for the weight constants.
\end{remark}

As a corollary of Theorem \ref{thm5.4}, we obtain the following weighted Coifman-Fefferman inequalities for  $L^r$-H\"{o}rmander operators and their general commutators. 

\begin{corollary}\label{cor5.6}
Let $m$, $p$ and $\vec{b}$ be as above. If $1<r<\infty$ and $T$ is a singular integral operator with kernel satisfying $L^r$-H\"{o}rmander condition, then for any $0<\varepsilon<1$ and $w\in A_{\infty}$,
\begin{equation}\label{ie10}
\int_{\Rn}|T_{\vec b}(f)(x)|^pw(x)dx\lesssim \prod_{s=1}^m\|b_s\|^p_{\rm BMO}[w]^{mp}_{A_{\infty}}[w]^{\max\{1,p\}}_{A_{\infty}}\int_{\Rn}(M_{r'+\varepsilon}f(x))^pw(x)dx.
\end{equation}
\end{corollary}

\begin{corollary}\label{cor5.7}
Let $1<r<\infty$ and $T$ be a singular integral operator with kernel satisfying $L^r$-H\"{o}rmander condition, then for any $0<p<\infty$ and $w\in A_{\infty}$,
\begin{equation*}
\int_{\Rn}|T(f)(x)|^pw(x)dx\lesssim [w]^{\max\{1,p\}}_{A_{\infty}}\int_{\Rn}(M_{r'}f(x))^pw(x)dx.
\end{equation*}
Furthermore, if $r'\leq p< \infty$, then for $w\in A_{p/r'}$, $T$ is bounded on $L^p(w)$ and 
\begin{equation*}
\int_{\Rn}|T(f)(x)|^pw(x)dx\lesssim [w]^{p}_{A_{\infty}}[w]^{\frac{p}{p-r'}}_{A_{p/r'}}\int_{\Rn}|f(x)|^pw(x)dx.
\end{equation*}
\end{corollary}

\begin{remark}
Corollary \ref{cor5.7} gives the weighted norm inequalities for $L^r$-H\"{o}rmander operators with an accurate weighted constant, which improves \cite[Theorem 2.2]{hor+mar}.
\end{remark}

As another corollary of Theorem \ref{thm5.4}, we have the following quantitative weighted estimates for the iterated commutators of $\omega$-Calder\'{o}n-Zygmund operators.

\begin{corollary}\label{cor5.4}
Let $0<p<\infty$, $m\in\mathbb{N}$ and $b\in \mathrm{BMO}$. Assume $T$ is an $\omega$-Calder\'{o}n-Zygmund operator with $\omega$ satisfying the Dini condition: $\int_{0}^1\omega(t)\frac{dt}{t}<\infty$. Then for any $w\in A_{\infty}$,
\begin{equation}\label{ie16}
\int_{\Rn}|T^m_{b}f(x)|^pw(x)dx\lesssim \|b\|^{mp}_{\rm BMO}[w]^{mp}_{A_{\infty}}[w]^{\max\{1,p\}}_{A_{\infty}}\int_{\Rn}(M^{m+1}f(x))^pw(x)dx.
\end{equation}
\end{corollary}

 In terms of the dependence on the weight constant, Corollary \ref{cor5.4} is consistent with Theorem C when $p>1$, but when $0<p<1$, Theorem C outperforms our result. This phenomenon arises from the insufficient smoothness of the operators under consideration. To address this issue, we impose more stronger smoothness conditions on the kernels of operators and obtain the following theorem.

\begin{theorem}\label{thm5.5}
 Let $0<p<\infty$, $m\in\mathbb{N}$ and $\vec{b}=(b_1,\ldots,b_m)\in \mathrm{BMO}^m$. Assume that $A,B$ are Young functions such that $B^{-1}(t)(C^{-1}(t))^m\leq cA^{-1}(t)$ for $t\geq 1$, where $C(t)=e^{t}-1$. If $T$ is an $\bar{A}$-H\"{o}rmander operator with kernel $K\in \mathcal{H}_{\bar{A},m}$ $($see Definition \ref{def7.1}$)$. Then for $w\in A_{\infty}$ and any bounded function f with compact support, it holds that
\begin{equation}\label{ie1.3}
\|T_{\vec{b}}f\|_{L^p(w)}\leq C\prod_{i=1}^{m}\|b_i\|_{\rm{BMO}}
[w]^{m+1}_{A_{\infty}}\|M_Bf\|_{L^p(w)}.
\end{equation}
\end{theorem}
\begin{remark}
We now make some comments on Theorem \ref{thm5.5}. Since the quantitative weighted estimate of $T_{\vec b}$ is our primary interest, we do not track the dependence on $p$, but it can be obtained according to our proof procedure; Moreover, compared to Theorem \ref{thm5.4}, Theorem \ref{thm5.5} provides a more precise $A_{\infty}$ constant independently of $p$, although it has the additional restriction $K\in \mathcal{H}_{\bar{A},m}$; Furthermore, if $b_i=b,i=1,\ldots,m$, then
$$\|T^m_{b}f\|_{L^p(w)}\leq C\|b\|^m_{\rm{BMO}}
[w]^{m+1}_{A_{\infty}}\|M_Bf\|_{L^p(w)},$$
which is the quantitative version of \cite[Theorem 3.3 (a)]{hor+lor2}.
\end{remark}

\begin{corollary}\label{cor5.5}
Let $0<p<\infty$, $m\in\mathbb{N}$ and $b\in \mathrm{BMO}$. Assume that $T$ is an $\omega$-Calder\'{o}n-Zygmund operator with $\omega$ satisfying the $\log$-Dini condition:
\begin{equation}\label{ie1.4}
\int_{0}^1\frac{\omega(t)}{t}\Big(1+\log \frac{1}{t}\Big)^{m}dt<\infty.
\end{equation}
If $w\in A_{\infty}$, then for any bounded function f with compact support,
\begin{equation}\label{ie1.5}
\|T_{\vec b}f\|_{L^p(w)}\leq C \prod_{i=1}^{m}\|b_i\|_{\rm BMO}[w]^{m+1}_{A_{\infty}}\|M^{m+1}f\|_{L^p(w)}.
\end{equation}
\end{corollary}


The organization of this article is as follows: In Section \ref{sec2}, we introduce some basic definitions and present some known facts as well as lemmata needed throughout the rest of this paper. Section \ref{sec3} contains the proof of pointwise sparse domination \eqref{ie0.1} for the general commutators, which plays a fundamental role in our analysis. The weak endpoint type  estimates with arbitrary weights are proved in Section \ref{sec4} and the purpose of Section \ref{sec5} is to prove Theorem \ref{thm5.3}. Section \ref{sec6} is devoted to establishing the quantitative weighted Coifman-Fefferman inequality based on the sparse domination. The proof of Theorem \ref{thm5.5} is presented in Section \ref{sec7}. Finally, some applications, including Calder\'{o}n commutators, homogeneous singular operators and Fourier multipliers, will be given in Section \ref{sec8}.
\vspace{0.1cm}

Throughout this paper, we will use $C$ or $c$ to denote a positive constant, which is independent of the main parameters and may vary at each occurrence. For a cube $Q$ and $a>0$, $aQ$ denotes the cube with center as $Q$ and whose side length is $a$ times that of $Q$. Given a function $f$ and a cube $Q$, $\langle f \rangle_Q$ denotes the average of $f$ on $Q$, that is, $\langle f \rangle_Q=\frac{1}{|Q|}\int_{Q}f(x)dx$. For a fixed $p$ with $1\leq p<\infty$, $p'$ denotes the dual exponent of $p$, that is, $p'=p/(p-1)$. For a multi-index $\alpha=(\alpha_1,\ldots,\alpha_n)\in \mathbb{N}^n$, $|\alpha|=\sum_{i=1}^n\alpha_i$. Besides, the nation $A\lesssim B$, we mean that $A \leq CB$ for some constant $C > 0$ and $A \approx B$ stands for $A\lesssim B\lesssim A$.

\section{Preliminary}\label{sec2}
In this section, we present some definitions,  fundamental facts and lemmas which will be used later.

\subsection{$\omega$-Calder\'{o}n-Zygmund operators}
Let $\omega:[0,1]\rightarrow [0,\infty)$ be a modulus of continuity, that is, a continuous, increasing, submultiplicative function with $\omega(0)=0$.

Recall that $T$ is called an $\omega$-Calder\'{o}n-Zygmund operator if it 
is linear, bounded on $L^2$ and admits the following representation
$$Tf(x)=\int_{\Rn}K(x,y)f(y)dy,\,\,x\notin \supp f,$$
where $f\in L^1_{\rm{loc}}(\Rn)$ and $K$, defined away from the diagonal $x=y$ in $\Rn$, is a locally integrable function, satisfying the size condition
$$|K(x,y)|\leq \frac{c_{K}}{|x-y|^n},\,\,x\neq y,$$
and a smoothness condition
$$|K(x,y)-K(x',y)|+|K(y,x)-K(y,x')|\leq \omega\Big(\frac{|x-x'|}{|x-y|}\Big)\frac{1}{|x-y|^n},$$
for $|x-y|>2|x-x'|$.
\subsection{Sparse family}
In this subsection, we will introduce a very useful tool named dyadic calculus, which was systematically studied in \cite{hor+ler18}.

Given a cube $Q\subset \Rn$, let $\mathcal{D}(Q)$ denote the set of all dyadic cubes with respect to $Q$, namely, the cubes from $\mathcal{D}(Q)$ are generated by repeating subdivision of $Q$ and its descendants in $2^n$ cubes with the same length.

\begin{definition}
A family of cubes $\mathcal{D}$ is called to be a dyadic lattice if it satisfies the following properties:
\begin{enumerate}
  \item if $Q\in \mathcal{D}$, then every descendant of $Q$ is in $\mathcal{D}$ as well; 
  \item for every two cubes $Q_1, Q_2 \in \mathcal{D}$, there exists a common ancestor $Q\in \mathcal{D}$ such that $Q_1, Q_2 \in \mathcal{D}(Q)$;
  \item for every compact set $K \subset \Rn$, there exists a cube $Q\in \mathcal{D}$ such that $K\subset Q$.
\end{enumerate}
\end{definition}
Note that, if $Q_0\in \mathcal{D}$, then every cube in $\mathcal{D}(Q_0)$ will also belong to $\mathcal D$.

The following Three Lattice Theorem, a fundamental concept in dyadic calculus (as outlined in \cite[Theorem 3.1]{hor+ler18}), palys a crucial role in elucidating the intricate structure of dyadic lattices.

\begin{lemma}[\cite{hor+ler18}]\label{lem2.1}
Given a dyadic lattice $\mathcal D$, there exist $3^n$ dyadic lattices $\{\mathcal{D}_j\}_{j=1}^n$
such that
$$\{3Q: Q\in \mathcal{D}\}=\bigcup_{j=1}^{3^n}\mathcal{D}_j$$
and for each cube $Q\in \mathcal{D}$, there 
exists a unique cube $R_Q$ in some $\mathcal{D}_j$ such that $Q\subset R_Q$ and $3l_Q=l_{R_{Q}}$.
\end{lemma}
\begin{remark}\label{rem2.1}
Given a dyadic lattice $\mathcal D$. For an arbitrary cube $Q\subset \Rn$, there exists a cube $Q'\in \mathcal{D}$ such that $\frac{l_Q}{2}<l_{Q'}\leq l_Q$ and $Q\subset3Q'$. 
It follows from Lemma \ref{lem2.1} that $3Q'=P\in \mathcal {D}_j$ for some $j\in \{1,\ldots,m\}$. Therefore, 
for every cube $Q\subset \Rn$, we can find a $P\in \mathcal{D}_j$ for some $j$ such that $Q\subset P$ and $l_P \leq 3l_Q$.
\end{remark}
With the preceding definition and lemma in hand, we are ready to give the definition of sparse family.
\begin{definition}
Let $\mathcal D$ be a dyadic lattice. $\mathcal {S}\subset \mathcal{D}$ is called a $\eta$-sparse family with $0<\eta<1$ if for every cube $Q\in \mathcal{S}$, there exists a measurable subset $E_{Q}\subset Q$ such that $\eta|Q|\leq |E_{Q}|$, where the $\{E_{Q}\}$ are pairwise disjoint.
\end{definition}

\subsection{Sharp maximal functions}
For a locally integrable function $f$, The basic tools we will use are the standard Hardy-Littlewood maximal function and the sharp maximal function of Fefferman and Stein \cite{hor+fef} defined by
$$M(f)(x)=\sup_{Q\ni x}\frac{1}{|Q|}\int_Q|f(y)|dy,$$
and
$$
M^{\#}(f)(x)=\sup _{Q\ni x}\inf _{c \in \R} \frac{1}{|Q|}\int_Q|f(y)-c|dy\approx \sup _{Q\ni x}\frac{1}{|Q|}\int_Q|f(y)-\langle f\rangle_Q|dy.
$$
where the supremum is taken over all cubes $Q$ containing $x$.
For $0<\delta<\infty$, denote $M_{\delta}$ and $M^{\#}_{\delta}$ by
$$
M_{\delta}(f)(x)=(M_{\delta}(|f|^{\delta})(x))^{1/\delta}=\Big(\sup_{Q\ni x}\frac{1}{|Q|}\int_Q|f(y)|^{\delta}dy\Big)^{1/\delta},
$$
and
$$M^{\#}_{\delta}(f)(x)=(M^{\#}(|f|^{\delta})(x))^{1/\delta}.$$
If we restrict our considerations to dyadic cubes, we adopt the following notations accordingly $M^{\#, d}$, $M^{d}_{\delta}$ and $M^{\#, d}_{\delta}$.
The relationship between $M^{d}_{\delta}$ and $M^{\#, d}_{\delta}$ is given in \cite[p.10]{hor+ort1}.
\begin{lemma}[\cite{hor+ort1}]\label{lem2.4}
Let $0 < p <\infty, 0< \delta <1$, and $w\in A_{\infty}$. Then there exists a constant $C$ only depending on $n$ such that
\begin{equation*}
\|M^{d}_{\delta}(f)\|_{L^p(w)}\leq C\max\{1, p\}[w]_{A_{\infty}}\|M^{\#, d}_{\delta}(f)\|_{L^p(w)},
\end{equation*}
for any function $f$ such that the left-hand side is finite.
\end{lemma}
\subsection{The classical weights}\label{section2.4}
This subsection primarily aims to introduce the pertinent definitions of classical weights and the weak $A_{\infty}$ class, along with their key properties. Let us start with the Muckenhoupt weights. 

We call that a nonnegative and locally integrable function $w$ on $\Rn$ belongs to the Muckenhoupt class $A_p$ with $1< p <\infty$, if
$$[w]_{A_p}:=\Big(\frac{1}{|Q|}\int_{Q}w(x)dx\Big)\Big(\frac{1}{|Q|}\int_{Q}w^{1-p'}(x)dx\Big)^{p-1}< \infty.$$ 
For $p=1$, a weight $w$ belongs to the class $A_1$ if there exsits a constant $C$ such that
$$\frac{1}{|Q|}\int_{Q}w(x)dx\leq C \mathop{\mathrm{ess\,inf}}\limits_{x\in Q} w(x),$$
and the infimum of these constants $C$ is called the $A_1$ constant of $w$, denoted by $[w]_{A_1}$. Since the $A_p$ classes are increasing with respect to $p$, the $A_{\infty}$ class of weights is defined in a natural way by $$A_{\infty}=\bigcup_{p>1}A_p.$$
It is not clear how to characterize the class $A_{\infty}$ in terms of a constant from the definition of $A_{\infty}$. However, it was proved by Fujji\cite{hor+fuj} and Wilson \cite{hor+wil}
that a weight $w\in A_{\infty}$ if and only if 
$$[w]_{A_{\infty}}:=\sup_{Q}\frac{1}{w(Q)}\int_{Q}M(w\chi_{Q})(x)dx<\infty,$$
which is known as the Fujii-Wilson $A_{\infty}$ constant. For $A_{\infty}$ weights, it turns out to the following lemma can be regarded as a characterization of its quantitative version. 
\begin{lemma}[\cite{iba}]\label{lem2.3}
There exists $c_n>0$ such that for every $w\in A_{\infty}$, every cube $Q$ and every measurable subset $E\subset Q$, it holds that
$$\frac{w(E)}{w(Q)}\leq 2\Big(\frac{|E|}{|Q|}\Big)^{\frac{1}{c_n[w]_{A_{\infty}}}}.$$
\end{lemma}

Now we introduce a class of weights which genuinely weakens the $A_{\infty}$ class. A weight $w$ is said to belong to weak $A_{\infty}$ class, written $w\in A^{weak}_{\infty}$, if there exist $0<c,\delta<\infty$ such that the estimate 
$$w(E)\leq c \Big(\frac{|E|}{|Q|}\Big)^{\delta}w(2Q)$$
holds for every cube and every measurable set $E\subset Q$. Similar to the Fujii-Wilson $A_{\infty}$ constant, $w\in A^{weak}_{\infty}$ if and only if
$$[w]^{weak}_{A_{\infty}}:=\sup_{Q}\frac{1}{w(2Q)}\int_{Q}M(w\chi_{Q})(x)dx<\infty.$$
It has been demonstrated in \cite{hor+and1} that the constant 2 in the average could be replaced by any parameter $\sigma >1$ and the following reverse H\"{o}lder's inequality is valid.
\begin{lemma}[\cite{hor+and1}]\label{lem2.2}
Let $w\in A^{weak}_{\infty}$, then for all cubes $Q$ in $\Rn$ with sides parallel to the axes,
$$\Big(\frac{1}{|Q|}\int_{Q}w^{r(w)}(x)dx\Big)^{\frac{1}{r(w)}}\leq \frac{2}{|2Q|}\int_{2Q}w(x)dx,$$
with 
$$r(w):=1+\frac{1}{\tau_n[w]^{weak}_{A_{\infty}}},$$
where $\tau_n$ is a dimensional constant that we may take to be $\tau_n\simeq 2^n$.
\end{lemma}
\subsection{Young functions and Orlicz maximal functions}\label{section2.5}
 In this subsection, we introduce some fundamental facts regarding Young functions and Orlicz spaces. We refer the readers to \cite{hor+rao} and \cite{hor+har} for more related information.

A function $A$ defined on $[0, \infty)$ is called a Young function, if $A$ is a non-negative, increasing and convex function with $A(0)=0$ and $A(t)\rightarrow \infty$ as $t\rightarrow \infty$. It is clear that if $A$ is a Young function, then $A(t)/t$ is not decreasing. Given a cube $Q$, the localized and normalized Luxemburg norm associated with the Orlicz space $L^A(\mu)$ is defined by
$$
\|f\|_{A(\mu),Q}=\inf\Big\{\lambda>0: \frac{1}{\mu(Q)}\int_QA\big(\frac{|f(x)|}{\lambda}\big)d\mu \leq 1\Big\}.$$
For convenience, if $\mu$ is the Lebesgue measure, we denote $\|f\|_{A(\mu),Q}=\|f\|_{A,Q}$
and if $\mu=wdx$ is an absolutely continuous measure with respect to the Lebesgue measure, we write $\|f\|_{A(\mu),Q}=\|f\|_{A(w),Q}$.

Now we review several useful conclusions about the Luxemburg norm. First of all, when $A(t)=t^r$ with $r\geq 1$, we have 
$\|f\|_{A,Q}=\|f\|_{L^r,Q}:=\big(\frac{1}{|Q|}\int_{Q}|f(x)|^rdx\big)^{1/r}$,
which goes back to the standard $L^{r}\big(Q,\frac{dx}{|Q|}\big)$ norm. In addition, it is easy to check that if $A$ and $B$ are Young functions with $A(t)\leq cB(t)$ for $t>t_0$, then $\|f\|_{A,Q}\leq c\|f\|_{B,Q}$, which can be seen as a generalized Jensen's inequality. Another noteworthy property of Orlicz space is the following generalized H\"{o}lder's inequality:
\begin{equation}\label{ie2.1}
\frac{1}{\mu(Q)}\int_{Q}|fg|d\mu\leq 2\|f\|_{A(\mu),Q}\|g\|_{\bar{A}(\mu),Q},
\end{equation}
where $\bar{A}(t)=\sup_{s>0}\{st-A(s)\}$ is the complementary function of Young function $A$. For instance, when $A(t)=t^p/p$ with $1<p<\infty$, then $\bar{A}(t)=t^{p'}/p'$ and if $A(t)=t$, a simple calculation yields that
\begin{equation*}
\bar{A}(t)=\sup_{s>0}\left\{s(t-1)\right\}=
    		\begin{cases}
    			0,  &t\leq 1,\\
    			\infty, &t>1.
    		\end{cases}
\end{equation*}
Now we introduce an additional generalization of \eqref{ie2.1} that is beneficial for our work. If $A,B$ and $C$ are Young functions satisfying $A^{-1}(t)B^{-1}(t)C^{-1}(t)\leq ct$ for $t\geq t_0$, then for cube $Q\subset \Rn$, it holds
$$\|fg\|_{\bar{C}(\mu),Q}\leq c\|f\|_{A(\mu),Q}\|g\|_{B(\mu),Q},$$
here $A^{-1}(x)=\inf\{y\in \R: A(y)>x\}$ with $0\leq x \leq \infty$, where $\inf \emptyset =\infty$.
In view of the average of the Luxemburg norm, it is natural to define a general maximal operator associated to the Young function $A$ as
$$M_{A}f(x)=\sup_{Q\ni x}\|f\|_{A,Q},$$
where the supremum is taken over all the cubes containing $x$. 

We now give some examples of maximal operators related to certain Young functions.
\begin{enumerate}
  \item If $A(t)=t^r$ with $1<r<\infty$, then $M_{A}=M_{r}$;
  \item If $A(t)=t\log^{\alpha}(e+t)$ with $\alpha>0$, then $\bar{A}(t)\simeq e^{t^{1/\alpha}}-1$ and we denote $M_{A}=M_{L (\log L)^{\alpha}}$. Notice that $M\leq M_A \lesssim M_r$ for all $1<r<\infty$. Moreover, $M_A\simeq M^{l+1}$, where $\alpha=l\in \mathbb{N}$ and $M^{l+1}$ is $M$ iterated $l+1$ times;
   \item If $A(t)=t\log^{\alpha}(e+t)\log^{\beta}(e+\log(e+t))$ with $\alpha,\beta>0$, then $M_A=M_{L(\log L)^\alpha(\log\log L)^{\beta}}$. In addition, it follows from a detailed calculations that for every $0<\varepsilon<1$,
       $$M_{L(\log L)^m(\log\log L)^{1+\varepsilon}}w(x)\leq C_{\varepsilon}M_{L(\log L)^{m+\varepsilon}}w(x),$$
       where $C_{\varepsilon}=e\log^m(2e)\log^{1+\varepsilon}(e+\log(2e))+
       2^{1+\varepsilon}(1+1/\varepsilon)^{1+\varepsilon}$.
\end{enumerate}

 We end this subsection by stating the Fefferman-Stein inequality for $M_A$.
\begin{lemma}[{\cite[Lemma 4.3]{iba}}]\label{lem5.4}
Let $A$ be a Young function. For any arbitrary weight $w$ and $\lambda>0$, we have
$$w\big(\{x\in\mathbb{R}^n:\ M_{A}f(x)>\lambda\}\big)\leq 3^n\int_{\mathbb{R}^n}A\Big(\frac{9^n|f(x)|}{\lambda}\Big)Mw(x)dx.$$
If additionally $A$ is submultiplicative, then
\begin{equation}\label{ie12}
w\big(\{x\in\mathbb{R}^n:\ M_{A}f(x)>\lambda\}\big)\leq C_{n,A}\int_{\mathbb{R}^n}A\Big(\frac{|f(x)|}{\lambda}\Big)Mw(x)dx.
\end{equation}
\end{lemma}
\vspace{0.2cm}

\section{ Proofs of Theorem \ref{thm5.1}}\label{sec3}
In this section, we aim to establish the sparse domination for $T_{\vec b}$. We begin by giving some notations and lemmas. 
Given a sublinear operator $T$, define the grand maximal truncated operator $\mathcal{M}_{T}$ by
$$\mathcal{M}_{T}f(x):=\sup_{Q\ni x} \mathop{\mathrm{ess\,sup}}\limits_{\xi \in Q}|T(f\chi_{\mathbb{R}^n\backslash 3Q})(\xi)|,$$
where the supremum is taken over all cubes $Q\subset \Rn$ containing $x$. Given a cube $Q_0$, for $x\in Q_0$, a local version of $\mathcal{M}_{T}$ is defined as follows
$$\mathcal{M}_{T,Q_0}f(x):=\sup_{Q\ni x,Q\subseteq Q_0} \mathop{\mathrm{ess\,sup}}\limits_{\xi \in Q}|T(f\chi_{3Q_0\backslash 3Q})(\xi)|.$$

The following results concerning the grand maximal truncated operator $\mathcal{M}_{T}$ and  A-H\"{o}rmander operator $T$ play a crucial role in our analysis, see \cite[Lemma 4.1]{iba} and \cite[Lemma 5.1]{iba}.

\begin{lemma}[\cite{iba}]\label{lem5.2}
Let $A$ be a Young function. If $T$ is an $\bar{A}$-H\"{o}rmander operator, then
$$\|T\|_{L^1\rightarrow L^{1,\infty}}\leq c_n\big(\|T\|_{L^2\rightarrow L^2}+H_{\bar A}\big).$$
Furthermore, as a consequence of Marcinkiewicz interpolation theorem and the fact that $T$ is almost self-dual, we have
$$\|T\|_{L^p\rightarrow L^p}\leq c_n\big(\|T\|_{L^2\rightarrow L^2}+H_{\bar A}\big),\,\,1<p<\infty.$$
\end{lemma}

\begin{lemma}[\cite{iba}]\label{lem5.1}
Let $A\in \mathcal{Y}(p_0,p_1)$ be a Young function and $T$ be an $\bar{A}$-H\"{o}rmander singular internal operator. Then the following estimates hold
\begin{enumerate}
\item For a.e. $x\in Q_0$,
    $$|T(f\chi_{3Q_0})(x)|\leq c_n\|T\|_{L^1\rightarrow L^{1,\infty}}f(x)+\mathcal{M}_{T,Q_0}f(x).$$
\item For all $x\in \mathbb{R}^n$ and $\delta\in(0,1)$,
    $$\mathcal{M}_Tf(x)\leq c_{n,\delta}\Big(H_{\bar{A}}M_{A}f(x)+M_{\delta}(Tf)(x)+\|T\|_{L^1\rightarrow L^{1,\infty}}Mf(x)\Big).$$
    Furthermore, for any $\lambda>0$,
    \begin{align*}
    \big|&\left\{x\in\mathbb{R}^n:\mathcal{M}_Tf(x)>\lambda\right\}\big|\\
    &\leq\int_{\mathbb{R}^n}A\bigg(c_{n,p_0,p_1}\max\{c_{A,p_0},c_{A,p_1}\}\Big(H_{\bar{A}}+\|T\|_{L^2\rightarrow L^2}\Big)\frac{|f(x)|}{\lambda}\bigg)dx.
\end{align*}
\end{enumerate}
\end{lemma}

\begin{proof}[Proof of Theorem $\ref{thm5.1}$]
For simplicity, we only consider the case $m=2$. Our method still holds for all $m\in \N$ with little modifications. 

Fix a cube $Q_0 \subset \Rn$. We claim that there exists a $\frac{1}{2}$-sparse family $\mathcal{F}\subseteq\mathcal{D}(Q_0)$ such that for a.e. $x\in Q_0$,
\begin{equation}\label{ie5.2}
|T_{\vec b}(f\chi_{3Q_0})(x)|\leq C_{n,m}C_{T}\sum_{\vec{\gamma}\in{\{0,1\}^2}}\mathcal{B}^{\vec \gamma}_{A,\mathcal{F}}(\vec{b},f)(x),
\end{equation}
where
$$\mathcal{B}^{\vec \gamma}_{A,\mathcal{F}}(\vec{b},f)(x):=\sum_{Q\in \mathcal{F}}\prod_{s:\gamma_s=1}\big|b_{s}(x)-\langle b_{s}\rangle_{R_Q}\big|\Big\|\prod_{s:\gamma_s=0}(b_{s}-\langle b_{s}\rangle_{R_Q})f\Big\|_{A,3Q}\chi_{Q}(x).$$
We postpone the proof of this claim to the end of this section. Assuming the above claim is true, now take a partition of $\Rn$ by cubes $Q_j$ such that ${\rm{supp}}\ f\subset 3Q_0$. In fact, we can take a cube $Q_0$ such that ${\rm{supp}}\ f\subset 3Q_0$ and cover $3Q_0\backslash Q_0$ by $3^n-1$ congruent cubes $Q_j$. Each of them satisfies $Q_0\subset 3Q_j$. Continue to decompose $9Q_0\backslash 3Q_0$ in the same way and repeat the above process. Then the union of the resulting cubes $Q_j$, including $Q_0$, will satisfy the desired properties.

With the preceding partition at our disposal, applying \eqref{ie5.2} to each $Q_j$, we obtain that for a.e. $x \in Q_j$,
$$|T_{\vec b}(f)(x)|\leq C_{n,m}C_{T}\sum_{\vec{\gamma}\in{\{0,1\}^2}}\mathcal{B}^{\vec \gamma}_{A,\mathcal{F}_j}(\vec{b},f)(x),$$
where each $\mathcal{F}_j$ is a $\frac{1}{2}$-sparse family. Let $\mathcal{F}=\cup_j\mathcal{F}_j$, then it is easy to check that $\mathcal{F}$ is also a $\frac{1}{2}$-sparse family and
$$|T_{\vec b}(f)(x)|\leq C_{n,m}C_{T}\sum_{\vec{\gamma}\in{\{0,1\}^2}}\mathcal{B}^{\vec \gamma}_{A,\mathcal{F}}(\vec{b},f)(x),\quad \text{a.e.} \ x \in \mathbb{R}^n.$$
By Remark \ref{rem2.1}, we can deduce that for every $Q \subset \Rn$, there exists a cube $R_Q\in \mathcal{D}_j$ for some $j$, for which $3Q \subset R_Q$ and $|R_Q|\leq 9^n|Q|$. Therefore, $\|f\|_{A,3Q}\leq c_n\|f\|_{A,R_Q}$. Let
$$\mathcal{S}_j=\{R_Q\in \mathcal{D}_j:~Q\in \mathcal{F}\}.$$
Then the $\frac{1}{2}$-sparseness of $\mathcal{F}$ implies that $\mathcal{S}_j$ is a $\frac{1}{2\cdot9^n}$-sparse family and
$$|T_{\vec b}f(x)|\leq
C_{n,m}C_{T}\sum_{j=1}^{3^n}\sum_{\vec{\gamma}\in\{0,1\}^2}\mathcal{A}^{\vec\gamma}_{A,\mathcal{S}_j}(\vec{b},f)(x), \ \text{a.e.}\ x \in \mathbb{R}^n.$$
The proof of Theorem \ref{thm5.1} is finished.

Now, we turn our attention to the proof of \eqref{ie5.2}. It suffices to show that for a.e. $x\in Q_0$, the following recursive inequality holds
\begin{align*}
|T_{\vec b}(f\chi_{3Q_0})(x)|\chi_{Q_0}(x)
&\lesssim \Big(\Big\|\prod_{i=1}^2(b_i-\langle b_i \rangle_{R_{Q_0}})f\Big\|_{A,3Q_0}+\prod_{i=1}^2|b_i(x)-\langle b_i \rangle_{R_{Q_0}}|\|f\|_{A,3Q_0}\\
&\quad+\big|b_1(x)-\langle b_1 \rangle_{R_{Q_0}}\big|\big\|(b_2-\langle b_2 \rangle_{R_{Q_0}})f\big\|_{A,3Q_0}\\
&\quad+\big|b_2(x)-\langle b_2 \rangle_{R_{Q_0}}\big|\big\|(b_1-\langle b_1 \rangle_{R_{Q_0}})f\big\|_{A,3Q_0}\Big)\chi_{Q_0}(x)\\
&\quad+\sum_j|T_{\vec{b}}(f\chi_{3P_j})(x)|\chi_{P_j}(x),
\end{align*}
where $\{P_j\}\subset \mathcal{D}({Q_0})$ are pairwise disjoint cubes such that $\sum_{j}|P_j|\leq\frac{1}{2}|Q_0|$.

Notice that the last term of the above inequality has the same form as the left side, therefore we can replace $|T_{\vec b}(f\chi_{3Q_0})(x)|\chi_{Q_0}(x)$ with $|T_{\vec{b}}(f\chi_{3P_j})(x)|\chi_{P_j}(x)$
and repeat the iterative process.

Let $\mathcal{F}=\{P^k_{j}\}$, $k\in \mathbb N$, where $\{P^0_j\}=\{Q_0\}$, $\{P^1_j\}=\{P_j\}$, $\{P^{k}_j\}$ are the cubes obtained at the $k$-th stage of the iterative process.
Then $\mathcal{F}$ is a $\frac{1}{2}$-sparse family, indeed, for each $P^{k}_j$, we can take $E_{P^{k}_j}=P^{k}_j\backslash P^{k}_{j+1}$.

Observe that for any pairwise disjoint cubes $\{P_j\} \subset \mathcal{D}(Q_0)$, the linear property of $T_{\vec{b}}$ implies that 
\begin{align*}
  |T_{\vec{b}}(f\chi_{3Q_0})(x)|\chi_{Q_0}(x)&\leq |T_{\vec{b}}(f\chi_{3Q_0})(x)|\chi_{{Q_0}\backslash \cup_jP_j}(x)+\sum_{j}|T_{\vec{b}}(f\chi_{3Q_0})(x)|\chi_{P_j}(x) \\
  &\leq|T_{\vec{b}}(f\chi_{3Q_0})(x)|\chi_{{Q_0}\backslash \cup_jP_j}(x)+\sum_{j}|T_{\vec{b}}(f\chi_{3Q_0 \backslash 3P_j})(x)|\chi_{P_j}(x)\\
  &\quad+\sum_{j}|T_{\vec{b}}(f\chi_{3P_j})(x)|\chi_{P_j}(x).
\end{align*}
In order to prove the recursive inequality, it is sufficient to show that there exist
pairwise disjoint cubes $\{P_j\} \subset \mathcal{D}(Q_0)$ such that $\sum_{j}|P_j|\leq\frac{1}{2}|Q_0|$ and for a.e. $x\in Q_0$,
\begin{align*}
|T&_{\vec{b}}(f\chi_{3Q_0})(x)|\chi_{{Q_0}\backslash \cup_jP_j}(x)+\sum_{j}|T_{\vec{b}}(f\chi_{3Q_0 \backslash 3P_j})(x)|\chi_{P_j}(x)\\
&\lesssim\Big(\Big\|\prod_{i=1}^2(b_i-\langle b_i \rangle_{R_{Q_0}})f\Big\|_{A,3Q_0}+\big|b_1(x)-\langle b_1 \rangle_{R_{Q_0}}\big|\big\|(b_2-\langle b_2 \rangle_{R_{Q_0}})f\big\|_{A,3Q_0}\\
&\quad+\big|b_2(x)-\langle b_2 \rangle_{R_{Q_0}}\big|\big\|(b_1-\langle b_1 \rangle_{R_{Q_0}})f\big\|_{A,3Q_0}
+\prod_{i=1}^2\big|b_i(x)-\langle b_i \rangle_{R_{Q_0}}\big|\|f\|_{A,3Q_0}\Big)\chi_{Q_0}(x).
\end{align*}
To this aim, using the translational invariance of $T_{\vec{b}}$, i.e., $T_{\vec{b}-\vec{c}}=T_{\vec{b}}$, where $\vec{c}=(c_1,\ldots,c_m)$, $c_i\in \mathbb{R}$, $1\leq i\leq m$, it follows that
\begin{align*}
|T_{\vec{b}}(f\chi_{3Q_0})(x)|\chi_{{Q_0}\backslash \cup_jP_j}(x)&\leq\bigg[\prod_{i=1}^2\big|b_i(x)-\langle b_i \rangle_{R_{Q_0}}\big|\big|T(f\chi_{3Q_0})(x)\big|\\
&\quad+\big|b_1(x)-\langle b_1 \rangle_{R_{Q_0}}\big|\big|T\big((b_2-\langle b_2 \rangle_{R_{Q_0}})f\chi_{3Q_0}\big)(x)\big|\\
&\quad+\big|b_2(x)-\langle b_2 \rangle_{R_{Q_0}}\big|\big|T\big((b_1-\langle b_1 \rangle_{R_{Q_0}})f\chi_{3Q_0}\big)(x)\big|\\
&\quad+\Big|T\Big(\prod_{i=1}^2(b_i-\langle b_i \rangle_{R_{Q_0}})f\chi_{3Q_0}\Big)(x)\Big|\bigg]\chi_{{Q_0}\backslash \cup_jP_j}(x),
\end{align*}
and
\begin{align*}
\sum_{j}|T_{\vec{b}}(f\chi_{3Q_0 \backslash 3P_j})(x)|\chi_{P_j}(x)&\leq\sum_{j}\bigg[\prod_{i=1}^2\big|b_i(x)-\langle b_i \rangle_{R_{Q_0}}\big|\big|T(f\chi_{3Q_0\backslash3P_{j}})(x)\big|\\
&\quad+\big|b_1(x)-\langle b_1 \rangle_{R_{Q_0}}\big|\big|T\big((b_2-\langle b_2 \rangle_{R_{Q_0}})f\chi_{3Q_0\backslash3P_{j}}\big)(x)\big|\\
&\quad+\big|b_2(x)-\langle b_2 \rangle_{R_{Q_0}}\big|\big|T\big((b_1-\langle b_1 \rangle_{R_{Q_0}})f\chi_{3Q_0\backslash3P_{j}}\big)(x)\big|\\
&\quad+\Big|T\Big(\prod_{i=1}^2(b_i-\langle b_i \rangle_{R_{Q_0}})f\chi_{3Q_0\backslash3P_{j}}\Big)(x)\Big|\bigg]\chi_{P_j}(x).
\end{align*}
For $i=1,2$, we define the sets $E_i$ as follows:
\begin{align*}
& E_i:=\Big\{x\in Q_0:|b_i(x)-\langle b_i \rangle_{R_{Q_0}}||f(x)|>\alpha_n\|(b_i-\langle b_i \rangle_{R_{Q_0}})f\|_{A,3Q_0}\Big\} \\
&\quad \ \ \bigcup \Big\{x\in Q_0:\ \mathcal{M}_{T,Q_0}\big(|b_i-\langle b_i \rangle_{R_{Q_0}}|f\big)(x)>\alpha_nC_{T}\|(b_i-\langle b_i \rangle_{R_{Q_0}})f\|_{A,3Q_0}\Big\}.
\end{align*}
Then define $E_3,E_4$:
\begin{align*}
& E_3:=\Big\{x\in Q_0:\ |f(x)|>\alpha_n\|f\|_{A,3Q_0}\Big\}\bigcup \Big\{x\in Q_0:\ \mathcal{M}_{T,Q_0}(f)(x)>\alpha_nC_{T}\|f\|_{A,3Q_0}\Big\}, \\
& E_4:=\Big\{x\in Q_0:\ \prod_{i=1}^2|b_i(x)-\langle b_i \rangle_{R_{Q_0}}||f(x)|>\alpha_n\Big\|\prod_{i=1}^2(b_i-\langle b_i \rangle_{R_{Q_0}})f\Big\|_{A,3Q_0}\Big\}\\
&\quad \ \ \bigcup \Big\{x\in Q_0:\ \mathcal{M}_{T,Q_0}\Big(\prod_{i=1}^2|b_i-\langle b_i \rangle_{R_{Q_0}}|f\Big)(x)>\alpha_nC_{T}\Big\|\prod_{i=1}^2(b_i-\langle b_i \rangle_{R_{Q_0}})f\Big\|_{A,3Q_0}\Big\},
\end{align*}
where the constant $\alpha_n$ will be choosen later. Let $E=\bigcup_{i=1}^4E_i$. The convexity of $A$, together with Lemma \ref{lem5.1}~(2), yields that
\begin{align*}
|E_1|&\leq \frac{\int_{Q_0}\big|(b_1(x)-\langle b_1\rangle_{R_{Q_0}})f(x)\big|dx}{\alpha_n\|(b_1-\langle b_1\rangle_{R_{Q_0}})f\|_{A,3Q_0}}
+\int_{\mathbb{R}^n}A\bigg(\frac{|(b_1(x)-\langle b_1 \rangle_{R_{Q_0}})f(x)|}{\alpha_n \|(b_1-\langle b_1\rangle_{R_{Q_0}})f\|_{A,3Q_0}}\bigg)dx\\
&\leq3^n\frac{c_n}{\alpha_n}|Q_0|.
\end{align*}
The same reasoning applies to the other $E_i$, we obtain
 $$|E_i|\leq 3^n\frac{c_n}{\alpha_n}|Q_0|,\quad \hbox{for\ }
i=2,3,4.$$
We then choose $\alpha_n$
big enough such that $|E|\leq\frac{1}{2^{n+2}}|Q_0|$.

Applying the Calder\'{o}n-Zygmund decomposition to the function $\chi_{E}$ on the cube $Q_{0}$ at height $\lambda=\frac{1}{2^{n+1}}$ yields that, there exists a sequence of pairwise disjoint cubes $\{P_j\} \subset \mathcal{D}(Q_{0})$ such that
$$
\chi_{E}(x) \leq \frac{1}{2^{n+1}},\ \text{a.e.}\ x \notin \bigcup_{j} P_j.
$$
It follows that $|E \backslash \bigcup_{j} P_{j}|=0$. Furthermore, it is easy to check that
$$\sum_{j}|P_j|=\Big|\bigcup_{j} P_{j}\Big|\leq 2^{n+1}|E|\leq \frac{1}{2}|Q_0|,$$
and
$$\frac{1}{2^{n+1}}\leq\frac{1}{|P_j|}\int_{P_j}\chi_{E}(x)dx=\frac{|P_j \bigcap E|}{|P_j|}\leq \frac{1}{2},$$
which implies that $|P_j\cap E^c|>0$.

We note that for each $P_j$, the fact $P_j\cap E^c\neq \emptyset$ gives that
$$\mathcal{M}_{T,Q_0}(f)(x)\leq\alpha_nC_{T}\|f\|_{A,3Q_0}.$$
Then,
\begin{equation}\label{ie9}
\mathop{\mathrm{ess\,sup}}\limits_{\xi \in P_{j}}\Big|T(f\chi_{3Q_0\backslash 3P_j})(\xi)\Big|\leq \alpha_n C_{T} \|f\|_{A,3Q_0}.
\end{equation}
Using the same argument as above, we know that for $i=1,2$,
$$\mathop{\mathrm{ess\,sup}}\limits_{\xi \in P_{j}}\Big|T\big(|b_i-\langle b_i\rangle_{R_{Q_0}}|f\chi_{3Q_0\backslash 3P_j}\big)(\xi)\Big|\leq \alpha_n C_{T} \|(b_i-\langle b_i\rangle_{R_{Q_0}})f\|_{A,3Q_0};$$
$$\mathop{\mathrm{ess\,sup}}\limits_{\xi \in P_{j}}\Big|T\Big(\prod_{i=1}^2|b_i-\langle b_i\rangle_{R_{Q_0}}|f\chi_{3Q_0\backslash 3P_j}\Big)(\xi)\Big|\leq \alpha_n C_{T} \Big\|\prod_{i=1}^2(b_i-\langle b_i\rangle_{R_{Q_0}})f\Big\|_{A,3Q_0}.$$
These inequalities, together with \eqref{ie9}, yield that
\begin{equation}\label{ie5.3}
\begin{aligned}
\sum_{j}&|T_{\vec{b}}(f\chi_{3Q_0 \backslash 3P_j})(x)|\chi_{P_j}(x)\\
&\lesssim \Big(\Big\|\prod_{i=1}^2(b_i-\langle b_i \rangle_{R_{Q_0}})f\Big\|_{A,3Q_0}+\big|b_1(x)-\langle b_1 \rangle_{R_{Q_0}}\big|\big\|(b_2-\langle b_2 \rangle_{R_{Q_0}})f\big\|_{A,3Q_0}\\
&+\big|b_2(x)-\langle b_2 \rangle_{R_{Q_0}}\big|\big\|(b_1-\langle b_1 \rangle_{R_{Q_0}})f\big\|_{A,3Q_0}
+\prod_{i=1}^2\big|b_i(x)-\langle b_i \rangle_{R_{Q_0}}\big|\|f\|_{A,3Q_0}\Big)\chi_{Q_0}(x).
\end{aligned}
\end{equation}

On the other hand, by Lemma \ref{lem5.2} and Lemma \ref{lem5.1} (1), it holds that
$$
  \big|T(f\chi_{3Q_0})(x)\big|\leq C_{n,T}|f(x)|+\mathcal{M}_{T,Q_0}f(x),
$$
\begin{align*}
\Big|T\Big(\prod_{i=1}^2&(b_i-\langle b_i \rangle_{R_{Q_0}})f\chi_{3Q_0}\Big)(x)\Big|\\
&\leq C_{n,T}\prod_{i=1}^2\big|b_i(x)-\langle b_i \rangle_{R_{Q_0}}\big||f(x)|
+\mathcal{M}_{T,Q_0}\Big(\prod_{i=1}^2|b_i-\langle b_i \rangle_{R_{Q_0}}|f\Big)(x),
\end{align*}
and for $i=1,2$,
 $$
  \big|T\big((b_i-\langle b_i \rangle_{R_{Q_0}})f\chi_{3Q_0}\big)(x)\big|\leq C_{n,T}\big|b_i(x)-\langle b_i \rangle_{R_{Q_0}}\big||f(x)|+\mathcal{M}_{T,Q_0}\big(|b_i-\langle b_i \rangle_{R_{Q_0}}|f\big)(x).
$$
Furthermore, noting $|E \backslash \bigcup_{j} P_{j}|=0$, then a simple calculation gives
\begin{equation}\label{ie5.4}
\begin{aligned}
|T_{\vec{b}}&(f\chi_{3Q_0})(x)|\chi_{{Q_0}\backslash \cup_jP_j}(x) \\
&\leq C_{n,T}\Big(\Big\|\prod_{i=1}^2(b_i-\langle b_i \rangle_{R_{Q_0}})f\Big\|_{A,3Q_0}+\big|b_1(x)-\langle b_1 \rangle_{R_{Q_0}}\big|\big\|(b_2-\langle b_2 \rangle_{R_{Q_0}})f\big\|_{A,3Q_0}\\
&+\big|b_2(x)-\langle b_2 \rangle_{R_{Q_0}}\big|\big\|(b_1-\langle b_1 \rangle_{R_{Q_0}})f\big\|_{A,3Q_0}
+\prod_{i=1}^2\big|b_i(x)-\langle b_i \rangle_{R_{Q_0}}\big|\|f\|_{A,3Q_0}\Big)\chi_{Q_0}(x).
\end{aligned}
\end{equation}
Combining \eqref{ie5.3} and \eqref{ie5.4}, we complete the proof of Theorem \ref{thm5.1}.
\end{proof}

\section{ Proofs of Theorem \ref{thm5.2}}\label{sec4}

We need to introduce some notations. Let $\mathcal{D}$ be a dyadic lattice and $k \in \mathbb{N}$, denote
$$
\mathcal{F}_k:=\{Q \in \mathcal{D}: 4^{-k-1}<\|f\|_{A, Q} \leq 4^{-k}\},
$$
where $A$ is a Young function satisfying
\begin{equation}\label{ie5.5}
A(4t)\leq \Lambda_{A}A(t),\quad  t>0,\ \Lambda_{A} \geq 1.
\end{equation}

The next lemma (see \cite[Lemma 4.3]{ler5}) plays a key role in the proof of Theorem \ref{thm5.2} , which builds a bridge between sparse families and Young functions.

\begin{lemma}[\cite{ler5}]\label{lem5.3}
Suppose that the family $\mathcal{F}_k$ is $(1-\frac{1}{2 \Lambda_{A}})$-sparse. Let $w$ be a weight and $E$ be an arbitrary measurable set with $w(E)<\infty$. Then for every Young function $\varphi$,
$$
\int_E\left(\sum_{Q \in \mathcal{F}_k} \chi_Q(x)\right) w(x) dx \leq 2^k w(E)+\frac{4 \Lambda_{A}}{\bar{\varphi}^{-1}\left(\left(2 \Lambda_{A}\right)^2\right)} \int_{\mathbb{R}^n} A\left(4^k|f(x)|\right) M_{\varphi}w(x)dx.
$$
\end{lemma}

Now we are in the position to demonstrate Theorem \ref{thm5.2}.
\begin{proof}[Proof of Theorem \ref{thm5.2}]
Similar to Theorem $\ref{thm5.1}$, we only discuss the case $m=2$.
Let $$A=c_n+\int_{1}^{\infty}\frac{\varphi^{-1}_{(2,2)}(t)A_2(\log ^2(e+t))}{t^2\log^3(e+t)}dt,B=c_n+\beta_{1,0}+\int_{1}^{\infty}\frac{\varphi^{-1}_{(1,0)}\circ \Phi^{-1}_{1}(t)A_{1}(\log^{3}(e+t))}{t^2\log^3(e+t)}dt,$$
$$ C=c_n+\beta_{2,0}+\int_{1}^{\infty}\frac{\varphi^{-1}_{(2,0)}\circ \Phi^{-1}_{2}(t)A_{0}(\log^{5}(e+t))}{t^2\log^4(e+t)}dt\ $$and $$ D=c_n+\beta_{2,1}+\int_{1}^{\infty}\frac{\varphi^{-1}_{(2,1)}\circ \Phi^{-1}_{1}(t)A_{0}(\log^{5}(e+t))}{t^2\log^4(e+t)}dt.$$
 It suffices to show that for any $\lambda>0$, the following estimate holds
 \begin{equation}\label{ie5.6}
 \begin{aligned}
 w\big(\{x\in \mathbb{R}^n:|T_{\vec b}f(x)|>\lambda\}\big)
 &\lesssim \quad A\int_{\mathbb{R}^n}A_2\Big(\frac{|f(x)|}{\lambda}\Big)M_{\varphi_{(2,2)}}w(x)dx\\
 &~\quad+B\int_{\mathbb{R}^n}A_1
 \Big(\frac{|f(x)|}{\lambda}\Big)M_{\Phi_1\circ \varphi_{(1,0)}}w(x)dx\\
  &~\quad+C\int_{\mathbb{R}^n}A_0
  \Big(\frac{|f(x)|}{\lambda}\Big)M_{\Phi_2\circ \varphi_{(2,0)}}w(x)dx\\
  &~\quad+D\int_{\mathbb{R}^n}A_0
  \Big(\frac{|f(x)|}{\lambda}\Big)M_{\Phi_1\circ \varphi_{(2,1)}}w(x)dx,
 \end{aligned}
 \end{equation}
where $\beta_{1,0}$, $\beta_{2,0}$ and $\beta_{2,1}$ are determined constants.
By Theorem \ref{thm5.1}, it suffices to verify \eqref{ie5.6} holds for every sparse operators $\mathcal{A}^{\vec \gamma}_{A_0,\mathcal S}$, $\vec \gamma\in \{0,1\}^2$. We consider three cases respectively.\\
\textbf{Case 1. The estimate for $\mathcal{A}^{(0,0)}_{A_0,\mathcal S}(\vec{b},f)$.} The facts  $A^{-1}_{i}(t)\bar{A}^{-1}_{i-1}(t)C^{-1}(t)\leq ct,\,\|b_i\|_{\mathrm {BMO}}=1,\,i=1,2$, together with the generalized H\"{o}lder's inequality, yield that
$$\mathcal{A}^{(0,0)}_{A_0,\mathcal S}(\vec{b},f)(x)\lesssim \sum_{Q\in \mathcal{S}}\|f\|_{A_2,Q}\chi_{Q}(x)=:\mathcal{A}_{A_2}f(x).$$
Let
$$E_{(0,0)}:=\left\{x\in \mathbb{R}^n:\ \mathcal{A}_{A_2}f(x)>4,\,M_{A_2}f(x)\leq \frac{1}{4}\right\}.$$
If one can prove
\begin{equation}\label{ie5.7}
w\big(E_{(0,0)}\big)\lesssim \int_{1}^{\infty}\frac{\varphi^{-1}_{(2,2)}(t)A_2(\log ^2(e+t))}{t^2\log^3(e+t)}dt\int_{\mathbb{R}^n}A_2(|f(x)|)M_{\varphi_{(2,2)}}w(x)dx,
\end{equation}
Then, using homogeneity and Lemma \ref{lem5.4}, the endpoint estimate for $\mathcal{A}^{(0,0)}_{A_0,\mathcal S}(\vec{b},f)$ follows immediately.

Let us now turn to the proof of \eqref{ie5.7}. Denote 
$$\mathcal{S}_{k,A_2}:=\big\{Q\in \mathcal{S}:\ 4^{-k-1}<\|f\|_{A_2, Q} \leq 4^{-k} \big\},$$
and
$$\mathcal{A}_{k,A_2}f(x):=\sum_{Q\in \mathcal{S}_{k,A_2}}\|f\|_{A_2,Q}\chi_{Q}(x).$$
If $Q\in \mathcal{S}$ and $E_{(0,0)}\bigcap Q\neq \emptyset$, then $\|f\|_{A_2,Q}\leq \frac{1}{4}$, which indicates that
$$\mathcal{A}_{A_2}f(x)\leq \sum_{k=1}^{\infty}\mathcal{A}_{k,A_2}f(x),\quad x\in E_{(0,0)}.$$
The submultiplicability of $A_2$ implies that \eqref{ie5.5} holds with $\Lambda_{A_2}=A_2(4)$. Combining Lemma \ref{lem5.3} (taking $\mathcal{F}_k=\mathcal{S}_{k,A_2}$) and the fact that $\mathcal{A}_{k,A_2}f(x)\leq 4^{-k}\sum_{Q\in \mathcal{S}_{k,A_2}}\chi_{Q}(x)$ gives
$$
\int_{E_{(0,0)}}\mathcal{A}_{k,A_2}f(x)w(x)dx\leq \frac{w\big(E_{(0,0)}\big)}{2^{k}}+\frac{c\cdot4^{-k+1}A_2(4^k)}{\bar{\varphi}^{-1}_{(2,2)}
\Big((2\Lambda_{A_{2}})^{2^k}\Big)}\int_{\mathbb{R}^n}
A_2(|f(x)|)M_{\varphi_{(2,2)}}w(x)dx.
$$
With this estimate in hand, we obtain
\begin{equation}\label{ie5.8}
\begin{aligned}
  w\big(E_{(0,0)}\big)&\leq \frac{1}{4}\int_{E_{(0,0)}}\mathcal{A}_{A_2}f(x)w(x)dx\leq \frac{1}{4}\sum_{k=1}^{\infty}\int_{E_{(0,0)}}\mathcal{A}_{k,A_2}f(x)w(x)dx\\
  &\leq\frac{1}{4}w\big(E_{(0,0)}\big)+c\sum_{k=1}^{\infty}\frac{4^{-k}A_2(4^k)}{\bar{\varphi}^{-1}_{(2,2)}(2^{2^k})}
  \int_{\mathbb{R}^n}
A_2(|f(x)|)M_{\varphi_{(2,2)}}w(x)dx.
\end{aligned}
\end{equation}
Observing that
\begin{equation}\label{ie5.9}
\int_{2^{2^{k-1}}}^{2^{2^k}}\frac{1}{t\log(e+t)}dt\geq c,
\end{equation}
and the function
$\frac{A(t)}{t}$ is non-decreasing, one may obtain
\begin{align*}
\sum_{k=1}^{\infty}\frac{4^{-k}A_2(4^k)}{\bar{\varphi}^{-1}_{(2,2)}(2^{2^k})}&\lesssim
\sum_{k=1}^{\infty}\frac{4^{-k}A_2(4^k)}{\bar{\varphi}^{-1}_{(2,2)}(2^{2^k})}\int_{2^{2^{k-1}}}^{2^{2^k}}\frac{1}{t\log(e+t)}dt
\\
&\lesssim \sum_{k=1}^{\infty}\frac{A_2(4^{k-1})}{4^{k-1}}\int_{2^{2^{k-1}}}^{2^{2^k}}\frac{1}{t\bar{\varphi}^{-1}_{(2,2)}(t)
\log(e+t)}dt\\
&\lesssim \sum_{k=1}^{\infty}\int_{2^{2^{k-1}}}^{2^{2^k}}\frac{A_2(\log^2(e+t))}{t\bar{\varphi}^{-1}_{(2,2)}(t)
\log(e+t)\log^2(e+t)}dt\\
&\lesssim
\int_{1}^{\infty}\frac{\varphi^{-1}_{(2,2)}(t)A_2(\log^2(e+t))}{t^2\log^3(e+t)}dt.
\end{align*}
Plugging this estimate into \eqref{ie5.8}, it follows the desired conclusion \eqref{ie5.7}.

\textbf{Case 2. The estimate for $\mathcal{A}^{(1,0)}_{A_0,\mathcal S}(\vec{b},f)$.}
The condition $A^{-1}_{1}(t)\bar{A}^{-1}_{0}(t){C}^{-1}(t)\leq ct$ and the generalized H\"{o}lder's inequality indicate that
$$\mathcal{A}^{(1,0)}_{A_0,\mathcal S}(\vec{b},f)(x)\lesssim \sum_{Q\in \mathcal{S}}|b_1(x)-\langle b_1\rangle_Q|\|f\|_{A_1,Q}\chi_{Q}(x)=:\mathcal{A}^{(1,0)}_{A_1}f(x).$$
Define
$$E_{(1,0)}:=\left\{x\in \mathbb{R}^n:\ \mathcal{A}^{(1,0)}_{A_1}f(x)>8,\,M_{A_1}f(x)\leq \frac{1}{4}\right\}.$$
Using homogeneity and Lemma \ref{lem5.4}, it suffices to show that
\begin{equation}\label{ie5.10}
\begin{aligned}
w\big(&E_{(1,0)}\big)\\
&\lesssim\Big(\beta_{1,0}+\int_{1}^{\infty}\frac{\varphi^{-1}_{(1,0)}\circ \Phi^{-1}_{1}(t)A_{1}(\log^{3}(e+t))}{t^2\log^3(e+t)}dt\Big)
\int_{\mathbb{R}^n}A_1(|f(x)|)M_{\Phi_1\circ \varphi_{(1,0)}}w(x)dx.
\end{aligned}
\end{equation}
Let
$$\mathcal{S}_{k,A_1}:=\big\{Q\in \mathcal{S}:\ 4^{-k-1}<\|f\|_{A_1, Q} \leq 4^{-k} \big\},$$
and for each $Q\in \mathcal{S}_{k,A_1}$, define $\mathcal{H}^{0}_k(Q)$ as follows
\begin{equation}\label{ie13}
\mathcal{H}^{0}_k(Q):=\Big\{x\in Q:\ |b_1(x)-\langle b_1 \rangle_{Q}|>\Big(\frac{8}{5}\Big)^k\Big\}.
\end{equation}
If $Q\in \mathcal{S}$ and $E_{(1,0)}\bigcap Q\neq \emptyset$, then $\|f\|_{A_1,Q}\leq \frac{1}{4}$. This shows that for any $x\in E_{(1,0)}$,
\begin{align*}
\mathcal{A}^{(1,0)}_{A_1}f(x)&\leq \sum_{k=1}^{\infty}\sum_{Q\in \mathcal{S}_{k,A_1}}|b_1(x)-\langle b_1\rangle_{Q}|\|f\|_{A_1,Q}\chi_{\mathcal{H}^{0}_k(Q)}(x)\\
&\quad +\sum_{k=1}^{\infty}\Big(\frac{8}{5}\Big)^k\sum_{Q\in \mathcal{S}_{k,A_1}}\|f\|_{A_1,Q}\chi_Q(x)\\
&=:\mathcal{A}^{0}_{A_1}f(x)+\mathcal{A}^{1}_{A_1}f(x).
\end{align*}
Let
$$E^{0}_{(1,0)}:=\{x\in E_{(1,0)}:\ \mathcal{A}^{0}_{A_1}f(x)>3\},\,\,E^{1}_{(1,0)}:=\{x\in E_{(1,0)}:\ \mathcal{A}^{1}_{A_1}f(x)>5\},$$
then $w\big(E_{(1,0)}\big)\leq w\big(E^{0}_{(1,0)}\big)+w\big(E^{1}_{(1,0)}\big).$
It follows from the proof of \cite[Lemma 4.3]{ler5} that there exist pairwise disjoint $E_Q\subset Q \in \mathcal{S}_{k,A_1}$ such that
\begin{equation}\label{ie5.11}
1\leq\frac{1}{|Q|}\int_{E_Q}A_{1}(4^k|f(x)|)dx,
\end{equation}
which yields that
\begin{equation}\label{ie5.12}
\begin{aligned}
w\big(E^{0}_{(1,0)}\big)&\leq \frac{1}{3}\int_{E^{0}_{(1,0)}}\mathcal{A}^{0}_{A_1}f(x)w(x)dx\\
&\lesssim\sum_{k=1}^{\infty}\sum_{Q\in \mathcal{S}_{k,A_1}}\frac{1}{4^k}\Big(\frac{1}{|Q|}\int_{\mathcal{H}^{0}_k(Q)}|b_1(x)-\langle b_1\rangle_Q|w(x)dx\Big)\Big(\int_{E_Q}A_{1}(4^k|f(x)|)dx\Big).
\end{aligned}
\end{equation}
By the generalized H\"{o}lder's inequality, it holds that
\begin{equation}\label{ie5.13}
\frac{1}{|Q|}\int_{\mathcal{H}^{0}_k(Q)}|b_1(x)-\langle b_1\rangle_Q|w(x)dx\lesssim\|w\chi_{\mathcal{H}^{0}_k(Q)}\|_{L(\log L),Q}.
\end{equation}
Define 
$\Psi^{-1}_{(1,0)}(t):=\frac{\Phi^{-1}_{1}(t)}{\varphi^{-1}_{(1,0)}\circ \Phi^{-1}_{1}(t)}.$
Note that $\Phi_1$ and $\varphi_{(1,0)}(t)/t$ are strictly increasing functions, then $\Psi_{(1,0)}$ is strictly increasing. The generalized H\"{o}lder's inequality implies that
\begin{equation}\label{ie5.14}
\begin{aligned}
\|w\chi_{\mathcal{H}^{0}_k(Q)}\|_{L(\log L),Q}&\lesssim \|\chi_{\mathcal{H}^{0}_k(Q)}\|_{\Psi_{(1,0)},Q}\|w\|_{\Phi_1\circ \varphi_{(1,0)},Q}\\
&=\frac{c}{\Psi^{-1}_{(1,0)}(|Q|/|\mathcal{H}^{0}_k(Q)|)}\|w\|_{\Phi_1\circ \varphi_{(1,0)},Q}.
\end{aligned}
\end{equation}
On the other hand, let $\alpha_{1,k}=\min\big\{1,\,e^{-{\frac{(8/5)^k}{2^ne}}+1}\big\}$, then the John-Nirenberg inequality ensures that $|\mathcal{H}^{0}_k(Q)|\leq \alpha_{1,k}|Q|$, which together with \eqref{ie5.12}-\eqref{ie5.14}, yields that
\begin{equation}\label{ie5.15}
\begin{aligned}
w\big(E^{0}_{(1,0)}\big)&\lesssim\sum_{k=1}^{\infty}\frac{1}{\Psi^{-1}_{(1,0)}(1/\alpha_{1,k})4^{k}}\sum_{Q \in \mathcal{S}_{k,A_1}}\|w\|_{\Phi_1\circ \varphi_{(1,0)},Q}\int_{E_Q}A_{1}(4^k|f(x)|)dx\\
&\lesssim\sum_{k=1}^{\infty}\frac{1}{\Psi^{-1}_{(1,0)}(1/\alpha_{1,k})}\frac{A_1(4^k)}{4^k}\int_{\mathbb{R}^n}A_1(|f(x)|)
M_{\Phi_{1}\circ\varphi_{(1,0)}}w(x)dx.
\end{aligned}
\end{equation}
In order to obtain a more precise estimate of the constant $\sum_{k=1}^{\infty}\frac{1}{\Psi^{-1}_{(1,0)}(1/\alpha_{1,k})}$, 
we choose positive integer $a_{n,1}$ such that for every $k>a_{n,1}$, $\frac{1}{\alpha_{1,k-1}}=e^{\frac{(8/5)^{k-1}}{2^ne}-1}\geq e^2$. Let $\frac{1}{\varrho_1}=\frac{\log 4}{\log(8/5)}$, it is easy to verify 
\begin{equation*}
 \int_{\frac{1}{\alpha_{1,k-1}}}^{\frac{1}{\alpha_{1,k}}}\frac{1}{t\log(e+t)}dt\geq c.
\end{equation*}
This, together with the submultiplicativity of $A_1$, the monotonicity of $A_1(t)/t$ and the fact $\frac{1}{\varrho_1}\leq 3$, yields that
\begin{equation}\label{ie5.16}
\begin{aligned}
\sum_{k=1}^{\infty}\frac{1}{\Psi^{-1}_{(1,0)}(1/\alpha_{1,k})}\frac{A_1(4^k)}{4^k}&\leq \sum_{k=1}^{a_{n,1}}\frac{A_1(4^k)\varphi^{-1}_{(1,0)}\circ \Phi^{-1}_{1}(1/\alpha_{1,k})}{4^k\Phi^{-1}_{1}(1/\alpha_{1,k})}+\sum_{k=a_{n,1}}^{\infty}
\frac{A_1(4^k)}{4^k\Psi^{-1}_{(1,0)}(1/\alpha_{1,k})}\\
&\lesssim
\beta_{1,0}+\sum_{k=a_{n,1}}^{\infty}\frac{A_1(4^{k-1})}{4^{k-1}\Psi^{-1}_{(1,0)}(1/\alpha_{1,k})}\int_{\frac{1}{\alpha_{1,k-1}}}^{\frac{1}{\alpha_{1,k}}}
\frac{1}{t\log(e+t)}dt\\
&\lesssim \beta_{1,0}+\int_{\frac{1}{\alpha_{1,k-1}}}^{\frac{1}{\alpha_{1,k}}}
\frac{A_1(\log^{1/\varrho_1}(e+t))}{\Psi^{-1}_{(1,0)}(t)t\log(e+t)\log^{1/\varrho_1}(e+t)}dt\\
&\lesssim
\beta_{1,0}+\int_{1}^{\infty}\frac{\varphi^{-1}_{(1,0)}\circ \Phi^{-1}_{1}(t)A_1(\log^{3}(e+t))}{\Phi^{-1}_{1}(t)t\log^4(e+t)}dt\\
&\lesssim
\beta_{1,0}+\int_{1}^{\infty}\frac{\varphi^{-1}_{(1,0)}\circ \Phi^{-1}_{1}(t)A_1(\log^{3}(e+t))}{t^2\log^3(e+t)}dt,
\end{aligned}
\end{equation}
where $\beta_{1,0}=\sum_{k=1}^{a_{n,1}}\frac{A_1(4^k)\varphi^{-1}_{(1,0)}\circ \Phi^{-1}_{1}(1/\alpha_{1,k})}{4^k\Phi^{-1}_{1}(1/\alpha_{1,k})}$. 

In order to obtain the weak type endpoint estimate \eqref{ie5.10} of $w\big(E_{(1,0)}\big)$, it remains to establish the estimate of $w\big(E^{1}_{(1,0)}\big)$. Applying Lemma \ref{lem5.3} with the Young function $\psi_{(1,0)}=\Phi_1\circ \varphi_{(1,0)}$, it may lead to
\begin{equation*}
\begin{aligned}
 \int_{E^1_{(1,0)}}\mathcal{A}^{1}_{A_1}f(x)w(x)dx \leq & \sum_{k=1}^{\infty}\Big(\frac{4}{5}\Big)^kw(E^{1}_{(1,0)})\\
 &+c\sum_{k=1}^{\infty}\Big(\frac{8}{5}\Big)^k
 \frac{A_1(4^k)}{\bar{\psi}^{-1}_{(1,0)}(2^{2^k})4^{k}}
 \int_{\mathbb{R}^n}A_1(|f(x)|)M_{\psi_{(1,0)}}w(x)dx.
\end{aligned}
\end{equation*}
Then
\begin{align*}
w\big(E^{1}_{(1,0)}\big)&\leq \frac{1}{5}\int_{E^1_{(1,0)}}\mathcal{A}^{1}_{A_1}f(x)w(x)dx\\
&\leq \frac{4}{5}w\big(E^{1}_{(1,0)}\big)+c\sum_{k=1}^{\infty}\Big(\frac{8}{5}\Big)^k
\frac{A_1(4^k)}{\bar{\psi}^{-1}_{(1,0)}(2^{2^k})4^{k}}
\int_{\mathbb{R}^n}A_1(|f(x)|)M_{\psi_{(1,0)}}w(x)dx.
\end{align*}
Furthermore, by \eqref{ie5.9}, one can get
\begin{align*}
 \sum_{k=1}^{\infty}\Big(\frac{8}{5}\Big)^k\frac{A_1(4^k)}{\bar{\psi}^{-1}_{(1,0)}(2^{2^k})4^{k}}&\lesssim \sum_{k=1}^{\infty}2^k\frac{A_1(4^k)}{\bar{\psi}^{-1}_{(1,0)}(2^{2^k})4^{k}}\int_{2^{2^{k-1}}}^{2^{2^k}}\frac{1}{t\log(e+t)}dt\\  &\lesssim\sum_{k=1}^{\infty}\frac{A_1(4^k)}{\bar{\psi}^{-1}_{(1,0)}(2^{2^k})4^{k}}\int_{2^{2^{k-1}}}^{2^{2^k}}
 \frac{1}{t}dt\\
 &\lesssim\int_{1}^{\infty}\frac{\varphi^{-1}_{(1,0)}\circ \Phi^{-1}_{1}(t)A_{1}(\log^{3}(e+t))}{t^2\log^3(e+t)}dt.
\end{align*}
Therefore, we finish the proof of \eqref{ie5.10}.

Note that $\mathcal{A} ^{(0,1)}_{A_0,\mathcal{S}}(\vec b,f)=\mathcal{A} ^{(0,1)}_{A_0,\mathcal{S}}((b_1,b_2),f)=\mathcal{A}^{(1,0)}_{A_0,\mathcal{S}}((b_2,b_1),f)$, thus the endpoint estimate \eqref{ie5.10} also holds for $\mathcal{A}^{(0,1)}_{A_0,\mathcal{S}}(\vec b,f)$.

\textbf{Case 3. The estimate for $\mathcal{A}^{(1,1)}_{A_0,\mathcal{S}}(\vec b,f)$.}
To estimate $\mathcal{A}^{(1,1)}_{A_0,\mathcal{S}}(\vec b,f)$, we need to give some notations.
Let
\begin{align*}
E_{(1,1)}&:=\Big\{x\in \mathbb{R}^n:\ \mathcal{A}^{(1,1)}_{A_0,\mathcal{S}}(\vec b,f)(x)>40,\,M_{A_0}f(x)\leq \frac{1}{4}\Big\},\\
\mathcal{S}_{k,A_0}&:=\big\{Q\in \mathcal{S}:\ 4^{-k-1}<\|f\|_{A_0, Q} \leq 4^{-k} \big\},
\end{align*}
and for each $Q\in \mathcal{S}_{k,A_0}$ and $i=1,2$, denote
\begin{equation}\label{ie14}
\mathcal{H}^{i}_k(Q):=\Big\{x\in Q:\ |b_i(x)-\langle b_i \rangle_{Q}|>\Big(\frac{4}{3}\Big)^k\Big\}.
\end{equation}
According to the weak inequality \eqref{ie12} of Orlicz maximal operators, we only need to estimate $w\big(E_{(1,1)}\big)$.
Note that if $Q\in \mathcal{S}$ and $E_{(1,1)}\bigcap Q\neq \emptyset$, then $\|f\|_{A_0,Q}\leq \frac{1}{4}$. For every $x\in E_{(1,1)}$, it follows that 
\begin{equation}\label{ie5.17}
\begin{aligned}
\mathcal{A}^{(1,1)}_{A_0,\mathcal{S}}(\vec b,f)(x)&\leq\sum_{k=1}^{\infty}\sum_{Q\in \mathcal{S}_{k,A_0}}|b_1(x)-\langle b_1\rangle_{Q}||b_2(x)-\langle b_2\rangle_{Q}|\|f\|_{A_0,Q}\chi_{Q}(x)\\
&\leq\sum_{i=1}^2\sum_{k=1}^{\infty}\Big(\frac{4}{3}\Big)^k\sum_{Q\in \mathcal{S}_{k,A_0}}|b_i(x)-\langle b_i\rangle_{Q}|\|f\|_{A_0,Q}\chi_{\mathcal{H}^{i}_{k}(Q)}(x)\\
&\quad+\sum_{k=1}^{\infty}\sum_{Q\in \mathcal{S}_{k,A_0}}\prod_{i=1}^2|b_i(x)-\langle b_i\rangle_{Q}|\|f\|_{A_0,Q}\chi_{\mathcal{H}^1_{k}\cap\mathcal{H}^2_{k}(Q)}(x)\\
&\quad+\sum_{k=1}^{\infty}\Big(\frac{4}{3}\Big)^{2k}\sum_{Q\in \mathcal{S}_{k,A_0}}\|f\|_{A_0,Q}\chi_{Q}(x)=:\sum_{i=1}^4\mathcal{A}^{i}_{A_0}f(x).
\end{aligned}
\end{equation}
Therfore, $w\big(E_{(1,1)}\big)\leq \sum_{i=1}^4w\big(E^{i}_{(1,1)}\big)$, where
$$E^{i}_{(1,1)}=\Big\{x\in E_{(1,1)}:\ \mathcal{A}^i_{A_0}f(x)>10\Big\},\,\, i=1,2,3,4.$$
Having performed the above decomposition of $w\big(E_{(1,1)}\big)$, our focus now shifts to estimating each $w\big(E^i_{(1,1)}\big).$
Since the estimates for $w\big(E^1_{(1,1)}\big)$ and $w\big(E^2_{(1,1)}\big)$ are analogous, we only give the estimate for $w\big(E^1_{(1,1)}\big)$. In fact, \eqref{ie5.11} indicates that
\begin{align*}
 w\big(E^1_{(1,1)}\big)&\leq \frac{1}{10}\int_{E^1_{(1,1)}}\sum_{k=1}^{\infty}\Big(\frac{4}{3}\Big)^{k}\sum_{Q\in \mathcal{S}_{k,A_0}}|b_1(x)-\langle b_1\rangle_{Q}|\|f\|_{A_0,Q}\chi_{\mathcal{H}^1_{k}(Q)}(x)w(x)dx\\
  &\lesssim  \sum_{k=1}^{\infty}\sum_{Q\in \mathcal{S}_{k,A_0}}\frac{1}{3^k}\Big(\frac{1}{|Q|}\int_{\mathcal{H}^1_{k}(Q)}|b_1(x)-\langle b_1\rangle_{Q}|w(x)dx\Big)\Big(\int_{E_Q}A_0(4^k|f(x)|)dx\Big).
\end{align*}
Using the same argument as what we have done in the estimate \eqref{ie5.15} for $E^{0}_{(1,0)}$ in $\mathcal{A}^{(1,0)}_{A_0,\mathcal S}(\vec{b},f)$, we obtain
$$w\big(E^1_{(1,1)}\big)\lesssim \sum_{k=1}^{\infty }\frac{1}{\Psi^{-1}_{(2,1)}(1/\alpha_{2,k})}\frac{A_0(4^k)}{3^k}\int_{\mathbb{R}^n}A_0(|f(x)|)
M_{\Phi_{1}\circ\varphi_{(2,1)}}w(x)dx,$$
where $\alpha_{2,k}=\min\{1,e^{-\frac{(4/3)^k}{2^ne}+1}\}$ and 
$\Psi^{-1}_{(2,1)}(t):=\frac{\Phi^{-1}_{2}(t)}{\varphi^{-1}_{(2,1)}\circ \Phi^{-1}_{1}(t)}.$

In the following, we will undertake a more precise estimate for $\sum_{k=1}^{\infty }\frac{1}{\Psi^{-1}_{(2,1)}(1/\alpha_{2,k})}\frac{A_0(4^k)}{3^k}$.
Picking positive integer $a_{n,2}$ such that for $k>a_{n,2}$, $\frac{1}{\alpha_{2,k-1}}=e^{\frac{(4/3)^{k-1}}{2^ne}-1}\geq e^2$. If $\frac{1}{\varrho_2}=\frac{\log 4}{\log(4/3)}$, a simple calculation yields that
\begin{equation}\label{ie5.18}
\begin{aligned}
\sum_{k=1}^{\infty }\frac{1}{\Psi^{-1}_{(2,1)}(1/\alpha_{2,k})}\frac{A_0(4^k)}{3^k}&\leq \sum_{k=1}^{a_{n,2}}\frac{(\frac{4}{3})^kA_0(4^k)\varphi^{-1}_{(2,1)}\circ \Phi^{-1}_{1}(1/\alpha_{2,k})}{4^k\Phi^{-1}_{1}(1/\alpha_{2,k})}+\sum_{k=a_{n,2}}^{\infty}
\frac{(\frac{4}{3})^kA_0(4^k)4^{-k}}{\Psi^{-1}_{(2,1)}(1/\alpha_{2,k})} \\
&\lesssim \beta_{2,1}+\sum_{k=a_{n,2}}^{\infty}\frac{A_0(4^{k-1})}{4^{k-1}}\Big(\frac{4}{3}\Big)^k\int_{\frac{1}{\alpha_{2,k-1}}}^{\frac{1}{\alpha_{2,k}}}
\frac{1}{\Psi^{-1}_{(2,1)}(t)t\log(e+t)}dt
\\
&\lesssim \beta_{2,1}+\int_{1}^{\infty}\frac{A_0(\log^{1/\varrho_2}(e+t))}
{t\Psi^{-1}_{(2,1)}(t)\log^{1/\varrho_2}(e+t)}dt\\
&\lesssim
\beta_{2,1}+\int_{1}^{\infty}\frac{\varphi^{-1}_{(2,1)}\circ \Phi^{-1}_{1}(t)A_0(\log^{5}(e+t))}{\Phi^{-1}_{1}(t)t\log^5(e+t)}dt\\
&\lesssim
\beta_{2,1}+\int_{1}^{\infty}\frac{\varphi^{-1}_{(2,1)}\circ \Phi^{-1}_{1}(t)A_0(\log^{5}(e+t))}{t^2\log^4(e+t)}dt,
\end{aligned}
\end{equation}
where $\beta_{2,1}=\sum_{k=1}^{a_{n,2}}\frac{(\frac{4}{3})^kA_0(4^k)\varphi^{-1}_{(2,1)}\circ \Phi^{-1}_{1}(1/\alpha_{2,k})}{4^k\Phi^{-1}_{1}(1/\alpha_{2,k})}.$

We now turn to estimate $w\big(E^3_{(1,1)}\big)$. In fact,
$$w\big(E^3_{(1,1)}\big)\lesssim\sum_{k=1}^{\infty}\sum_{Q\in \mathcal{S}_{k,A_0}}\frac{1}{4^k}\Big(\frac{1}{|Q|}\int_{\tilde{Q}}
\prod_{i=1}^2|b_i(x)-\langle b_i\rangle_{Q}|w(x)dx\Big)\Big(\int_{E_Q}A_0(4^k|f(x)|)dx\Big),$$
where $\tilde{Q}:=\mathcal{H}^1_{k}(Q)\cap\mathcal{H}^2_{k}(Q).$
Then the generalized H\"{o}lder's inequality gives that
\begin{equation*}
\frac{1}{|Q|}\int_{\mathcal{H}^1_{k}(Q)\cap\mathcal{H}^2_{k}(Q)}\prod_{i=1}^2|b_i(x)-\langle b_i\rangle_{Q}|w(x)dx\lesssim \|w\chi_{\mathcal{H}^1_{k}(Q)\cap\mathcal{H}^2_{k}(Q)}\|_{L(\log L)^2,Q}.
\end{equation*}
Define 
$\Psi^{-1}_{(2,0)}(t):=\frac{\Phi^{-1}_{2}(t)}{\varphi^{-1}_{(2,0)}\circ \Phi^{-1}_{2}(t)}.$
Note that $\Phi_2$ and $\varphi_{(2,0)}(t)/t$ are strictly increasing functions, then $\Psi_{(2,0)}$ is strictly increasing, too. The generalized H\"{o}lder's inequality implies 
\begin{equation*}
\begin{aligned}
\|w\chi_{\mathcal{H}^1_{k}(Q)\cap\mathcal{H}^2_{k}(Q)}\|_{L(\log L)^2,Q}&\lesssim \|\chi_{\mathcal{H}^1_{k}(Q)\cap\mathcal{H}^2_{k}(Q)}\|_{\Psi_{(2,0)},Q}\|w\|_{\Phi_2\circ \varphi_{(2,0)},Q}\\
&=\frac{c}{\Psi^{-1}_{(2,0)}\Big(|Q|/|\mathcal{H}^1_{k}(Q)\cap\mathcal{H}^2_{k}(Q)|\Big)}\|w\|_{\Phi_2\circ \varphi_{(2,0)},Q}.
\end{aligned}
\end{equation*}
The same reasoning as in the estimate for \eqref{ie5.15} and \eqref{ie5.16} then gives
\begin{align*}
w\big(E^{3}_{(1,1)}\big)&\lesssim\sum_{k=1}^{\infty}\frac{1}{\Psi^{-1}_{(2,0)}(1/\alpha_{2,k})4^{k}}\sum_{Q \in \mathcal{S}_{k,A_0}}\|w\|_{\Phi_2\circ \varphi_{(2,0)},Q}\int_{E_Q}A_{0}(4^k|f(x)|)dx\\
&\lesssim\sum_{k=1}^{\infty}\frac{1}{\Psi^{-1}_{(2,0)}(1/\alpha_{2,k})}\frac{A_0(4^k)}{4^k}\int_{\mathbb{R}^n}A_0(|f(x)|)
M_{\Phi_{2}\circ\varphi_{(2,0)}}w(x)dx,
\end{align*}
and
\begin{equation}\label{ie5.19}
\sum_{k=1}^{\infty}\frac{1}{\Psi^{-1}_{(2,0)}(1/\alpha_{2,k})}\frac{A_0(4^k)}{4^k}\lesssim \beta_{2,0}+\int_{1}^{\infty}\frac{\varphi^{-1}_{(2,0)}\circ \Phi^{-1}_2(t)A_{0}(\log^5(e+t))}{t^2\log^4(e+t)}dt,
\end{equation}
where $\beta_{2,0}=\sum_{k=1}^{a_{n,2}}\frac{A_0(4^k)\varphi^{-1}_{(2,0)}\circ \Phi^{-1}_{2}(1/\alpha_{2,k})}{4^k\Phi^{-1}_{2}(1/\alpha_{2,k})}.$

It remains to estimate $w\big(E^4_{(1,1)}\big)$. Using Lemma \ref{lem5.3} with $\varphi=\psi_{(2,0)}:=\Phi_{2}\circ \varphi_{(2,0)}$, we have
\begin{align*}
 \int_{E^4_{(1,1)}}\mathcal{A}^4_{A_0}f(x)w(x)dx&\leq\sum_{k=1}^{\infty}\Big(\frac{8}{9}\Big)^kw\big(E^4_{(1,1)}\big)\\
 &\quad+c\sum_{k=1}^{\infty}\Big(\frac{16}{9}\Big)^k
 \frac{A_0(4^k)}{\bar{\varphi}^{-1}_{(2,0)}\Big((2\Lambda_{A_{0}})^{2^k}\Big)}\int_{\mathbb{R}^n}A_0(|f(x)|)
 M_{\psi_{(2,0)}}w(x)dx.
\end{align*}
By the definition of $E^4_{(1,1)}$, it holds that
\begin{align*}
w\big(E^4_{(1,1)}\big)&\leq \frac{1}{10}\int_{E^4_{(1,1)}}\mathcal{A}^4_{A_0}f(x)w(x)dx\\
&\leq \frac{4}{5}w\big(E^4_{(1,1)}\big)+c\sum_{k=1}^{\infty}\Big(\frac{16}{9}\Big)^k
 \frac{A_0(4^k)}{\bar{\varphi}^{-1}_{(2,0)}(2^{2^k})}\int_{\mathbb{R}^n}A_0(|f(x)|)
 M_{\psi_{(2,0)}}w(x)dx.
\end{align*}
Then, $\eqref{ie5.9}$ implies that
\begin{equation}\label{ie5.20}
\begin{aligned}
\sum_{k=1}^{\infty}\Big(\frac{16}{9}\Big)^k
\frac{A_0(4^k)}{\bar{\varphi}^{-1}_{(2,0)}(2^{2^k})}&\lesssim \sum_{k=1}^{\infty}2^k\frac{A_0(4^k)}{\bar{\psi}^{-1}_{(2,0)}(2^{2^k})4^{k}}\int_{2^{2^{k-1}}}^{2^{2^k}}
\frac{1}{t\log(e+t)}dt\\
& \lesssim\sum_{k=1}^{\infty}\frac{A_0(4^k)}{\bar{\psi}^{-1}_{(2,0)}(2^{2^k})4^{k}}\int_{2^{2^{k-1}}}^{2^{2^k}}
\frac{1}{t}dt\\
 &\lesssim\int_{1}^{\infty}\frac{\psi^{-1}_{(2,0)}(t)A_0(\log^2(e+t))}{t^2\log^2(e+t)}dt.
\end{aligned}
\end{equation}
Therefore, it follows from the estimates \eqref{ie5.18}-\eqref{ie5.20} that  $w\big(E_{(1,1)}\big)$ can be controlled by
\begin{align*}
 {}&\Big(\beta_{2,1}+\int_{1}^{\infty}\frac{\varphi^{-1}_{(2,1)}\circ \Phi^{-1}_{1}(t)A_0(\log^{5}(e+t))}{t^2\log^4(e+t)}dt\Big)\int_{\mathbb{R}^n}A_0(|f(x)|)
M_{\Phi_{1}\circ\varphi_{(2,1)}}w(x)dx\\
&+\Big(\beta_{2,0}+\int_{1}^{\infty}\frac{\varphi^{-1}_{(2,0)}\circ \Phi^{-1}_2(t)A_{0}(\log^5(e+t))}{t^2\log^4(e+t)}dt\Big)\int_{\mathbb{R}^n}A_0(|f(x)|)
M_{\Phi_{2}\circ\varphi_{(2,0)}}w(x)dx,
\end{align*}
which combining with \eqref{ie5.7} and \eqref{ie5.10} completes the proof of Theorem \ref{thm5.2}.
\end{proof}

\begin{remark}
It follows from the proof of Theorem \ref{thm5.2} that the constants 8/5 in \eqref{ie14} and 4/3 in \eqref{ie15} used in the previous proof are not essential. In fact, we can replace them with variables $r_1$ and $r_2$, where $r_1$ satisfies
$4\leq r^3_1 <8$ and $r_2$ satisfies $4\leq r^5_2<4\sqrt{2}$. Additionally, we can derive the precise expression of the constant $\beta_{m,l_1,l_2}$ as follows:
\begin{equation}\label{ie15}
\beta_{m,l_1,l_2}=\sum_{k=1}^{a_{n,l_1}}\frac{r^{l_{2}k}_{l_1}A_{m-l_1}(4^k)\varphi^{-1}_{(l_1,l_2)}
\circ\Phi^{-1}_{l_1-l_2}(1/\alpha_{r_{l_1},k})}{4^k\Phi^{-1}_{l_1-l_2}(1/\alpha_{r_{l_1},k})},
\end{equation}
where $\alpha_{r_{l_1},k}=\min\Big\{1,e^{-\frac{r^{k}_{l_1}}{2^ne}+1}\Big\}$ , $1<r^{2l_1}_{l_1}<4\leq r^{2l_1+1}_{l_1}$, and the constant $a_{n,l_1}$ satisfies when $k>a_{n,l_1}$, $e^{\frac{r^{k-1}_{l_1}}{2^ne}-1}\geq e^2$. It is worth pointing out that \eqref{ie15} is crucial for the exact constant estimate in Theorem \ref{thm5.3}.
\end{remark}

\section{Proofs of Theorem \ref{thm5.3}, Corollaries \ref{cor5.1} and \ref{cor5.2}}\label{sec5}
\begin{proof}[Proof of Theorem \ref{thm5.3}]
If $T$ is an $\omega$-Calder\'{o}n-Zygmund operator and $\omega$ satisifies a Dini condition: $\int_{0}^1\frac{\omega(t)}{t}dt<\infty$, it is easy to verify that the kernel of $T$ satisfies $L^{\infty}$-H\"{o}mander condition, see \cite[Proposition 3.1]{li3}. This means that $T$ is an $\bar{A}_0$-H\"{o}rmander operator, where $A_0(t)=t$. 
Assume that 
 $\|\vec{b}\|_{\mathrm{BMO}}=1$ and let $A_{i}(t)=\Phi_{i}(t)=t\log^{i}(e+t),\,i=1,\ldots,m$, it follows from a simple calculation that $A_{i}$ satisifies the conditions of Theorem \ref{thm5.2}. We need to show that \eqref{ie5.28} holds for the chosen Young function $A_{i}$ and $\Phi_{i}$.

We consider the following three cases.

\textbf{Case I:} $l_1=l_2=m$. In this case, we have
\begin{align*}
  K_{\varphi_{(m,m)}}&=\int_{1}^{\infty}\frac{\varphi^{-1}_{(m,m)}(t)A_m(\log^2(e+t))}{t^2\log^3(e+t)}dt=\int_{1}^{\infty}\frac{\varphi^{-1}_{(m,m)}(t)\log^m(e+\log^2(e+t))}{t^2\log(e+t)}dt.
\end{align*}
For any fixed $0<\varepsilon<1$, let $\varphi_{(m,m)}(t)=t\log^m(e+t)\log^{1+\varepsilon}(e+\log(e+t))$, then
\begin{align*}
K_{\varphi_{(m,m)}}
\lesssim \int_{\log(e+1)}^{\infty}\frac{e^u}{(e^u-e)u\log^{1+\varepsilon}(e+u)}du\lesssim\frac{1}{\varepsilon},
\end{align*}
in the first inequality, we use the fact that $\varphi^{-1}_{(m,m)}(t)\leq ct\log^{-m}(e+t)\log^{-1-\varepsilon}(e+\log(e+t))$, where $c$ is an absolute constant which is independent of $\varepsilon$.

Therefore, when $l_1=l_2=m$, we get
\begin{equation}\label{ie5.21}
\begin{aligned}
K&_{\varphi_{(m,m)}}\int_{\mathbb{R}^n}A_m(\frac{|f(x)|}{\lambda})M_{\varphi_{(m,m)}}w(x)dx\\
&\lesssim \frac{1}{\varepsilon}\int_{\mathbb{R}^n}\Phi_m(\frac{|f(x)|}{\lambda})M_{L(\log L)^m(\log \log L)^{1+\varepsilon}}w(x)dx.
\end{aligned}
\end{equation}

\textbf{Case II:} $1\leq l_1<m$. For this case, the following estimate holds 
\begin{align*}
&\int_{1}^{\infty}\frac{\varphi^{-1}_{(l_1,l_2)}\circ\Phi^{-1}_{l_1-l_2}(t)A_{m-l_1}(\log^{2l_1+1}(e+t))}
{t^2\log^{l_1+2}(e+t)}dt\\
&=\int_{c_0}^{\infty}\frac{\varphi^{-1}_{(l_1,l_2)}(t)\log^{l_1-1}(e+\Phi_{l_1-l_2}(t))\log^{m-l_1}(e+\log^{2l_1+1}
(e+\Phi_{l_1-l_2}(t)))}{\Phi^{2}_{l_1-l_2}(t)}\Phi^{'}_{l_1-l_2}(t)dt\\
&\lesssim \int_{c_0}^{\infty}\frac{\varphi^{-1}_{(l_1,l_2)}(t)
\log^{l_2}(e+t)\log^{m-l_1}(e+\log^{2l_1+1}
(e+t))}{t^2\log(e+t)}dt,
\end{align*}
where $c_0=\Phi^{-1}_{l_1-l_2}(1)$ and $\Phi'$ denotes the derivative of $\Phi$.

Take $\varphi_{(l_1,l_2)}(t)=t\log^{m-l_1+l_2}(e+t)\log^{1+\varepsilon}(e+\log(e+t))$, then
\begin{align*}
\int_{1}^{\infty}&\frac{\varphi^{-1}_{(l_1,l_2)}\circ\Phi^{-1}_{l_1-l_2}(t)A_{m-l_1}(\log^{2l_1+1}(e+t))}
{t^2\log^{l_1+2}(e+t)}dt\\
&\lesssim \int_{c_0}^{\infty}\frac{\log^{l_2}(e+t)\log^{m-l_1}(e+\log^{2l_1+1}
(e+t))}{t\log^{m-l_1+l_2+1}(e+t)\log^{1+\varepsilon}(e+\log(e+t))}dt\\
&\lesssim \int_{c_0}^{\infty}\frac{1}{t\log(e+t)\log^{1+\varepsilon}(e+\log(e+t))}dt\\
&\lesssim \frac{1}{\varepsilon}.
\end{align*}
Combining $\Phi_{l_1-l_2}\circ\varphi_{(l_1,l_2)}(t)\lesssim \varphi_{(l_1,l_2)}(t) \log^{l_1-l_2}(e+t)=t\log^m(e+t)\log^{1+\varepsilon}(e+\log(e+t))$ with the fact 
\begin{align*}
\beta_ {m,l_1,l_2}&=\sum_{k=1}^{a_{n,l_1}}\frac{r^{l_{2}k}_{l_1}A_{m-l_1}(4^k)\varphi^{-1}_{(l_1,l_2)}\circ\Phi^{-1}_{l_1-l_2}
(1/\alpha_{r_{l_1},k})}{4^k\Phi^{-1}_{l_1-l_2}(1/\alpha_{r_{l_1},k})}\\ &\lesssim \sum_{k=1}^{a_{n,l_1}}\frac{C_{m,l_1,l_2,k}}{\log^{1+\varepsilon}
\Big(e+\log\big(\frac{1/\alpha_{r_{l_1},k}}{(e+1/\alpha_{r_{l_1},k})^m}\big)\Big)}\\
&\lesssim\frac{1}{\varepsilon},
\end{align*}
we deduce that for any $1\leq l_1<m$ and $\,0\leq l_2\leq l_1-1$,
\begin{equation}\label{ie5.22}
\begin{aligned}
K_{\varphi_{(l_1,l_2)}}&\int_{\mathbb{R}^n}A_{m-l_1}
(\frac{|f(x)|}{\lambda})M_{\Phi_{l_1-l_2}\circ\varphi_{(l_1,l_2)}}w(x)dx\\
&\lesssim \frac{1}{\varepsilon}\int_{\mathbb{R}^n}\Phi_{m-l_1}
(\frac{|f(x)|}{\lambda})M_{L(\log L)^m(\log \log L)^{1+\varepsilon}}w(x)dx\\
&\leq \frac{1}{\varepsilon}\int_{\mathbb{R}^n}\Phi_{m}
(\frac{|f(x)|}{\lambda})M_{L(\log L)^m(\log \log L)^{1+\varepsilon}}w(x)dx.
\end{aligned}
\end{equation}
\textbf{Case III:} $l_1=m$. In the last case, we observe that for any $0\leq l_2\leq m-1$,
\begin{align*}
\int_{1}^{\infty}&\frac{\varphi^{-1}_{(m,l_2)}\circ\Phi^{-1}_{m-l_2}(t)A_{0}(\log^{2m+1}(e+t))}{t^2\log^{m+2}(e+t)}dt\\	 &=\int_{1}^{\infty}\frac{\varphi^{-1}_{(m,l_2)}\circ\Phi^{-1}_{m-l_2}(t)\log^{m-1}(e+t)}{t^2}dt\\
&=\int_{\Phi^{-1}_{m-l_2}(1)}^{\infty}\frac{\varphi^{-1}_{(m,l_2)}(t)\log^{m-1}(e+\Phi_{m-l_2}(t))}{\Phi^{2}_{m-l_2}(t)}\Phi^{'}_{m-l_2}(t)dt\\
&\lesssim
\int_{\Phi^{-1}_{m-l_2}(1)}^{\infty}\frac{\varphi^{-1}_{(m,l_2)}(t)\log^{l_2}(e+t)}{t^2\log(e+t)}dt.
\end{align*}
Take $\varphi_{(m,l_2)}(t)=t\log^{l_2}(e+t)\log^{1+\varepsilon}(e+\log(e+t))$,
a simple calculation shows that
$$\int_{1}^{\infty}\frac{\varphi^{-1}_{(m,l_2)}\circ\Phi^{-1}_{m-l_2}(t)
A_{0}(\log^{2m+1}(e+t))}{t^2\log^{m+2}(e+t)}dt\lesssim \frac{1}{\varepsilon}.$$
Note that $\beta_{m,m,l_2}\lesssim \frac{1}{\varepsilon}$ and
$$\Phi_{m-l_2}\circ \varphi_{(m,l_2)}(t)\lesssim \varphi_{(m,l_2)}(t)\log^{m-l_2}(e+t)=t\log^{m}(e+t)
\log^{1+\varepsilon}(e+\log(e+t)),$$
then
\begin{equation}\label{ie5.23}
\begin{aligned}
K_{\varphi_{(m,l_2)}}&\int_{\mathbb{R}^n}A_{0}
(\frac{|f(x)|}{\lambda})M_{\Phi_{m-l_2}\circ\varphi_{(m,l_2)}}w(x)dx \lesssim \int_{\mathbb{R}^n}\frac{|f(x)|}{\varepsilon\lambda}M_{L(\log L)^m(\log \log L)^{1+\varepsilon}}w(x)dx.
\end{aligned}
\end{equation}
Combining all estimates in \eqref{ie5.21}-\eqref{ie5.23}, we have
$$w\big(\{x\in \mathbb{R}^n:|T_{\vec b}f(x)|>\lambda\}\big)
\leq\frac{ C_{n,m,T}}{\varepsilon}\int_{\mathbb{R}^n}\Phi_m\Big(\frac{|f(x)|}{\lambda}\Big)M_{L(\log L)^m(\log \log L)^{1+\varepsilon}}w(x)dx.$$
Hence taking all these cases into account and using the homogeneity, we deduce that \eqref{ie5.28} holds for every $\vec{b}\in \mathrm{BMO}^m$.
\end{proof}

\begin{proof}[Proof of Corollary \ref{cor5.1}]
On the one hand, $M_{L(\log L)^m(\log \log L)^{1+\varepsilon}} w(x)\leq c_{\varepsilon} M_{L(\log L)^{m+\varepsilon} }w(x)$ and \eqref{ie5.28} implies \eqref{ie5.27}. On the other hand, in the proof of Theorem \ref{thm5.3}, if we take $\varphi_{(l_1,l_2)}(t)=t\log^{m-l_1+l_2+\varepsilon}(e+t),\,0\leq l_2<l_1\leq m$ and repeat the above process yield \eqref{ie5.24}.
\end{proof}

\begin{proof}[Proof of Corollary \ref{cor5.2}]
It suffices to show that \eqref{ie5.25} holds. Note that $w\in A^{\rm{weak}}_{\infty}$ and it holds for $t\geq 1, \alpha >0$ that $\log t\leq \frac{t^{\alpha}}{\alpha}$, we have
\begin{equation}\label{ie5.29}
M_{L(\log L)^{m+\varepsilon}}w(x)\leq \bigg(1+2\max\{1,2^{s-2}\}\Big(\frac{m+\varepsilon}{s-1}\Big)^{m+\varepsilon}e^{s-1}\bigg)M_{s}w(x),
\end{equation}
where $s=1+(m+\varepsilon)\alpha=1+\frac{1}{\tau_n[w]_{A^{\rm{weak}}_{\infty}}}$.
Then Lemma \ref{lem2.2} implies that
$$M_sw(x)\leq 2Mw(x).$$
Now it remains to verify \eqref{ie5.29}. Indeed, it follows from the elementary calculation that
\begin{align*}
1+2\max\{1,2^{s-2}\}\Big(\frac{m+\varepsilon}{s-1}\Big)^{m+\varepsilon}e^{s-1}
&\leq \max\{2,2^{(m+\varepsilon)\alpha}\}\Big(\frac{1}{\alpha}\Big)^{m+\varepsilon}e^{(m+\varepsilon)\alpha}\\
&\lesssim \max\left\{2,2^{\frac{1}{\tau_n[w]_{A^{\rm{weak}}_{\infty}}}}\right\}
e^{\frac{1}{\tau_n[w]_{A^{\rm{weak}}_{\infty}}}}[w]^{m+\varepsilon}_{A^{\rm{weak}}_{\infty}}\\
&\lesssim \max\left\{e,e^{\frac{2}{\tau_n[w]_{A^{\rm{weak}}_{\infty}}}}\right\}[w]^{m+\varepsilon}_{A^{\rm{weak}}_{\infty}}.
\end{align*}
Take $\varepsilon=\frac{1}{\log(e+[w]_{A^{\rm{weak}}_{\infty}})}$, we obtain
$$\frac{1}{\varepsilon}M_{L(\log L)^{m+\varepsilon}}w(x)\lesssim \max\left\{e,e^{\frac{2}{\tau_n[w]_{A^{\rm{weak}}_{\infty}}}}\right\}
[w]^m_{A^{\rm{weak}}_{\infty}}\log(e+[w]_{A^{\rm{weak}}_{\infty}})Mw(x).$$

In particular, if $w\in A_1$, according to the definition of the $A_1$ weight, we have $$Mw(x)\leq[w]_{A_1}w(x).$$
This estimate combined with \eqref{ie5.25} yields \eqref{ie5.26}, which completes the proof of Corollary \ref{cor5.2}.
\end{proof}

\section{Proofs of Theorem \ref{thm5.4}, Corollaries \ref{cor5.6} and \ref{cor5.4}}\label{sec6}
This section is devoted to establish the Coifman-Fefferman inequality for $T_{\vec{b}}$. For this purpose, we need a generalized H\"{o}lder’s inequality of multilinear version\cite[Lemma 4.1]{hor+tan}, which is a generalization of \cite[Lemma 1]{per9} under the general measure.

\begin{lemma}[\cite{hor+tan}]\label{lem3.1}
Let $\Phi_0,\Phi_1,\Phi_2,\ldots,\Phi_m$ be Young functions. If
$$\Phi^{-1}_{1}(t)\Phi^{-1}_{2}(t)\cdots\Phi^{-1}_{m}(t)\leq D \Phi^{-1}_0(t),$$
then for all functions $f_1,\ldots,f_m$ and all cubes $Q$, we have
$$\|f_1f_2\ldots f_m\|_{\Phi_{0}(\mu),Q}\leq mD\|f_1\|_{\Phi_{1}(\mu),Q}\|f_2\|_{\Phi_{2}(\mu),Q}\cdots\|f_{m}\|_{\Phi_{m}(\mu),Q}.$$
In particular, for any weight $w$ and $s_1,\ldots,s_m\geq 1$. Let $\frac{1}{s}=\sum_{i=1}^m\frac{1}{s_i}$, then 
$$\frac{1}{w(Q)}\int_{Q}|f_1(x)\cdots f_m(x)g(x)|w(x)dx\leq 2^{\frac{1}{s}}\Big(1+\frac{1}{s}\Big)^{\frac{1}{s}}
\prod_{i=1}^m\|f_i\|_{\exp L^{s_i}(w),Q}\|g\|_{L(\log L)^{1/s}(w),Q}.$$
\end{lemma}

By Lemma \ref{lem2.3}, we establish the following weighted John-Nirenberg inequality for BMO functions with $0<p<\infty$, which generalizes the classical case $p>1$.
\begin{lemma}\label{lem3.2}
Let $b\in \rm{BMO}$ and $w\in A_{\infty}$. Then for any $0<p<\infty$, there exists a constant $c$ independent of $b,w$ such that for every cube $Q$, 
$$\Big(\frac{1}{w(Q)}\int_{Q}|b(x)-\langle b\rangle_Q|^pw(x)dx\Big)^{\frac{1}{p}}\leq c[w]_{A_{\infty}}\|b\|_{\rm BMO}.$$
\end{lemma}
\begin{proof}
It suffices to prove that
$$\frac{1}{w(Q)}\int_{Q}\Big(\frac{|b(x)-\langle b\rangle_Q|}{[w]_{A_{\infty}}\|b\|_{\rm BMO}}\Big)^pw(x)dx\leq c,$$
where $c$ is independent of $b,w$ and $Q$.
By the layer cake formula and Lemma \ref{lem2.3}, we obtain
\begin{align*}
\frac{1}{w(Q)}&\int_{Q}\Big(\frac{|b(x)-\langle b\rangle_Q|}{[w]_{A_{\infty}}\|b\|_{\rm BMO}}\Big)^pw(x)dx\\
&=\frac{p}{w(Q)}\int_{0}^{\infty}\alpha^{p-1}w(\{x\in Q:|b(x)-\langle b\rangle_Q|>\alpha[w]_{A_{\infty}}\|b\|_{\rm BMO}\})d\alpha\\
&\leq 2p\int_{0}^{\infty}\alpha^{p-1}\Big(\frac{|\{x\in Q:|b(x)-\langle b\rangle_{Q}|>\alpha[w]_{A_{\infty}}\|b\|_{\rm BMO}\}|}{|Q|}\Big)^{\frac{1}{c_n[w]_{A_{\infty}}}}d\alpha.
\end{align*}
Then the John-Nirenberg inequality

\begin{equation*}
	\big|\big\{x\in Q:|b(x)-\langle b\rangle_{Q}|>\alpha\big\}\big|\leq e|Q|e^{-\frac{\alpha}{2^ne\|b\|_{\rm BMO}}},
\end{equation*}
gives
\begin{align*}
\frac{1}{w(Q)}&\int_{Q}\Big(\frac{|b(x)-\langle b\rangle_Q|}{[w]_{A_{\infty}}\|b\|_{\rm BMO}}\Big)^pw(x)dx\\
&\leq 2p\int_{0}^{\infty}\alpha^{p-1}\Big(e^{1-\frac{\alpha[w]_{A_{\infty}}\|b\|_{\rm BMO}}{2^ne\|b\|_{\rm BMO}}}\Big)^{\frac{1}{c_n[w]_{A_{\infty}}}}d\alpha\\
&\leq 2pe^{\frac{1}{c_n}}\int_{0}^{\infty}\alpha^{p-1}e^{-\frac{\alpha}{2^nc_ne}}d\alpha\leq c.
\end{align*}

\end{proof}
We are now ready to prove Theorem \ref{thm5.4}.
\begin{proof}[Proof of Theorem $\ref{thm5.4}$]
According to Theorem \ref{thm5.1}, we have
\begin{equation*}
|T_{\vec{b}}f(x)|\leq C_{n,m}C_{T}\sum_{j=1}^{3^n}\sum_{\vec{\gamma}\in\{0,1\}^m}\mathcal{A}^{\vec\gamma}_{A,\mathcal{S}_j}(\vec{b},f)(x),\quad \text{a.e.} \ x \in \mathbb{R}^n.
\end{equation*}
By symmetry, we only need to consider the case $\vec{\gamma}\neq \vec{0}$. It follows from the triangle inequality that in order to show the inequality \eqref{ie1.1}, it suffices to show that for every $1\leq j \leq 3^n$,
\begin{equation}\label{ie3.2}
\|\mathcal{A}^{\vec\gamma}_{A,\mathcal{S}_j}(\vec{b},f)\|_{L^p(w)}\lesssim
\prod_{s=1}^m\|b_s\|_{\rm BMO}[w]^{m}_{A_{\infty}}[w]^{\max\{1,1/p\}}_{A_{\infty}}\|M_Bf\|_{L^p(w)}.
\end{equation}
Consider first the case $0<p\leq 1$. By the definition of $\mathcal{A}^{\vec\gamma}_{A,\mathcal{S}_j}(\vec{b},f)$ and $0<p\leq 1$, we have
$$\|\mathcal{A}^{\vec\gamma}_{A,\mathcal{S}_j}(\vec{b},f)\|^p_{L^p(w)}
\leq \sum_{Q\in \mathcal{S}_j}\int_{Q}\prod_{s:\gamma_s=1}|b_s(x)-\langle b_s \rangle_{Q}|^p\Big\|\prod_{s:\gamma_s=0}(b_s-\langle b_s \rangle_Q)f\Big\|^p_{A,Q}w(x)dx.$$
Let 
\begin{equation}\label{ie3.9}
|S|=\#\{s:\gamma_s=1\},
\end{equation}
where $\#$ denotes the cardinal number of sets. Fix exponents $\frac{1}{p}=\frac{1}{p_1}+\cdots+\frac{1}{p_{|S|}}$
with $0<p_1,\ldots,p_{|S|}<\infty$, then the H\"{o}lder's inequality and Lemma \ref{lem3.2} give that
\begin{equation}\label{ie3.3}
\begin{aligned}
\frac{1}{w(Q)}\int_{Q}\prod_{s:\gamma_s=1}|b_s(x)-\langle b_s \rangle_{Q}|^pw(x)dx
&\leq \Big(\frac{1}{w(Q)}\int_{Q}|b_s(x)-\langle b_s \rangle_{Q}|^{p_s}w(x)dx\Big)^{\frac{p}{p_s}}\\
&\lesssim \prod_{s:\gamma_s=1}([w]_{A_{\infty}}\|b_s\|_{\rm BMO})^p.
\end{aligned}
\end{equation}
On the other hand, using the condition $B^{-1}(t)(C^{-1}(t))^{m}\leq cA^{-1}(t)$ and Lemma \ref{lem3.1}, we obtain
$$\Big\|\prod_{s:\gamma_s=0}(b_s-\langle b_s \rangle_Q)f\Big\|_{A,Q}\lesssim \prod_{s:\gamma_s=0}\|b\|_{\rm{BMO}}\|f\|_{B,Q}.$$
This inequality, together with \eqref{ie3.3}, implies that
\begin{equation}\label{ie3.4}
\begin{aligned}
\|\mathcal{A}^{\vec\gamma}_{A,\mathcal{S}_j}
(\vec{b},f)\|^p_{L^p(w)}
&\lesssim \prod_{s=1}^m\|b_s\|^p_{\rm {BMO}}[w]^{p|S|}_{A_{\infty}}\sum_{Q\in \mathcal{S}_j}\|f\|^p_{B,Q}w(Q)\\
&\leq 
\prod_{s=1}^m\|b_s\|^p_{\rm {BMO}}[w]^{mp}_{A_{\infty}}\sum_{Q\in \mathcal{S}_j}\Big(\frac{1}{w(Q)}\int_{Q}\|f\|^{\frac{p}{2}}_{B,Q}w(x)dx
\Big)^2\\
&\leq \prod_{s=1}^m\|b_s\|^p_{\rm {BMO}}[w]^{mp}_{A_{\infty}}\sum_{Q\in \mathcal{S}_j}\Big(\frac{1}{w(Q)}\int_{Q}(M_Bf(x))^{\frac{p}{2}}w(x)dx
\Big)^2.
\end{aligned}
\end{equation}
Note that $\mathcal{S}_j$ is a $\frac{1}{2\cdot 9^n}$-sparse family, i.e, for any $R\in \mathcal{S}_j$, there exists $E_R\subset R$ such $|E_R|\geq \frac{1}{2\cdot 9^n}|R|$, then for every dyadic cube $Q\in \mathcal{S}_j$, it follows that
\begin{align*}
  \sum_{R\subseteq Q}w(R)= \sum_{R\subseteq Q}\frac{w(R)}{|R|}|R|&\leq 2\cdot 9^n\sum_{R\subseteq Q}\frac{w(R)}{|R|}|E_R|\\
   & \leq 2\cdot 9^n \sum_{R\subseteq Q} \int_{E_R}M(w\chi_{R})(x)dx\\
   & \leq2\cdot 9^n\int_{Q}M(w\chi_{Q})(x)dx\\
   & \leq2\cdot 9^n[w]_{A_{\infty}}w(Q).
\end{align*}
Combining the Carleson embedding theorem with \eqref{ie3.4}, it yields that
\begin{equation}\label{ie3.5}
\begin{aligned}
\|\mathcal{A}^{\vec\gamma}_{A,\mathcal{S}_j}
(\vec{b},f)\|^p_{L^p(w)}
&\lesssim \prod_{s=1}^m\|b_s\|^p_{\rm {BMO}}[w]^{mp}_{A_{\infty}}\sum_{Q\in \mathcal{S}_j}\Big(\frac{1}{w(Q)}\int_{Q}(M_Bf(x))^{\frac{p}{2}}w(x)dx
\Big)^2\\
& \lesssim \prod_{s=1}^m\|b_s\|^p_{\rm {BMO}}[w]^{mp+1}_{A_{\infty}} \|M_Bf\|^p_{L^p(w)}.
\end{aligned}
\end{equation}
We now turn our attention to the case $p>1$. By duality, it follows that
\begin{equation}\label{ie3.6}
\|\mathcal{A}^{\vec\gamma}_{A,\mathcal{S}_j}
(\vec{b},f)\|_{L^p(w)}=\sup_{\|g\|_{L^{p'}(w)}\leq 1}\bigg|\int_{\Rn}\mathcal{A}^{\vec\gamma}_{A,\mathcal{S}_j}
(\vec{b},f)(x)g(x)w(x)dx\bigg|.
\end{equation}
For any fixed function $g\in L^{p'}(w)$ with $\|g\|_{L^{p'}(w)}\leq 1$. Applying Lemma \ref{lem3.1} and the same argument as in the case of $0<p\leq 1$, we obtain
\begin{align*}
\bigg|&\int_{\Rn}\mathcal{A}^{\vec\gamma}_{A,\mathcal{S}_j}
(\vec{b},f)(x)g(x)w(x)dx\bigg|\\
&\leq \sum_{Q\in \mathcal {S}_j}\int_{Q}\Big\|\prod_{s:\gamma_s=0}(b_s-\langle b_s\rangle_Q)f\Big\|_{A,Q}\frac{w(Q)}{w(Q)}\int_{Q}\prod_{s:\gamma_s=1}
|b_s(x)-\langle b_s\rangle_Q||g(x)|w(x)dx\\
&\lesssim \sum_{Q\in \mathcal {S}_j}\Big\|\prod_{s:\gamma_s=0}(b_s-\langle b_s\rangle_Q)f\Big\|_{A,Q}\prod_{s:\gamma_s=1}\|b_s-\langle b_s\rangle_Q\|_{\exp L(w),Q}\|g\|_{L(\log L)^{1/|S|}(w),Q}w(Q)\\
&\lesssim \prod_{s=1}^m\|b_s\|_{\rm BMO}[w]^m_{A_{\infty}}\sum_{Q\in \mathcal {S}_j}\|f\|_{B,Q}\|g\|_{L(\log L)^{1/|S|}(w),Q}w(Q),
\end{align*}
where the last inequality is due to the fact that $\|b-\langle b\rangle_Q\|_{\exp L(w),Q}\lesssim [w]_{A_{\infty}}\|b\|_{\rm BMO}$ (\cite[Lemma 4.7]{iba}) and $|S|$ is defined in \eqref{ie3.9}.
Let $ M^{\mathcal{D}}_{r(w)}$ be defined by
$$M^{\mathcal{D}}_{r(w)}:=\sup_{Q\ni x,Q \in \mathcal{D}}\Big(\frac{1}{w(Q)}\int_{Q}|f(x)|^rw(x)dx\Big)^{\frac{1}{r}}.$$ 
Then the Carleson embedding theorem and the H\"{o}lder’s inequality give that
\begin{equation*}
\begin{aligned}
&\bigg|\int_{\Rn}\mathcal{A}^{\vec\gamma}_{A,\mathcal{S}_j}
(\vec{b},f)(x)g(x)w(x)dx\bigg|\\
&\lesssim \prod_{s=1}^m\|b_s\|_{\rm BMO}[w]^m_{A_{\infty}}\sum_{Q\in \mathcal {S}_j}\|f\|_{B,Q}\|g\|_{L(\log L)^{1/|S|}(w),Q}w(Q)\\
& \lesssim \prod_{s=1}^m\|b_s\|_{\rm BMO}[w]^m_{A_{\infty}}\sum_{Q\in \mathcal {S}_j}\Big(\frac{1}{w(Q)}\int_{Q}\big(M_{B}f(x)
M^{\mathcal D}_{L(\log L)^{1/|S|}(w)}g(x)\big)^{\frac{1}{2}}w(x)dx\Big)^2\\
&\lesssim \prod_{s=1}^m\|b_s\|_{\rm BMO}[w]^{m+1}_{A_{\infty}}\|M_Bf\|_{L^p(w)}\|M^{\mathcal D}_{L(\log L)^{1/|S|}(w)}g\|_{L^{p'}(w)}\\
&\lesssim \prod_{s=1}^m\|b_s\|_{\rm BMO}[w]^{m+1}_{A_{\infty}}\|M_Bf\|_{L^p(w)}\|M^{\mathcal{D}}_{r(w)}g\|_{L^{p'}(w)},
\end{aligned}
\end{equation*}
where the last inequality follows from the fact that $$M^{\mathcal D}_{L(\log L)^{1/|S|}(w)}g(x)\lesssim M^{\mathcal{D}}_{r(w)}g(x), \,1<r<p'.$$ 

Note that for any weight $w$, $\|M^{\mathcal D}_{w}g\|_{L^p(w)}\leq C\|g\|_{L^p(w)}$ (\cite[Theorem 15.1]{hor+ler18}), where $C$ is independent of $w$, we have
\begin{equation}\label{ie3.8}
\begin{aligned}
\|\mathcal{A}^{\vec\gamma}_{A,\mathcal{S}_j}
(\vec{b},f)\|_{L^p(w)}&\lesssim \prod_{s=1}^m\|b_s\|_{\rm BMO}[w]^{m+1}_{A_{\infty}}\|M_Bf\|_{L^p(w)}
\sup_{\|g\|_{L^{p'}(w)}\leq 1}\|M^{\mathcal{D}}_{w}(|g|^r)\|^{\frac{1}{r}}_{L^{p'/r}(w)}\\
&\lesssim \prod_{s=1}^m\|b_s\|_{\rm BMO}[w]^{m+1}_{A_{\infty}}\|M_Bf\|_{L^p(w)}.
\end{aligned}
\end{equation}

Combining \eqref{ie3.5} with \eqref{ie3.8} yields the desired estimate, which finishes the proof of Theorem \ref{thm5.4}.
\end{proof}
Consider now the proofs of Corollaries \ref{cor5.6} and \ref{cor5.4}.

\begin{proof}[Proof of Corollary $\ref{cor5.6}$]
Let $A(t)\approx t^{r'}$ and $B(t)= t^{r'}\log^{r'm}(e+t)$, then $T$ is an $\bar A$-H\"{o}rmander operator and a simple calculation gives that
$$B^{-1}(t)\approx\frac{t^{\frac{1}{r'}}}{\log^m(e+t)}.$$
It is easy to check that $B^{-1}(t)(C^{-1}(t))^{m}\leq cA^{-1}(t)$. Note that for any cube $Q$ and $0<\varepsilon<1$, 
\begin{equation}\label{ie11}
\|f\|_{L^{r'}(\log L)^{r'm},Q}\leq C\Big(\frac{1}{|Q|}\int_Q|f(y)|^{r'+\varepsilon}dy\Big)^{\frac{1}{r'+\varepsilon}}.
\end{equation}
Then by Theorem \ref{thm5.4}, we obtain
$$\int_{\Rn}|T_{\vec b}(f)(x)|^pw(x)dx\lesssim \prod_{s=1}^m\|b_s\|^p_{\rm BMO}[w]^{mp}_{A_{\infty}}[w]^{\max\{1,p\}}_{A_{\infty}}\int_{\Rn}(M_{r'+\varepsilon}f(x))^pw(x)dx,$$
which completes the proof of Corollary \ref{cor5.6}.
\end{proof}

\begin{proof}[Proof of Corollary $\ref{cor5.4}$]
As mentioned above, if $T$ is an $\omega$-Calder\'{o}n-Zygmund operator and $\omega$ satisifies a Dini condition, then $T$ is an $\bar{A}$-H\"{o}rmander operator with $A(t)=t$.   
Let $B(t)=t\log^{m}(e+t)$, it follows that $B^{-1}(t)(C^{-1}(t))^{m}\leq cA^{-1}(t)$. Therefore, Theorem \ref{thm5.4}, together with the fact that $M_{L(\log L)^{m} }f\approx M^{m+1}f$, yields the required estimate \eqref{ie1.2}.
\end{proof}

\section{ Proofs of Theorem \ref{thm5.5}}\label{sec7}
This section will be devoted to demonstrate Theorem \ref{thm5.5}. For this purpose, we first introduce a new class of kernels depending on Young function $A$ and exponent $m$. 

\begin{definition}\label{def7.1}
Let $A$ be a Young function and $m\in \mathbb N$. The kernel $K$ is said to satisfy the $L^{A,m}$-H\"{o}rmander condition, if
\begin{align*}
   & \sup_{Q}\sup_{x,z\in \frac{1}{2} Q}\sum_{k=1}^{\infty}|2^kQ|m^k\left\|(K(x,\cdot)-K(z,\cdot))\chi_{2^k Q \backslash 2^{k-1} Q}\right\|_{A,2^k Q}<\infty,  \\
   & \sup _Q \sup _{x, z \in \frac{1}{2} Q} \sum_{k=1}^{\infty}|2^kQ|m^k\left\|(K(\cdot, x)-K(\cdot, z)) \chi_{2^k Q \backslash 2^{k-1} Q}\right\|_{A, 2^k Q}<\infty,
\end{align*}
we denote by $\mathcal{H}_{A,m}$ the class of kernels satisfying the $L^{A,m}$-H\"{o}rmander condition.
\end{definition}

We need to introduce some notations. Fix two integers $j$ and $m$ such that $0\leq j <m$. Let $C^m_{j}$ denote the family of all finite subsets $\sigma=\{\sigma(1),\ldots,\sigma(j)\}$ of $\{1,\ldots,m\}$ with $j$ different elements, when $j=0$, $\sigma$ is abbreviated as $\sigma=\varnothing$. For any $\sigma\in C^m_{j}$, the complementary sequence of $\sigma$ is given by $\sigma'=\{1,\ldots,m\}\backslash\sigma$.

For a finite family of integrable functions $\vec {b}=(b_1,\ldots,b_m)$ and $\sigma\in C^m_{j}$, we denote $\vec{b}_{\sigma}=(b_{\sigma(1)},\ldots,b_{\sigma(j)})$, $\|\vec{b}_{\sigma}\|_{\rm BMO}=\prod_{i\in \sigma}\|b_i\|_{\rm BMO}$ and $\|\vec{b}_{\sigma'}\|_{\rm BMO}=\prod_{i\in \sigma'}\|b_i\|_{\rm BMO}$.
With this notation, given an A-H\"{o}rmander operator $T$, the general commutator is defined by 
$$T_{\vec{b}_{\sigma}}f(x)=\int_{\Rn}\prod_{i\in \sigma}(b_i(x)-b_i(y))K(x,y)f(y)dy, \,\, x\notin \supp f.$$
In the particular case of $\sigma=\{1,\ldots,m\}$, we write $T_{\vec{b}_{\sigma}}=T_{\vec{b}}$.

The following Kolmogorov's inequality, which can be found in \cite[P. 485]{gar} will be used in our analysis.
\begin{lemma}[\cite{gar}]\label{lem7.1}
Let $0 < p < q < \infty$, there exists a constant $C$ depending on $p, q$ such that for any measurable function $f$,
$$\|f\|_{L^p(Q, \frac{dx}{|Q|})}\leq C\|f\|_{L^{q, \infty}(Q, \frac{dx}{|Q|})}.$$
\end{lemma}

The following lemma presents a pointwise estimate of $M^{\#}_{\delta}(T_{\vec b}f)$, which plays a crucial role in the quantitative weighted Coifman-Fefferman inequalities.

\begin{lemma}\label{lem6.1}
Let $m\in\mathbb{N}$ and $\vec{b}=(b_1,\ldots,b_m)\in \mathrm{BMO}^m$. Assume that $A,B$ are Young functions such that $B^{-1}(t)(C^{-1}(t))^m\leq cA^{-1}(t)$ for $t\geq 1$, where $C(t)=e^{t}-1$. If $T$ is an $\bar{A}$-H\"{o}rmander operator with kernel $K\in \mathcal{H}_{\bar{A},m}$, then for any $0<\delta <\varepsilon <1$,
\begin{equation}
M^{\#}_{\delta}(T_{\vec b}f)(x)\leq C \sum_{j=0}^{m-1}\sum_{\sigma\in C^m_{j}}\|\vec{b}_{\sigma'}\|_{\rm BMO}M_{\varepsilon}(T_{\vec{b}_{\sigma}}f)(x)+C
\|\vec{b}\|_{\rm BMO}M_Bf(x).
\end{equation}
\end{lemma}
\begin{proof}
The case $b_i=b~(i=1,\ldots,m)$ was already proved in \cite{hor+lor2}. Fix $x\in \Rn$ and a cube $Q\ni x$ with center $x_Q$, set $Q^{*}=2Q$. Since $0<\delta<1$, then $\left||a|^{\delta}-|b|^{\delta}\right| \leq |a-b|^{\delta}$ for $a,b\in \mathbb R$, it is enough to show that 
\begin{equation}\label{ie1}
\Big(\frac{1}{|Q|}\int_{Q}|T_{\vec b}f(y)-c_{Q}|^{\delta}dy\Big)^{\frac{1}{\delta}}
\leq C \sum_{j=0}^{m-1}\sum_{\sigma\in C^m_{j}}\|\vec{b}_{\sigma'}\|_{\rm BMO}M_{\varepsilon}(T_{\vec{b}_{\sigma}}f)(x)+C
\|\vec{b}\|_{\rm BMO}M_Bf(x),
\end{equation}
where the constant $c_Q$ will be choosen later and $C$ is independent of $x$ and $Q$.
Applying \cite[Lemma 3.2]{hor+xue}, one may get
\begin{align*}
\prod_{j=1}^m(b_i(y)-b_i(z))
&=\prod_{i=1}^m(\langle b_i\rangle_{Q^{*}}-b_i(z))+\sum_{j=0}^{m-1}\sum_{\sigma\in C^m_{j}}(-1)^{\#\sigma'+1}\prod_{i\in \sigma}(b_i(y)-b_i(z))\\
&\quad\times\prod_{i\in\sigma'}(b_i(y)-\langle b_i\rangle_{Q^{*}}).
\end{align*}
Therefore,
\begin{align*}
|T_{\vec b}f(y)-c_{Q}|&\leq \Big|\int_{\Rn}\prod_{i=1}^m (\langle b_i\rangle_{Q^{*}}-b_i(z))K(y,z)f(z)dz-c_{Q}\Big|\\
&\quad+\sum_{j=0}^{m-1}\sum_{\sigma\in C^m_{j}}
\prod_{i\in\sigma'}|b_i(y)-\langle b_i\rangle_{Q^{*}}|\Big|\int_{\Rn}
\prod_{i\in \sigma}(b_i(y)-b_i(z))K(y,z)f(z)dz\Big|\\
&=\Big|T\Big(\prod_{i=1}^m(\langle b_i\rangle_{Q^{*}}-b_i)f\Big)(y)-c_Q\Big|+
\sum_{j=0}^{m-1}\sum_{\sigma\in C^m_{j}}\prod_{i\in{\sigma'}}|b_i(y)-\langle b_i\rangle_{Q^{*}}|
|T_{\vec{b}_{\sigma}}f(y)|.
\end{align*}
Now we split $f$ into $f=f_1+f_2$, where $f_1=f\chi_{Q^{*}}$ and $f_2=f\chi_{(Q^{*})^c}$, which implies that
\begin{equation}\label{ie2}
\begin{aligned}
{}&\Big(\frac{1}{|Q|}\int_{Q}|T_{\vec b}f(y)-c_{Q}|^{\delta}dy\Big)^{\frac{1}{\delta}}\\&
\lesssim\Big(\frac{1}{|Q|}\int_{Q}\Big|T\Big(\prod_{i=1}^m(\langle b_i\rangle_{Q^{*}}-b_i)f_1\Big)(y)\Big|
^{\delta}dy\Big)^{\frac{1}{\delta}}\\
&\quad+\Big(\frac{1}{|Q|}\int_{Q}\Big|T\Big(\prod_{i=1}^m(\langle b_i\rangle_{Q^{*}}-b_i)f_2\Big)(y)-c_{Q}\Big|
^{\delta}dy\Big)^{\frac{1}{\delta}}\\
&\quad+\sum_{j=0}^{m-1}\sum_{\sigma\in C^m_{j}}\Big(\frac{1}{|Q|}\int_{Q}\prod_{i\in{\sigma'}}|b_i(y)-\langle b_i\rangle_{Q^{*}}|^{\delta}
|T_{\vec{b}_{\sigma}}f(y)|^{\delta}dy\Big)^{\frac{1}{\delta}}\\&=:I+II+III.
\end{aligned}
\end{equation}
For $I$, the condition $\mathcal{H}_{\bar A}\subset \mathcal{H}_1$ implies that $T$ is of weak type $(1,1)$. Therefore, Lemma \ref{lem3.1} and Lemma \ref{lem7.1} yield that
\begin{equation}\label{ie3}
\begin{aligned}
I&\lesssim\Big\|T\Big(\prod_{i=1}^m(\langle b_i\rangle_{Q^{*}}-b_i)f_1\Big)\Big\|_{L^{1,\infty}(Q,\frac{dy}{|Q|})}\\
&\lesssim\frac{1}{|Q^{*}|}\int_{Q^{*}}\prod_{i=1}^m|\langle b_i\rangle_{Q^{*}}-b_i(y)||f(y)|dy\\
&\lesssim\prod_{i=1}^{m}\|\langle b_i\rangle_{Q^{*}}-b_i\|_{\exp L,Q^{*}}\|f\|_{B,Q^{*}}\\&\leq C\|\vec b\|_{\rm BMO}M_{B}f(x).
\end{aligned}
\end{equation}
Consider now the term $II$. Let $c_Q=T\big(\prod_{i=1}^m(\langle b_i\rangle_{Q^{*}}-b_i)f_2\big)(x_Q)$. Then it follows from the Jensen's inequality that

\begin{equation}\label{ie4}
\begin{aligned}
II&\lesssim \frac{1}{|Q|}\int_{Q}\Big|\int_{\Rn}
\prod_{i=1}^m(\langle b_i\rangle_{Q^{*}}-b_i(z))(K(y,z)-K(x_Q,z))
f_2(z)dz\Big|dy\\
&\lesssim \frac{1}{|Q|}\int_{Q}\sum_{j=1}^{\infty}\int_{2^{j+1}Q\backslash2^{j}Q}
\prod_{i=1}^m|\langle b_i\rangle_{Q^{*}}-b_i(z)||K(y,z)-K(x_Q,z)|f(z)|dzdy\\
&\lesssim \frac{1}{|Q|}\int_{Q}\sum_{j=1}^{\infty}\int_{2^{j+1}Q\backslash2^{j}Q}
\prod_{i=1}^m|\langle b_i\rangle_{Q^{*}}-\langle b_i\rangle_{2^{j+1}Q}||K(y,z)-K(x_Q,z)|f(z)|dzdy\\
&\quad+\frac{1}{|Q|}\int_{Q}\sum_{j=1}^{\infty}\int_{2^{j+1}Q\backslash2^{j}Q}
\prod_{i=1}^m|\langle b_i\rangle_{2^{j+1}Q}-b_i(z)||K(y,z)-K(x_Q,z)|f(z)|dzdy\\
&=:II_1+II_2.
\end{aligned}
\end{equation}
By the generalized H\"{o}lder’s inequality with $\bar{A},B$ and $K\in \mathcal{H}_{\bar{A},m}$, we have
\begin{equation}\label{ie5}
\begin{aligned}
II_1 &\lesssim \frac{1}{|Q|}\int_{Q}\sum_{j=1}^{\infty} \frac{|2^{j+1}Q|}{|2^{j+1}Q|}\int_{2^{j+1}Q\backslash2^{j}Q}\prod_{i=1}^m|\langle b_i\rangle_{Q^{*}}-\langle b_i\rangle_{2^{j+1}Q}||K(y,z)-K(x_Q,z)|f(z)|dzdy\\
& \lesssim \frac{1}{|Q|}\int_{Q}\|\vec b\|_{\rm BMO}\sum_{j=1}^{\infty}|2^{j+1}Q|j^m\Big\|(K(y,\cdot)-K(x_Q,\cdot))\chi_{2^{j+1}Q\backslash2^{j}Q}\Big\|_{\bar A,2^{j+1}Q}\|f\|_{B,2^{j+1}Q}dy\\
&\lesssim \|\vec b\|_{\rm BMO}M_{B}f(x),
\end{aligned}
\end{equation}
where the second inequality follows from the fact that $|\langle b_i\rangle_{2^{j+1}Q}-\langle b_i\rangle_{2Q}|\leq Cj\|b_i\|_{\rm BMO}$.\\
For $II_2$, using Lemma \ref{lem3.1}, it holds that
\begin{equation}\label{ie6}
\begin{aligned}
II_2&\lesssim \frac{1}{|Q|}\int_{Q}\sum_{j=1}^{\infty} \frac{|2^{j+1}Q|}{|2^{j+1}Q|}\int_{2^{j+1}Q\backslash2^{j}Q}\prod_{i=1}^m|\langle b_i\rangle_{2^{j+1}Q}-b_i(z)||K(y,z)-K(x_Q,z)|f(z)|dzdy\\
&\lesssim \frac{1}{|Q|}\int_{Q}\|\vec b\|_{\rm BMO}\sum_{j=1}^{\infty}|2^{j+1}Q|\Big\|(K(y,\cdot)-K(x_Q,\cdot))\chi_{2^{j+1}Q\backslash2^{j}Q}\Big\|_{\bar A,2^{j+1}Q}\|f\|_{B,2^{j+1}Q}dy\\
&\lesssim\|\vec b\|_{\rm BMO}M_{B}f(x).
\end{aligned}
\end{equation}
Now we turn to the last term $III$. For every $\sigma$, since $0<\delta<\varepsilon<1$, take $\delta_i,\,i\in \sigma'$ such that $\sum_{i\in \sigma'}\frac{1}{\delta_i}+
\frac{1}{\varepsilon}=\frac{1}{\delta}$, by H\"{o}lder's inequality, we obtain
\begin{equation*}
\begin{aligned}
III&\lesssim\sum_{j=1}^m\sum_{\sigma\in C^m_{j}}
\prod_{i\in\sigma'}\Big(\frac{1}{|Q^{*}|}\int_{Q^{*}}|\langle b_i\rangle_{Q^{*}}-b_i(y)|^{\delta_i}dy\Big)^{\frac{1}{\delta_i}}
\Big(\frac{1}{|Q|}\int_{Q}|T_{\vec b_{\sigma}}f(y)|^{\varepsilon}dy\Big)^{\frac{1}{\varepsilon}}\\
&\lesssim\sum_{j=1}^m\sum_{\sigma\in C^m_{j}}\|\vec{b}_{\sigma'}\|_{\rm BMO}M_{\varepsilon}(T_{\vec {b}_{\sigma}}f)(x).
\end{aligned}
\end{equation*}
This estimate, together with \eqref{ie2}-\eqref{ie6}, gives the desired conclusion.
\end{proof}
We are ready to demonstrate Theorem \ref{thm5.5}.
\begin{proof}[Proof of Theorem $\ref{thm5.5}$]
As shown in \cite{hor+per11}, we may assume that $\{b_i\}_{i=1}^m$ and the weight $w$ are all bounded functions. By the same scheme of that in \cite{hor+liyu}, we may obtain 
\begin{claim}\label{claim A} 
There exists some $\delta_0$ with $0< \delta_0 < 1$ such that for any finite subset $\sigma$ of $\{1,\ldots,m\}$ and any bounded functions $f$ with compact support, $\|M_{\delta_0}(T_{\vec{b}_{\sigma}}f)\|_{L^p(w)} <\infty$.
\end{claim}
Combining Lemma \ref{lem2.4} with Claim \ref{claim A} and applying Lemma \ref{lem6.1}, it suffices to show that there exsits $\delta$ with $0<\delta<\delta_0 < 1$ such that 
\begin{equation}\label{ie7}
\|M^{\#}_{\delta}(T_{\vec b}f)\|_{L^p(w)}\leq C\|\vec b\|_{\rm BMO}[w]^{m}_{A_{\infty}}\|M_Bf\|_{L^p(w)}.
\end{equation}
Note that when $m=0$, then $T_{\vec b}=T$ and $K\in \mathcal{H}_{\bar A}$. Lemma \ref{lem6.1}, together with Lemma \ref{lem2.4}, gives that   
\begin{equation}\label{ie7.2}
\|M^{\#}_{\delta_0}(Tf)\|_{L^p(w)}\leq C\|M_Bf\|_{L^p(w)}.
\end{equation}
To confirm the validity of \eqref{ie7}, we proceed by employing mathematical induction. If $m=1$, for $0<\delta<\delta_0<1$, by Lemma \ref{lem2.4}, Lemma \ref{lem6.1} and \eqref{ie7.2}, we have
\begin{align*}
\|M^{\#}_{\delta_0}(T_{b}f)\|_{L^p(w)}&\leq C\|b\|_{\rm{BMO}}[w]_{A_{\infty}}\|M^{\#}_{\delta_0}(Tf)\|_{L^p(w)}+C\|b\|_{\rm BMO}\|M_Bf\|_{L^p(w)}\\
&\leq C\|b\|_{\rm BMO}[w]_{A_{\infty}}\|M_Bf\|_{L^p(w)}.
\end{align*}
Now we turn to the inductive step. Suppose that \eqref{ie7} holds for $1 \leq \#\sigma \leq m-1$, and let us prove that it also holds for $\#\sigma=m$. Notice that $B^{-1}(t)(C^{-1}(t))^j\leq cA^{-1}(t)$ and $K\in \mathcal{H}_{\bar{A},m}\subset\mathcal{H}_{\bar{A},j}$
with $1\leq j\leq m-1$, therefore, the induction hypothesis implies that for any $\sigma$ with $1\leq \#\sigma \leq m-1$, there exists $\delta(\sigma)$, with $0< \delta(\sigma)< \delta_0$ such that
\begin{equation}\label{ie8}
\|M^{\#}_{\delta(\sigma)}(T_{\vec{b}_{\sigma}}f)\|_{L^p(w)}\leq C\|\vec b_{\sigma}\|_{\rm BMO}[w]^{\#\sigma}_{A_{\infty}}\|M_Bf\|_{L^p(w)}.
\end{equation}
Choose $\delta$ satisfing $\delta <\tilde{\delta}:=\min\limits_{\sigma}\{\delta(\sigma): \sigma \subset \{1,\ldots,m\}\} <\delta_0$, it then follows from Lemma \ref{lem6.1}, Claim \ref{claim A} and \eqref{ie8} that
\begin{equation*}
\begin{aligned}
\|M^{\#}_{\delta}(T_{\vec{b}}f)\|_{L^p(w)}&\leq C
\sum_{j=0}^{m-1}\sum_{\sigma\in C^m_{j}}\|\vec{b}_{\sigma'}\|_{\rm BMO}\|M_{\tilde{\delta}}(T_{\vec{b}_{\sigma}}f)\|_{L^p(w)}+C
\|\vec{b}\|_{\rm BMO}\|M_Bf\|_{L^p(w)}\\
&\leq C [w]_{A_{\infty}}\sum_{j=0}^{m-1}\sum_{\sigma\in C^m_{j}}\|\vec{b}_{\sigma'}\|_{\rm BMO}\|M^{\#}_{\delta(\sigma)}(T_{\vec{b}_{\sigma}}f)\|_{L^p(w)}+C
\|\vec{b}\|_{\rm BMO}\|M_Bf\|_{L^p(w)}\\
&\leq C [w]_{A_{\infty}}\|\vec{b}\|_{\rm BMO}\sum_{j=0}^{m-1}\sum_{\sigma\in C^m_{j}}[w]^{\#\sigma}_{A_{\infty}}\|M_Bf\|_{L^p(w)}+C
\|\vec{b}\|_{\rm BMO}\|M_Bf\|_{L^p(w)}\\
&\leq C [w]^{m}_{A_{\infty}}\|\vec{b}\|_{\rm BMO}\|M_Bf\|_{L^p(w)}. 
\end{aligned}
\end{equation*}
Thus, \eqref{ie7} is proved by induction, which finishes the proof of Theorem \ref{thm5.5}.
\end{proof}

\section{applications}\label{sec8}
In this section, we present some applications of the results obtained in Theorem \ref{thm5.2} and Theorem \ref{thm5.4}. To be specific, we will give the Fefferman-Stein inequality with arbitrary weights and the quantitative weighted Coifman-Fefferman inequality for the general commutators of Calder\'{o}n commutators of Baj\v{s}anski-Coifman type, homogeneous singular operators, and Fourier multipliers.
\subsection{Calder\'{o}n commutators of Baj\v{s}anski-Coifman type}
It is well known that the study of Calder\'{o}n commutators is closely connected to the
Cauchy integral on Lipschitz curves and the elliptic boundary value problems. In 1967, Baj\v{s}anski and Coifman\cite{hor+baj} introduced a more general Calder\'{o}n commutator as follows. Given $k\in\mathbb N_{+}$ and $A$ satisfies Lipschitz condition \eqref{ie20}, define the singular operator $T_{A,k}$ as
\begin{equation*}
T_{A,k}f(x):=p.v.\int_{\Rn}\frac{1}{|x-y|^n}
\frac{P_{k}(A;x,y)}{|x-y|^k}f(y)dy,
\end{equation*}
where $$P_{k}(A;x,y):=A(x)-\sum\limits_{|\alpha|<k}\frac{A_{\alpha}(y)}{\alpha!}(x-y)^{\alpha},\,\,
A_{\alpha}(x)=\partial^{\alpha}_{x}A(x).$$

Assume that $A_{\alpha}\in L^{\infty}(\Rn)$ for each $|\alpha|=k$. If we denote the kernel of $T_{A,k}$ by
$$K_{A,k}(x,y):=\frac{1}{|x-y|^n}
\frac{P_{k}(A;x,y)}{|x-y|^k},$$
then it follows from \cite[p.1671]{hor+ding} that
$K_{A,k}$ is a standard Calder\'{o}n-Zygmund kernel. Furthermore, if $k$ is odd, then  $T_{A,k}$ is bounded on $L^2(\Rn)$ \cite[Theorem F]{hor+ding}. We may obtain the following results by Theorem 
\ref{thm5.3} and Theorem \ref{thm5.4}.

\begin{theorem}\label{thm8.1}
Let $m\in\mathbb{N}$ and $\vec{b}=(b_1,\ldots,b_m)\in \mathrm{BMO}^m$. Suppose that $k$ is odd and $A_{\alpha}\in L^{\infty}(\Rn)$ for each $|\alpha|=k$, then for any weight $w$ and every $0<\varepsilon<1$, we have
\begin{equation*}
w\big(\{x\in \mathbb{R}^n:|T_{A,k,\vec b}f(x)|>\lambda\}\big)
\lesssim \frac{1}{\varepsilon}\int_{\mathbb{R}^n}\Phi_m\Big(\frac{\big\|\vec{b}\big\|_{\mathrm {BMO}}|f(x)|}{\lambda}\Big)M_{L(\log L)^m(\log \log L)^{1+\varepsilon}}w(x)dx.
\end{equation*}
\end{theorem}

\begin{theorem}
Let $0<p<\infty$, $m\in\mathbb{N}$ and $b\in \mathrm{BMO}$. Assume $k,A$ satisfy the same conditions as Theorem \ref{thm8.1}, then for any $w\in A_{\infty}$,
\begin{equation*}
\int_{\Rn}|T^m_{A,k,b}f(x)|^pw(x)dx\lesssim \|b\|^{mp}_{\rm BMO}[w]^{mp}_{A_{\infty}}[w]^{\max\{1,p\}}_{A_{\infty}}\int_{\Rn}(M^{m+1}f(x))^pw(x)dx.
\end{equation*}
\end{theorem}

\subsection{Homogeneous singular integral operators}We will apply our results to derive the weak type endpoint estimatesand the Coifman-Fefferman inequality of homogeneous integral singular operators.

For our purpose, we need to give some definitions. Assume that $\Omega$ is a function of homogeneous of degree zero, $\Omega\in L^{1}({\mathbb S}^{n-1})$ and satisfies the cancellation condition:
\begin{equation}\label{ie8.1}
\int_{\mathbb{S}^{n-1}}\Omega(\theta)d\theta=0,
\end{equation}
where $d\theta$ denotes the surface measure of $\mathbb{S}^{n-1}$. Set $K(x)=\Omega(x)/|x|^n$ and let $T$ be the 
singular integral operator of convolution type with the kernel $K$, that is, $T$ is bounded on $L^2(\Rn)$ and 
\begin{equation}\label{ie8.2}
T_{\Omega}f(x)=p.v.\int_{\Rn}K(x-y)f(y)dy=p.v.\int_{\Rn}\frac{\Omega(y)}{|y|^n}f(x-y)dy.
\end{equation}
For a Young function $A$, define the $L^A$-modulus of continuity of $\Omega$ as
$$\Omega_{A}(t)=\sup\limits_{|y|\leq t}\|\Omega(\cdot+y)-\Omega(y)\|_{A,\mathbb{S}^{n-1}}.$$
Notice that when $A(t)=t^q,\,1<q\leq \infty$, $\Omega_A$ is exactly $L^q$-modulus of continuity in the usual sense.
As was shown in \cite[Theorem 3.2]{iba}, if $\Omega_{A}$ satisfies
\begin{equation*}
\int_{0}^1\Omega_{A}(t)\frac{dt}{t}<\infty,
\end{equation*}
then $K\in H_{A}$. By Theorem \ref{thm5.2} and Theorem \ref{thm5.4}, one can obtain
\begin{theorem}
Suppose that $m,\vec{b},A_i(i=0,\ldots,m)$ satisfy the conditions in Theorem \ref{thm5.2}.
If $T_{\Omega}$ is the operator defined in \eqref{ie8.2} and $\Omega$ satisfies 
\begin{equation*}
\int_{0}^1\Omega_{\bar A_0}(t)\frac{dt}{t}<\infty,
\end{equation*}
then for any weight $w$ and every family of Young functions $\varphi_{\vec{l}}:=\varphi_{(l_1,l_2)},\,0\leq l_2 < l_1\leq m$, it holds that 
\begin{align*}
    w\big(\{x\in \mathbb{R}^n:|T_{\Omega,\vec b}f(x)|>\lambda\}\big)&\lesssim K_{\varphi_{(m,m)}}\int_{\mathbb{R}^n}A_m\Big(\frac{|f(x)|}{\lambda}\Big)M_{\varphi_{(m,m)}}w(x)dx\\
    &+\sum_{l_1=1}^m\sum_{l_2=0}^{l_1-1}K_{\varphi_{(l_1,l_2)}}
    \int_{\mathbb{R}^n}A_{m-l_1}\Big(\frac{|f(x)|}{\lambda}\Big)M_{\Phi_{l_1-l_2}\circ\varphi_{(l_1,l_2)}}w(x)dx,
\end{align*}
where $K_{\varphi_{(l_1,l_2)}}$ is given in \eqref{ie8.3}.
\end{theorem}

\begin{theorem}
Let $0<p<\infty$, $m\in\mathbb{N}$ and $\vec{b}=(b_1,\ldots,b_m)\in \mathrm{BMO}^m$. Assume that $A,B$ are Young functions such that $A\in \mathcal{Y}(p_0,p_1)~(1\leq p_0\leq p_1<\infty)$ and $B^{-1}(t)(C^{-1}(t))^{m}\leq cA^{-1}(t)$ for $t\geq 1$, where $C(t)=e^{t}-1$. If $T$ is the operator defined in \eqref{ie8.2} with $\Omega$ satisfying
$$\int_{0}^1\Omega_{\bar A}(t)\frac{dt}{t}<\infty,$$
then for any $w\in A_{\infty}$,
\begin{equation}\label{ie1.1}
\int_{\Rn}|T_{\Omega,\vec b}(f)(x)|^pw(x)dx\lesssim \prod_{s=1}^m\|b_s\|^p_{\rm BMO}[w]^{mp}_{A_{\infty}}[w]^{\max\{1,p\}}_{A_{\infty}}\int_{\Rn}(M_Bf(x))^pw(x)dx.
\end{equation}
\end{theorem}

\subsection{Fourier multipliers}
Let $h\in L^{\infty}(\Rn)$ and consider the muliplier operator $T$ defined a priori for functions $f$ in the Schwartz class by $\widehat{T_hf}(\xi)=h(\xi)\widehat{f}(\xi)$. Given $1<s\leq 2$ and $l$ a non-negative integer, we say that $h\in M(s,l)$ if for all $|\alpha|\leq l$,
$$\sup\limits_{R>0}R^{|\alpha|}\|\chi_{Q(0,2R)\setminus Q(0,R)}D^{\alpha}h\|_{L^s,Q(0,2R)}<\infty,$$
where $Q(0,R)$ is the cube of side $R$ centered at zero in $\Rn$.

In order to give the quantitative estimates of the above operators, we present the following lemma, which is a slightly weaker version of \cite[Proposition 6.2]{hor+lor2}.

\begin{lemma}\label{lem8.1}
Let $h\in M(s,l)$ with $1<s\leq 2$, $1\leq l\leq n$ and with $l>n/s$, then for all $1<r<(n/l)'$, we have that $K^{N}\in \mathcal{H}_{r}$ uniformly in $N$,
where $K_{N}$ is the certain truncations of the kernel and the detailed definition can be seen in \rm{\cite[p.1419]{hor+lor2}}.
\end{lemma}

Using Theorem \ref{thm5.2}, Corollaries \ref{cor5.6} and \ref{cor5.7}, combined with Lemma \ref{lem8.1} and a standard approximation argument as in \cite{kur}, we have

\begin{theorem}
Let $h\in M(s,l)$ with $1<s\leq 2$, $1\leq l\leq n$ and with $l>n/s$. Assume that $1<r<(n/l)'$ and $\vec{b}=(b_1,\ldots,b_m)\in {\rm BMO}^m$ with $\|b_i\|_{\mathrm{BMO}}=1,i=1,\ldots,m$. Then for any weight $w$ and $0<\varepsilon<1$,
\begin{align*}
    w\big(&\{x\in \mathbb{R}^n:|T_{h,\vec b}f(x)|>\lambda\}\big)\\
    &\lesssim \frac{1}{\varepsilon}\int_{\Rn}\Big(\frac{|f(x)|}{\lambda}\Big)^{r'}\Big(\log \big(e+\frac{|f(x)|}{\lambda}\big)\Big)^{mr'}M_{\psi}w(x)dx,
\end{align*}
where $\psi(t)=t\log^{2mr'+r'}(e+t)\log^{1+\varepsilon}(e+\log(e+t)).$
\end{theorem}

\begin{theorem}
Let $h\in M(s,l)$ with $1<s\leq 2$, $1\leq l\leq n$ and with $l>n/s$. If $1<r<(n/l)'$, $0<p<\infty$, then for any $0<\varepsilon<1$ and $w\in A_{\infty}$, 
\begin{equation*}
\int_{\Rn}|T_{h,\vec b}(f)(x)|^pw(x)dx\lesssim \prod_{s=1}^m\|b_s\|^p_{\rm BMO}[w]^{mp}_{A_{\infty}}[w]^{\max\{1,p\}}_{A_{\infty}}\int_{\Rn}(M_{r'+\varepsilon}f(x))^pw(x)dx,
\end{equation*}
and
\begin{equation*}
\int_{\Rn}|T_h(f)(x)|^pw(x)dx\lesssim [w]^{\max\{1,p\}}_{A_{\infty}}\int_{\Rn}(M_{r'}f(x))^pw(x)dx.
\end{equation*}
Furthermore, if $n/l<r'\leq p< \infty$, then for $w\in A_{p/r'}$, $T$ is bounded on $L^p(w)$ and 
\begin{equation*}
\int_{\Rn}|T_h(f)(x)|^pw(x)dx\lesssim [w]^{p}_{A_{\infty}}[w]^{\frac{p}{p-r'}}_{A_{p/r'}}\int_{\Rn}|f(x)|^pw(x)dx.
\end{equation*}
\end{theorem}

\end{document}